\documentclass[12pt,a4paper]{amsart}
\usepackage[english]{babel}

\usepackage[latin1]{inputenc}
\usepackage{amsmath}
\usepackage{amscd}
\usepackage{amstext}
\usepackage{color}
\usepackage{amsbsy}
\usepackage{amsopn}
\usepackage[nohug,heads=vee]{diagrams}
\diagramstyle[labelstyle=\scriptstyle]
\usepackage{amsthm}
\usepackage{amsxtra}
\usepackage{amsfonts}
\usepackage{amssymb}
\usepackage{mathrsfs}
\usepackage{euscript}
\usepackage{array}
\usepackage{stmaryrd}
\usepackage{verbatim}

\theoremstyle{plain}
      \newtheorem{theorem}{Theorem}
      \newtheorem{lemma}[theorem]{Lemma}
      \newtheorem{corollary}[theorem]{Corollary}
      \newtheorem{proposition}[theorem]{Proposition}
      
      \theoremstyle{definition}
      \newtheorem{definition}[theorem]{Definition}
      
      \theoremstyle{ex}
      \newtheorem{ex}[theorem]{Example}
 \theoremstyle{remark}
      \newtheorem{remark}[theorem]{Remark}
      \theoremstyle{proof}

\newcounter{Step}
\newenvironment{step}[0]{\bigskip\addtocounter{Step}{1}\noindent\textbf{Step \theStep :} }{\
  \begin{flushright} \end{flushright}}

\def\N{\mbox{I\hspace{-.15em}N} }
\def\R{\mbox{I\hspace{-.15em}R} }

\def\C{\hspace{.17em}\mbox{l\hspace{-.47em}C} }

\def\o{\otimes}

\begin{document}
\title[Functional properties of Generalized H\"ormander spaces of distributions I]{Functional properties of Generalized H\"ormander spaces of distributions I :\\
Duality theory, completions and bornologifications}

\author{Yoann Dabrowski}
\address{ Universit\'{e} de Lyon, 
Universit\'{e} Lyon 1,
Institut Camille Jordan UMR 5208,
43 blvd. du 11 novembre 1918,
F-69622 Villeurbanne cedex,
France
}
\email{ dabrowski@math.univ-lyon1.fr}
\date{}

\begin{abstract}
The space $\mathcal{D}'_\Lambda$  of distributions 
having their $C^\infty$ wavefront
set in a cone $\Lambda$ has become important in physics
because of its role in the formulation of quantum field theory
in curved spacetime. It is also a basic object in microlocal analysis, but not well studied from a functional analytic viewpoint. In order to compute its completion in the open cone case, we introduce generalized spaces $\mathcal{D}'_{\gamma,\Lambda}$ where we also control the union of $H^s$-wave front sets in a second cone $\gamma\subset \Lambda$. We can compute bornological and topological duals, completions and bornologifications of natural topologies for spaces in this class. All our topologies are nuclear, ultrabornological when bornological and we can describe when they are quasi-LB. We also give concrete microlocal representations of bounded and equicontinuous sets in those spaces and work with general support conditions including future compact or space compact support conditions on globally hyperbolic manifolds, as motivated by physics applications to be developed in a second paper.
\end{abstract}

\maketitle

\begin{center}\section*{\textsc{Introduction}}
\end{center}

In 1992, Radzikowski~\cite{Radzikowski-92-PhD,Radzikowski} showed 
the wavefront set of distributions 
to be  a key
concept to define quantum fields in curved spacetime.
This idea was fully developed into a renormalized
scalar field theory in curved spacetimes by
Brunetti and Fredenhagen~\cite{Brunetti2}, followed by
Hollands and Wald~\cite{Hollands2} and later extended to more general fields (e.g. recently \cite{RejznerYangMills},\cite{RejznerBF}).

Following those developments, the natural space where quantum field
theory seems to take place is not the space of distributions
$\mathcal{D}'$, but the space $\mathcal{D}'_\Gamma$ of distributions having
their wavefront set in a 
specified closed cone $\Gamma$.
This space and its simplest properties were described
by H{\"o}rmander in 1971~\cite{Hormander-71}. Since recent developments \cite{DutschBF}, most papers in algebraic quantum filed theory have used microcausal functionals where the natural space to control the wave front set is rather the dual $\mathcal{E}'_\Lambda$ of the previous space  with control of the wave front set by an open cone $\Lambda=-\Gamma^c$. Thus, we started recently \cite{BrouderDabrowski} the investigation of the functional analytic properties of these spaces. Among other properties, we showed that these spaces are both
nuclear and normal spaces of distributions but
$\mathcal{D}'_\Gamma$ and $\mathcal{E}'_\Lambda$ still have very contrasted
properties; (i) $\mathcal{D}'_\Gamma$ is semi-reflexive and complete while
$\mathcal{E}'_\Lambda$ is not even sequentially complete;
(ii) $\mathcal{E}'_\Lambda$ is barrelled,
and ultrabornological, while
$\mathcal{D}'_\Gamma$ is neither barrelled nor bornological.

From our viewpoint the lack of completeness of $\mathcal{E}'_\Lambda$ is the most annoying for applications, since it was used in the physics literature before we pointed out a counter-example. Beyond an easy abstract characterization we mentioned in the conclusion of \cite{BrouderDabrowski} that the completion has to be the bornological dual of $\mathcal{D}'_\Gamma$ (the strong dual of its bornologification), it is hard to get general results about those spaces without a concrete description.

The study of this series of papers started with finding such a natural characterization in terms of what we will call ``dual wave front set"  which is nothing but the union of the usual $H^s$-wave front sets (see e.g. \cite[p8]{Delort}) of  a distribution : $DWF(u)=\bigcup_{s >0}WF_s(u)$ so that $WF(u)=\overline{DWF(u)}$ is recovered as its closure. These are cones in what we write $\dot{T}^*U$, namely the cotangent bundle with zero section removed.  With this definition the completion $\widehat{\mathcal{E}'_\Lambda}$ is nothing but : $$\widehat{\mathcal{E}'_\Lambda}=\{u\in \mathcal{E}', DWF(u)\subset \Lambda\}.$$

Moreover, from a functional analytic point of view, studying spaces of multilinear operators or forms on $\mathcal{D}'_\Gamma,\mathcal{E}'_\Lambda$ is quite natural and enables to simplify technical wave front set computations in new contexts as the one of field dependent propagator in \cite{RibeiroBF,RejznerBF} (see part 2 of the second paper of this series \cite{Dab14b}). Motivated by this study, we will see that controlling the wave front set with cones that are neither open nor closed seems natural. Moreover  several geometrically natural conditions on wave front sets, appeared recently in \cite{Dang} (the notion of polarized wave front set is a conic condition) or \cite{ColliniH}, also motivates the study of slightly generalized H\"ormander spaces of distribution with wave front set controlled by any natural cone. Thus for $\gamma\subset \Lambda\subset \overline{\gamma}$ cones ($\overline{\gamma}$ will always be the closure of the cone $\gamma$ in $\dot{T}^*U$), we define :
$$\mathcal{E}'_{\gamma,\Lambda}(U)=\{u\in \mathcal{E}'(U) : DWF(u)\subset \gamma, WF(u)\subset \Lambda\},$$
$$\mathcal{D}'_{\gamma,\Lambda}(U)=\{u\in \mathcal{D}'(U) : DWF(u)\subset \gamma, WF(u)\subset \Lambda\}.$$
The  main result of the first paper of this series then states that we can understand the main properties of the most important among those spaces $\mathcal{E}'_{\gamma,\gamma}(U),\mathcal{E}'_{\gamma,\overline{\gamma}}(U)$ and their $\mathcal{D}'$ variants (with quite more general support conditions). This gives the following result (see proposition \ref{FAGeneral2} for a proof, generalized to bundles and other support conditions, and \ref{FAGeneral3} for complements on the general case $\mathcal{E}'_{\gamma,\Lambda}(U)$, 
see also section 1 in the second paper of this series \cite{Dab14b} for the stated approximation properties)
. Recall closures and complements are taken in $\dot{T}^*U$ and we denote $-\gamma=\{(x,-\xi): (x,\xi)\in\gamma\}.$

\medskip

 \begin{theorem}\label{FAMain}Let $\gamma\subset\dot{T}^*U$ 
any cone   and let $\lambda=-\gamma^c.$
 There is a duality pairing between  $\mathcal{D}'_{\gamma,{\gamma}}(U)$ and $\mathcal{E}'_{\lambda,{\lambda}}(U)$ (which is continuous for a natural (non countable) inductive limit topology $\mathcal{I}_{iii}$, see section 1.3). Their Mackey topology, let us call it $\mathcal{I}$ (which is $\mathcal{I}_{iii}$) is nuclear.
  If we call $i$  the canonical injection from a space to its completion and $b$ the canonical continuous map with bounded inverse between a bornologification and the original space (which is identity in the open cone case), we have the two following commuting diagrams :

$$
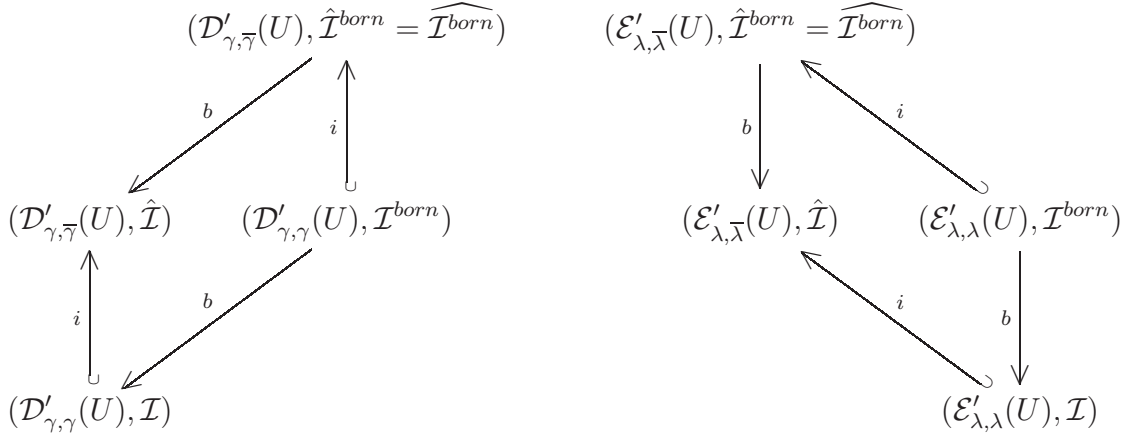
\begin{diagram}[inline]
&&(\mathcal{D}'_{\gamma,\overline{\gamma}}(U),\hat{\mathcal{I}}^{born}=\widehat{\mathcal{I}^{born}})
\\
&\ldTo^b&\uInto^i 
\\
(\mathcal{D}'_{\gamma,\overline{\gamma}}(U),\hat{\mathcal{I}})& &(\mathcal{D}'_{\gamma,{\gamma}}(U),\mathcal{I}^{born})
\\
\uInto^i&\ldTo^b& 
\\
(\mathcal{D}'_{\gamma,{\gamma}}(U),\mathcal{I})& &
\end{diagram}\quad \quad \quad \begin{diagram}[inline]
(\mathcal{E}'_{\lambda,\overline{\lambda}}(U),\hat{\mathcal{I}}^{born}=\widehat{\mathcal{I}^{born}})&&
\\
\dTo^b& \luInto^i&
\\
(\mathcal{E}'_{\lambda,\overline{\lambda}}(U),\hat{\mathcal{I}})& &(\mathcal{E}'_{\lambda,{\lambda}}(U),\mathcal{I}^{born})
\\
&\luInto^i&\dTo^b
\\& & 
(\mathcal{E}'_{\lambda,{\lambda}}(U),\mathcal{I})
\end{diagram}$$

Spaces symmetric with respect to the middle vertical line are topological duals of one another equipped with their Mackey topology. All the spaces involved are nuclear, strictly normal spaces of distributions with the sequential approximation property, and in each of them, compact sets are metrizable.
Any bornological space among them is ultrabornological and the two top spaces are thus complete nuclear ultrabornological, thus Montel spaces.

 When $\lambda,\gamma$ are $\mathbf{\Delta_2^0}$-cones (i.e. both $F_\sigma$,$G_\delta$) all the spaces involved are moreover quasi-LB spaces (for which de Wilde's closed graph theorem applies \cite{ValdiviaQuasiLB}) of class $\mathfrak{G}$ \cite{CascalesOrihuela}.  
\end{theorem}

Even in the closed and open case, it was not clear (from general reasons) that the bornologification of $\mathcal{D}'_\Gamma$ was nuclear or that the completion $\widehat{\mathcal{E}'_\Lambda}$ was ultrabornological. Those spaces fully restore the missing symmetry in \cite{BrouderDabrowski} (for which we refer for a reminder of most of the definitions) since they are both nuclear complete ultrabornological normal spaces of distributions. We give explicit descriptions of bounded and equicontinuous sets in all the above spaces during our proofs (see any computation of Mackey=Arens duals for equicontinuous sets that are also subsets of absolutely convex compact sets in any case of the previous theorem and see the regular inductive limit statement in proposition \ref{FAGeneral2} and section 1 in \cite{Dab14b}
in the non-complete cases for bounded sets). As in \cite{BrouderDabrowski} we also have the general convenient results that enable to prove sequential continuity and boundedness in these spaces weakly by duality.
 
\begin{corollary}With the same notation.
\begin{itemize}\item If $u_i$ is a sequence of elements of $D=(\mathcal{D}'_{\gamma,\overline{\gamma}}(U),\hat{\mathcal{I}}^{born})$ (resp. $E=(\mathcal{E}'_{\lambda,\overline{\lambda}}(U),\hat{\mathcal{I}}^{born})$) such that 
$\langle u_i,v\rangle$ converges to some number
$\lambda(v)$ in $\C$ for all $v\in E$ (resp $v\in D$), then 
$u_i$ converges to $\lambda$ in $D$ (resp $E$).
\item If $(u_\epsilon)_{0<\epsilon < \alpha}$ is a family of elements of $D$ (resp $E$). such that 
$\langle u_\epsilon,v\rangle$ converges to some number
$\lambda(v)$ in $\C$ as $\epsilon\to 0$
for all $v\in E$ (resp $v\in D$), then 
$u_\epsilon\to\lambda$ in $D$ (resp $E$)
as $\epsilon\to0$.
\item In $E,D,$ weakly bounded sets are strongly bounded.
\end{itemize}
\end{corollary}

After this general study of generalized H\"ormander spaces of distributions which ends the first paper of this series, our second goal will be to provide tools to study spaces of functionals with wave front set conditions on their derivatives, as microcausal functionals appearing in algebraic quantum field theory \cite{DutschBF}. For this, the second paper of this series will gather  miscellaneous results on tensor products of our spaces and spaces of multilinear maps, and then study a general class of spaces of functionals. The third paper of this series will investigate more systematically those tensor products following the general advice of Grothendieck and mostly because general results don't give a full understanding of those tensor products. As for the completion above, there is a need for concrete (microlocal) representation to obtain a fully satisfactory functional analytic understanding. 


Let us now describe in more detail the content of this paper, as a guiding summary of our tools for the reader.

Part 1 deals with preliminary material, gives the definition of spaces and topologies and explain in our slightly more general context the results of \cite{BrouderDabrowski}. The results here are probably well known to experts (or may have been known to experts in the 60's) and we provide proofs only for the reader's convenience. 

First, in our applications in the second paper of this series, using support conditions more general than compact and full support will be crucial to make spaces of distributions stable by advanced/retarded propagators. As in \cite{Sanders} we thus need to treat future/past compact support on globally hyperbolic spacetimes (cf example \ref{enlargeable3}). Rather than dealing with them in an ad-hoc way, we give in section 1 a general family of support conditions having nice stability properties by product, duality, intersection and covering these examples. This also gives in a basic context of inductive limit on support the basic idea of producing alternative inductive/projective limit representation of topologies that will be one of our elementary tools. All our spaces of distributions with controlled support will be ultrabornological nuclear complete and will be the basis to control support in our generalized H\"ormander spaces (even when those nice functional analytic properties may be lost). 
A notion of openly generated support condition enables to get continuous injections between spaces of distributions and functions (lemma \ref{opengenlemma}) and then a notion of enlargeability enables to get nice density results. The dense continuous injection $\mathcal{E}(U,(\mathscr{O}_{\mathcal{C}})^{oo};E)\to \mathcal{D}'(U,\mathcal{C};E)$ will serve as a framework for a notion of ``normal" system controlling stability of support by duality. 

Section 2 gives some detail on the definition of dual wave front set and section 3 introduces the topologies we need. At the end, lots of them will be identified, and, as we said, this will provide a crucial tool to get functional analytic properties better preserved by inductive limit or projective limit. All the definitions should be thought of as various ideas of controlling from above (projective limit) or bellow (inductive limit)  the spaces we are interested in, starting from more basic spaces for which we already have a topology, where cones will be mostly open or closed, as studied in \cite{BrouderDabrowski}. Section 4 studies continuous injections and density in this context and generalize the topological duality results of \cite{BrouderDabrowski}, cf theorem \ref{DualityBD}. We are still here in the closed and open cone cases and support will be treated easily by our normal system technique once considered the duality of $\mathcal{D}'_\Lambda,\mathcal{E}'_\Gamma$ not obtained in   \cite{BrouderDabrowski}. Note that already here, non-countable inductive limits appear and the game of playing with equivalent projective limit is crucial to keep for instance nuclearity.

Part 2 contains the main new material to prove Theorem \ref{FAMain}. Section 5 gives topological and bornological duality results. In lemma \ref{bornImbedding}, we see that our notion of dual wave front set is the right one imposed to bounded linear mappings. This surprisingly reduces to the minimal case of a cone supported on one point and with only an half line in the cotangent bundle above that point. Then  Theorem \ref{duals} finishes the computation of duals by giving a bounded duality pairing in the missing cases.
Section 6 then explains the consequences in the open and closed case concerning bornologifications and completions. In lemma \ref{bornologyIdentities}, \ref{bornologyIdentitiesEqui} we use bornology techniques and identify topologies in computing equicontinuous sets in duals. The end of the subsection is technical, it uses crucially the notion of quasi-(LB) spaces of \cite{ValdiviaQuasiLB} and the dually related notion of class $\mathfrak{G}$ to provide metrizability of compact sets and then to identify absolutely convex compact sets in non-complete cases. As we will see, it is much more crucial than that and most of our non-trivial results are based on a variant of this argument. The argument in proposition \ref{bornologyIdentitiesArens} thus uses ideas from our proof of the non-completeness of $\mathcal{E}'_\Lambda$ in \cite{BrouderDabrowski} to compute absolutely convex compact sets and thus the remaining Mackey-Arens duals not yet known in the open/closed case. 

Section 6 almost completes the proof of theorem \ref{FAMain} in the case $\gamma$ closed or open (except for nuclearity of bornologifications). As we will see in the section 3 of definition of topologies, this case is the key to all other cases. In section 7, proposition \ref{FAGeneral} builds on the easy nuclearity computations that come easily from our definitions of topologies and explains consequences. The end of the section can then complete the proof of theorem \ref{FAMain}. It uses again the computation of absolutely convex compact sets to compute missing Mackey/Arens duals and it gets nuclearity of bornologifications by studying the dual notion of convex bornological space. The key here is to identify all (completant) bounded disks with an argument reminiscent from our computation of  absolutely convex compact sets again. Then the proof of bornological nuclearity will be standard but, by hand, no known stability result seem to be able to deduce it. 

Section 8 applies our techniques to generalize the pull-back theorem in our context, as we needed from the very beginning to deal with the manifold case. Of course, the whole paper can be developed in the open case in $\R^d$ without using this section, and then reread using this result in the general case.


\subsection*{Acknowledgments} The author is grateful to Katarzyna Rejzner and Christian Brouder for helpful discussions on the physical motivations. He also thanks Christian Brouder for many comments on previous versions of this paper that helped improving its exposition. 

Finally, the author acknowledges the support and hospitality of the Erwin Schr\"odinger Institute during the workshop ``Algebraic Quantum Field Theory: Its Status and Its Future" in May 2014. 

\part{Preliminaries on H\"ormander spaces of distributions and their generalization}
Let $U\subset M$ be an open set in a smooth connected manifold (implicitly assumed orientable 
 $\sigma$-compact without boundary) of dimension $d$.
We assume given on $M$ a complete Riemannian metric $D$ giving the topology (so that, by Hopf-Rinow theorem, closed balls for $D$ are compact.)

 Let $E\to U$ a smooth real vector bundle (with finite dimensional fiber of real dimension $e$).

We will use a fixed smooth partition of unity $f_i$, indexed by $I$, subordinated to a covering $U_i$, with $\overline{U_i}\subset U$ compact and $(\overline{U_i})_{i\in I}$ locally finite 
 $U_i$ smoothly isomorphic via a chart $\varphi_i:U_i\to \R^d$ to an open set in $\R^d$, extending homeomorphically to $\overline{U_i}$ and  trivializing  $E\to U$ (This will be fixed except when we will prove our spaces are independent of this choice). We don't write the fixed vector bundle trivialization above $\varphi_i$ and any expression $u\circ \varphi_i^{-1}=(\varphi_i^{-1})^*u$ always depend on this fixed trivialization.

 We consider $\mathcal{D}'(U;E)$ the space of distributional sections of $E$ with support in $U$ and $\mathcal{E}'(U;E)$ the space of distributional sections of $E$ with compact support in $U$. Note that $\mathcal{D}'(U_i,E)\simeq \mathcal{D}'(\varphi_i(U_i),\R^e)\simeq (\mathcal{D}'(\varphi_i(U_i)))^e $ (see e.g. \cite[(3.16) p 234]{GrosserKOS}), via the map written above $u\mapsto u\circ \varphi_i^{-1}$.
 
Likewise we call  $\mathcal{E}(U;E)$ the space of smooth sections of $E$ on $U$ and $\mathcal{D}(U;E)$ the space of smooth sections of $E$ with compact support in $U$. We write as usual $E'$ for the dual bundle, and $E^*=E'\otimes \R^*$ its version twisted by the density bundle $\R^*$ (cf e.g. \cite[chapter 3]{GrosserKOS}). We of course don't write $E$ in any notation when $E$ is the trivial line bundle over $\R$. Note that, since we don't make explicit our trivialization of vector bundles, we make the choice for those of $E,E^*$  so that for $g\in\mathcal{D}(U_i),v\in \mathcal{D}(U)$: 
$$\langle f_iu, gv\rangle=\langle (f_iu)\circ \varphi_i^{-1},(gv)\circ \varphi_i^{-1}\rangle.$$

This reduces duality pairings to those on $\R^n$ and the emphasis on the difficult analytic part of pullback rather than on the multiplication by smooth map part involved in change of bundle trivialization.
 
\section{Support properties}\label{support}
We first have to deal with variants of the above spaces of distributions with various support properties to model typical examples in applications, for instance, space compact support or future compact support on globally hyperbolic space-times. We start with a general support condition and add technical assumptions, namely what we call below openly generated and enlargeable, in order to get nice dense inclusions. We don't claim to consider the most general support condition useful in applications, especially, all our spaces will contain compactly supported smooth functions.

Generalizing slightly the classical notion of a normal space of distributions \cite{Schwarz}, we are looking for spaces of smooth functions $D,E$ with $D\subset E'$ such that we will be able to control the support of a space of distribution $F$ by  two dense continuous injections $D\subset F\subset E',$ in such a way that the dual will satisfy $E\subset F'\subset D'.$ In this subsection, we are thus looking for support conditions on smooth functions and distributions enabling to get our spaces  $D,E'$ having nice functional analytic properties. This will be mostly an exercise in standard distribution theory (for which we did not find any reference in the literature though), motivated by the need of using specific support conditions in applications.

Let $\mathcal{K}$ be the family of compact sets in $U$, $\mathcal{B}$ the family of sets with compact closure in $U$,  $\mathcal{T}$ the family of all open sets in $U$ and $\mathcal{F}$ the family of all closed sets in $U.$

\subsection{Polar support condition}

For duality purposes, we need a notion of polar of a support condition, one more notion of polar used in mathematics after orthogonality, commutants, etc. (cf. \cite[p 141]{Girard}). In our case, the idea is to get a condition of support $\mathcal{C}^o$ that can be paired with $\mathcal{C}$ to get a compact support condition that will allow a duality pairing, nothing more. In the terminology of \cite{Girard}, we use the set of compact sets $\mathcal{K}$ as the pole for our notion of polar.

For $\mathcal{O}$ a family of open sets in $U$, $\mathcal{C}$ a family of closed sets in $U$, we define $$\mathcal{O}^o=\{C\in \mathcal{F}, \forall W\in \mathcal{O}, \overline{C\cap {W}}\in \mathcal{K}\},\ \ \mathcal{C}^o=\{V\in \mathcal{T}, \forall C\in \mathcal{C}, \overline{V\cap C}\in \mathcal{K}\} .$$


As for any polarity relation, if $\mathcal{O}\subset \mathcal{P}$, then $\mathcal{P}^o\subset \mathcal{O}^o$ and similarly for families of closed sets, and $\mathcal{O}\subset (\mathcal{O}^o)^o=:\mathcal{O}^{oo}, \mathcal{C}\subset (\mathcal{C}^o)^o=:\mathcal{C}^{oo}$. We will be interested in what we call \textit{polar families} for which there is equality $\mathcal{O}=\mathcal{O}^{oo}, \mathcal{C}=\mathcal{C}^{oo}$. Moreover $\mathcal{C}^o$ and especially any polar family, is stable by intersection with a given open set, and stable by finite union, i.e. this is an ideal of open sets in an obvious sense which always contains $\mathcal{B}\cap \mathcal{T}$. 

We have for instance $(\mathcal{K})^o=\mathcal{T}$,$(\mathcal{T}^o)=\mathcal{K},$ $((\mathcal{B}\cap \mathcal{T})^o)=\mathcal{F}$,$(\mathcal{F}^o)=\mathcal{B}\cap \mathcal{T},$ which are all polar families.


\begin{ex}\label{countgenclosed}
If $\mathscr{F}=\{F_i\}_{i\in \N}$ is a countable family of closed sets we say $\mathcal{C}$ is countably generated  (by $\mathscr{F}$) if it is the polar family $\mathcal{C}=\mathcal{C}_{\mathscr{F}}:=\{F_i, i\in \N\}^{oo},$ for some such ${\mathscr{F}}.$ 
Without loss of generality one may assume $F_i$ increasing. Then, if $(K_j)_{j\in \N}$ is an increasing sequence of compacts in $U$, with their interiors exhausting $U$, $$\mathcal{C}_{\mathscr{F}}=\{F\in\mathcal{F}, \exists i F\subset F_i\cup K_i\}.$$

Indeed, for the only non obvious inclusion, assume that $F$ contains a sequence $x_i\not\in F_i\cup K_i$. From the definition of $K_i$'s any such sequence has no converging sub-sequence in $U$. One can construct a sub-sequence $y_i$ with the same property and moreover such that we have built a fixed open $V_i$ in $K_{k_i}\cap F_i^c\subset U$ containing $y_i$ and we impose at the next step $y_{i+1}\not\in \cup_{j\leq i}V_j.$ Then define the open set $W=\cup_{j\in \N} V_j$. Since $V_j\cap F_i\subset V_j\cap F_j=\emptyset$ if $j\geq i$, $W\cap F_i=\cup_{j<i} V_j\cap F_i\subset \cup_{j<i}K_{k_j}$ i.e. $\overline{W\cap F_i}$ is compact and thus $W\in \{F_i, i\in \N\}^{o}.$
But now, $y_j\in \overline{W\cap F}$ has no converging sub-sequence, thus  $\overline{W\cap F}$ is not compact, i.e. $F\not\in \mathcal{C}_{\mathscr{F}},$ as expected. 

Moreover it is obvious by definition that  $\mathcal{C}_{\mathscr{F}}^o=\{O\in\mathcal{T}, \forall i, \exists k_i, O\subset F_i^c\cup K_{k_i}\}.$

\end{ex}

\begin{ex}\label{countgenopen}
If $\mathscr{O}=\{O_i\}_{i\in \N}$ is a countable family of open sets we say $\mathcal{O}$ is countably generated (by $\mathscr{0}$) if it is the polar family $\mathcal{O}=\mathcal{O}_{\mathscr{O}}:=\{O_i, i\in \N\}^{oo},$ for some such ${\mathscr{O}}.$ 
Without loss of generality one may assume $O_i$ increasing. Then, if $(K_j)_{j\in \N}$ is an increasing sequence of compacts in $U$, with their interiors exhausting $U$, then, similarly to the previous example, we have  : $\mathcal{O}_{\mathscr{O}}=\{O\in\mathcal{T}, \exists i O\subset O_i\cup K_i\},\ \ \mathcal{O}_{\mathscr{O}}^o=\{F\in\mathcal{F}, \forall i, \exists k_i, F\subset O_i^c\cup K_{k_i}\}.$


\end{ex}

\begin{ex}\label{ProductCase}
If $U=U_1\times U_2$, $\mathcal{C}_i$ polar family of closed sets on $U_i$. Let us show that $(\mathcal{C}_1\times \mathcal{C}_2)^{oo}=\{F\subset U, F\subset F_1\times F_2, F_i\in  \mathcal{C}_i\}.$

Indeed, since polar families are stable by subsets, $\supset$ is obvious. Conversely, let us show that if $F\in (\mathcal{C}_1\times \mathcal{C}_2)^{oo}$ its projection $\overline{p_i(F)}\in \mathcal{C}_i$. Otherwise, take $O\in \mathcal{C}_i^o$ (say $i=1$) such that $\overline{\overline{p_1(F)}\cap O}=\overline{p_1(F)\cap O}$ is not compact, and there exists $x_n\in p_1(F)\cap O$ such that $\{x_n\}_{n\in \N}$ is not precompact. Take $(x_n,y_n)\in F$ and $\epsilon_n\to 0$ such that $B(x_n,\epsilon_n)\subset O,B(y_n,\epsilon_n)\subset U_2$ (with compact closure) and $B(x_n,\epsilon_n)$ disjoint for different $n$. Let us show $A=\cup_{n\in \N} B(x_n,\epsilon_n)\times B(y_n,\epsilon_n) \in (\mathcal{C}_1\times \mathcal{C}_2)^{o} $. Indeed, for $F_i\in \mathcal{C}_i$, if $(x,y)\in (F_1\times F_2)\cap A$, $x\in B(x_n,\epsilon_n)\cap F_1\subset \overline{O\cap F_1}$ which is compact, and has non empty intersection with only finitely many $B(x_n,\epsilon_n)$ (since $\{x_n, n\in \N\}$ not compact, thus so is any $\{x_n', n\in \N\}$ with $x_n'\in B(x_n,\epsilon_n)$ since $\epsilon_n\to 0$), thus $\overline{(F_1\times F_2)\cap A}\subset \overline{\cup_{n\leq N}B(x_n,\epsilon_n)\times B(y_n,\epsilon_n)}$ which is compact, as wanted. Finally $A \in (\mathcal{C}_1\times \mathcal{C}_2)^{o} $ is in contradiction with $A\cap F$ contains the non-precompact sequence $\{(x_n,y_n), n\in \N\}.$

Moreover, we obviously have  $(\mathcal{C}_1^o\times \mathcal{C}_2^o)^{oo}\subset (\mathcal{C}_1\times \mathcal{C}_2)^{o}$.
\end{ex}

\begin{ex}\label{IntersectionCase}
Let  $\mathcal{C}_i$  be polar families of closed sets on $U$.

 Let us show that $(\mathcal{C}_1\cap \mathcal{C}_2)=\{F\in\mathcal{T}, \exists F_i\in\mathcal{C}_i^o, F\subset F_1\cup F_2 \}^o$ so that it is also polar. For convenience we write $\mathcal{C}_1^o\vee \mathcal{C}_2^o=(\{F\in\mathcal{T}, \exists F_i\in\mathcal{C}_i^o, F\subset F_1\cup F_2, \})^{oo}$
 
 Indeed,  obviously $ (\mathcal{C}_1^o\vee \mathcal{C}_2^o)^o\subset \mathcal{C}_1^{oo}\cap \mathcal{C}_2^{oo}$,  since $\mathcal{C}_i^{o}\subset \mathcal{C}_1^o\vee \mathcal{C}_2^o.$ Conversely, if $C\in \mathcal{C}_1^{oo}\cap \mathcal{C}_2^{oo},$ and $F\subset F_1\cup F_2, F_i\in \mathcal{C}_i^o$ $\overline{F\cap C}\subset \overline{F_1\cap C}\cup \overline{F_2\cap C}$ which is compact as finite union of compacts.

Moreover, if $\mathcal{C}_i^o=\{F_{i,n}, n\in \N\}^{oo}$ are countably generated, it is easy to see that $\mathcal{C}_1^o\vee \mathcal{C}_2^o=\{F\in\mathcal{T}, \exists F_i\in\mathcal{C}_i^o, F\subset F_1\cup F_2 \}$ (which is already polar) since it is countably generated by $F_{1,n}\cup F_{2,m}, (n,m)\in \N^2$ and the stated equality follows from our description of this family. 

We also of course have  variants with open/closed sets families exchanged.

\end{ex}

\subsection{Natural topologies on distributions and smooth functions with support condition}
For $\mathcal{O}$ 
a 
family of open sets in $U$, $\mathcal{C}$ a family of closed sets, { stable by subsets}, we call $$\mathcal{E}(U,\mathcal{O};E^*)=\{ \varphi\in \mathcal{E}(U;E^*)\ : (Ker^{\infty}(\varphi))^c\in\mathcal{O} 
\},$$ $$\mathcal{D}'(U,\mathcal{C};E)=\{ u\in \mathcal{D}'(U;E)\ :  \text{supp}(u)\in\mathcal{C} \}.$$ 
Above, we wrote $(Ker^{\infty}(\varphi))^c:=\cup_{i\in I}\cup_{\alpha}\varphi_i^{-1}[(\partial_{\alpha}\varphi\circ\varphi_i)^{-1}(\R^e-\{(0,...,0)\})]$ where the union over $\alpha$ runs over all multi-indices for derivatives in $\R^d$, so that $\partial_{\alpha}\varphi\circ\varphi_i:\varphi_i(U_i)\to \R^e.$
Thus $(Ker^{\infty}(\varphi))\supset (\text{supp}(\varphi))^c$ is the set of points where in our fixed chart system all the derivatives of $\varphi$ vanish, set which is of course independent of the chosen chart system.

Note also that $\overline{(Ker^{\infty}(\varphi))^c}=\text{supp}(\varphi),$ since $Int( Ker^{\infty}(\varphi))=(\text{supp}(\varphi))^c$ since on a open set vanishing of all derivatives or the function alone coincide.

Thus  we have $\mathcal{D}'(U,\mathcal{K};E)=\mathcal{E}'(U;E),\mathcal{E}(U,\mathcal{B}\cap \mathcal{T};E^*)=\mathcal{D}(U;E^*).$

We put on $\mathcal{E}(U,\mathcal{O};E^*)$ the inductive limit topology (with $\mathcal{O} $ ordered by inclusion, {and where its stability by subsets is used to identify the two sets}), $$\mathcal{E}_i(U,\mathcal{O};E^*)=\text{ind} \lim_{O\in \mathcal{O}} \mathcal{E}(U,\{P\in \mathcal{T}, P\subset O\};E^*),$$ this last space being given the topology of convergence of all derivatives on compact sets, i.e. the topology induced by $\mathcal{E}(U;E^*)$ on this closed subspace defined by the relation $O^c\subset Ker^\infty(\varphi)$. This is a separated  strict inductive limit continuously embedded in $\mathcal{E}(U;E^*)$ of Fr\'echet spaces, especially, it is always ultrabornological \cite[Prop 6.2.1]{PerrezCarreras}. If $\mathcal{O}$ is countably generated as in example \ref{countgenopen}, this inductive limit is equivalent to a strict inductive sequence which is complete, regular and nuclear (as a locally convex space, not as an inductive limit, we never use the notion of nuclear inductive limit, see \cite{Bierstedt}) by general results. Since this sequence comes from a cofinal sequence, the original inductive limit is then also regular.

When $\mathcal{O}=\mathcal{C}^o$ is polar, there is an a priori different topology which is interesting. The idea is to approximate the support condition from above to get a projective limit using the polarity. This idea will be much more crucial for spaces of distributions with a wave front set condition. Here it will enable to get an alternative projective limit description.

 Since $\mathcal{C}^o\subset \{F\}^o,$ for $F\in \mathcal{C}$, there is a continuous map $$\mathcal{E}_i(U,\mathcal{O};E^*)\to \text{proj} \lim_{F\in \mathcal{C}}\mathcal{E}_i(U,\{F\}^o;E^*)=:\mathcal{E}_p(U,\mathcal{C};E^*)$$ where $\mathcal{E}_i(U,\{F\}^o;E^*)$ is given the previous complete nuclear topology (since by example \ref{countgenclosed}, it is countably generated). This projective limit is again by definition $\mathcal{E}(U,\mathcal{O};E^*)$ but we write it with a index $p$ with this topology at this point. By the general remark above and standard facts about separated projective limits, $\mathcal{E}_p(U,\mathcal{O};E^*)$ is complete and nuclear.

Likewise, we put on $\mathcal{D}'(U,\mathcal{C};E)$ the inductive limit topology (with $\mathcal{C}$ ordered by inclusion), $$\mathcal{D}_i'(U,\mathcal{C};E)=\text{ind}\lim_{C\in \mathcal{C}} \mathcal{D}'(U,\{F\in \mathcal{F}, F\subset C\};E),$$ this last space being given the topology induced by the strong topology of $\mathcal{D}'(U;E)$ (from its duality with $\mathcal{D}(U;E^*)$) on this closed subspace. This is a separated strict inductive limit continuously embedded in $\mathcal{D}'(U;E)$. If $\mathcal{C}$ is countably generated as in example \ref{countgenclosed}, this inductive limit is equivalent to a strict inductive sequence which is complete, regular and nuclear. Similarly, if $\mathcal{C}=\mathcal{O}^o$ is polar, we have a different candidate projective  topology :
$$\mathcal{D}_i'(U,\mathcal{C};E)\to \text{proj} \lim_{O\in \mathcal{O}}\mathcal{D}_i'(U,\{O\}^o;E)=:\mathcal{D}_p'(U,\mathcal{O};E)\to \mathcal{D}_p'(U,\mathcal{O}^{oo};E)$$ with  as before $\mathcal{D}_p'(U,\mathcal{O};E) $ nuclear complete. 

The following fact giving computation of duals is mainly a reformulation of well-known results :

\begin{lemma}\label{DualitySupport}
Let $\mathcal{O},\mathcal{C}$ families of open and closed sets respectively, {stable by subsets and containing $\mathcal{B}\cap\mathcal{T}$ and $\mathcal{K}$ respectively}. The dual of  
$(\mathcal{E}_i(U,\mathcal{O};E^*))'=\mathcal{D}'(U,\mathcal{O}^o;E)$ and $(\mathcal{D}_i'(U,\mathcal{C};E))'=\mathcal{E}(U,\mathcal{C}^o;E^*)$. Moreover, 
we also have $(\mathcal{E}_p(U,\mathcal{C}^{oo};E^*))'=\mathcal{D}'(U,\mathcal{C}^{oo};E)$ and $(\mathcal{D}_p'(U,\mathcal{O}^{oo};E))'=\mathcal{E}(U,\mathcal{O}^{oo};E^*)$. Finally strongly bounded sets in $(\mathcal{E}_i(U,\mathcal{O};E^*))_b'$ (resp. $(\mathcal{D}_i'(U,\mathcal{C};E))'_b$)  are bounded in some term of the defining inductive limit $\mathcal{D}_{i}'(U;\mathcal{O}^o;E)$ (resp. $\mathcal{E}_i(U,\mathcal{C}^o;E^*)$ , i.e. they are bounded for the inductive limit bornology).


\end{lemma}

\begin{proof}
Let $F\in \mathcal{O}^o$ (resp $F\in \mathcal{C}$), $W\in \mathcal{O}$ (resp $W\in \mathcal{C}^o$), $u\in \mathcal{D}'(U,\mathcal{O}^o;E)$ (resp. $u\in \mathcal{D}'(U,\mathcal{C};E)$), $\varphi\in \mathcal{E}(U,\mathcal{O};E^*)$  (resp. $\varphi\in \mathcal{E}(U,\mathcal{C}^o;E^*)$) let $f\in \mathcal{D}(U)$  equal to $1$ on the compact $\overline{F\cap  W}$ and assume $\text{supp}(u)\subset F$ and  $Ker^\infty(\varphi)\subset {W}$.
Let us understand the support of $u\varphi\in (\mathcal{D}(U))',$ defined by $\langle u\varphi, f\rangle= \langle u, \varphi f\rangle$ since $\varphi f\in \mathcal{D}(U,E^*).$

 Let $g\in \mathcal{D}(U)$ of support in 
$(\overline{F\cap  W})^c,$ then $g\varphi$ and all its derivatives are zero on (a neighborhood of) $\overline{F\cap  W}$ and on $W^c$, the first because of the support of $g$ the second because of the assumption on $\varphi$ and then by Leibniz formula. Thus $g\varphi$ and all its derivatives are zero on $\text{supp}(u)\subset F\subset W^c\cup (F\cap W)$ thus by (an obvious bundle variant, obtained in coordinates, of) \cite[Th 2.3.3]{Hormander}   $\langle u,g\varphi \rangle=0$ (replacing first $u$ by a distribution of compact support by multiplication by a smooth function valued 1 on a neighborhood of $\text{supp}(g)$.) We deduce $\text{supp}(u\varphi)\subset \overline{F\cap  W}.$

 Thus $\langle u,\varphi\rangle:=\langle u\varphi,f \rangle=\langle u,\varphi f \rangle$ is independent of $f$ and obviously a continuous duality pairing, first for the topologies  with support properties fixed  by $F,W$ and then for the inductive limit topologies  $(\mathcal{E}_i(U,\mathcal{O};E^*),\mathcal{D}_i'(U,\mathcal{O}^o;E))$,$(\mathcal{E}_i(U,\mathcal{C}^o;E^*),\mathcal{D}_i'(U,\mathcal{C};E))$. Especially, this shows $\mathcal{D}'(U,\mathcal{O}^o;E)\subset (\mathcal{E}_i(U,\mathcal{O};E^*))', \mathcal{E}(U,\mathcal{C}^o;E^*)\subset (\mathcal{D}_i'(U,\mathcal{C};E))'$ (the inclusion as usual by a dense inclusion of $\mathcal{D}(U;E^*),\mathcal{E}'(U;E)$ in the spaces of which we consider duals since we assumed $\mathcal{B}\cap\mathcal{T}\subset \mathcal{O}$ and $\mathcal{K}\subset \mathcal{C}$, otherwise we still have built a not necessarily injective map).

Moreover, for each $\varphi\in \mathcal{E}(U,\mathcal{O};E^*)$ as above, the only requirement for the argument above is $F\in\{W\}^o$, i.e. we have the continuity on $\mathcal{D}_i'(U,\{W\}^o;E)$ and by composition on the projective limit on $\mathcal{D}_p'(U,\mathcal{O};E).$ Thus,
 we have $\mathcal{E}(U,\mathcal{O};E^*)\subset (\mathcal{D}_p'(U,\mathcal{O};E))',$ in particular  $\mathcal{E}(U,\mathcal{O}^{oo};E)\subset (\mathcal{D}_p'(U,\mathcal{O}^{oo};E^*))'.$ Likewise this shows $\mathcal{D}'(U,\mathcal{C}^{oo};E)\subset (\mathcal{E}_p(U,\mathcal{C}^{oo};E^*)'.$

Conversely, to identify the duals, take a continuous linear form $T$ on $\mathcal{E}_i(U,\mathcal{O};E^*)\supset \mathcal{D}(U;E^*)$ (with continuous dense inclusion), it gives by restriction an element  $u\in(\mathcal{D}(U;E^*))'=\mathcal{D}'(U;E)$. Note that a linear form on $\mathcal{E}_p(U,\mathcal{C};E^*)$ in the polar case $\mathcal{O}=\mathcal{C}^o$ would have given this too. It suffices to prove $\text{supp}(u)\in\mathcal{O}^o$ since then it will coincide with the above pairing by the stated density of $\mathcal{D}(U;E^*)$. 
Thus take $W\in \mathcal{O}$. By a partition of unity like argument, there exists a non-negative function $\varphi\in \mathcal{E}(U),$ with $\varphi^{-1}(]0,\infty[)=W,$  and $\varphi$ vanishing with all its derivatives on $W^c$, i.e. $Ker^\infty(\varphi)^c={W}.$
For any $f\in \mathcal{E}(U;E^*), f\varphi\in \mathcal{E}(U,\mathcal{O};E^*)$ {(since $Ker^\infty(f\varphi)\supset Ker^\infty(\varphi)= W^c$ and $\mathcal{O}$ stable by subset thus $(Ker^\infty(f\varphi))^c\in \mathcal{O}$)} so that $f\mapsto T(f\varphi)$ is a continuous linear form, coinciding with $u\varphi$ on $\mathcal{D}(U;E)$, i.e. $u\varphi\in \mathcal{E}'(U;E^*)$, and $\text{supp}(u\varphi)=\overline{\text{supp}(u)\cap W}$ is compact, as wanted. (Note that for the equality of support above, $\subset$ is similar to what was explained at the beginning of our proof, and $\supset$, which is the really interesting case here comes from the fact that for $x\in \text{supp}(u)\cap W$ one can take a compactly supported function $f$ in $W$, $f(x)=\xi\in E^*-\{0\}$ supported as close to $x$ as we want, such that $\langle u,f\rangle\neq 0$ and, since $\varphi$ is strictly positive on its support included in $W$, $f\varphi^{-1}$ satisfies the same support assumption and $\langle u\varphi ,f\varphi^{-1}\rangle\neq 0$. Thus we got $x\in \text{supp}(u\varphi)$.)

We have thus deduced the equalities $(\mathcal{E}_i(U,\mathcal{O};E^*))'=\mathcal{D}'(U,\mathcal{O}^o;E)$  and $\mathcal{D}'(U,\mathcal{C};E)\subset (\mathcal{E}_p(U,\mathcal{C};E^*))'\subset\mathcal{D}'(U,\mathcal{C}^{oo};E)$ and thus $(\mathcal{E}_p(U,\mathcal{C}^{oo};E^*))'=\mathcal{D}'(U,\mathcal{C}^{oo};E).$

Assume now given a strongly bounded set $B \subset (\mathcal{E}_i(U,\mathcal{O};E^*))'$ and it is of course in $\mathcal{D}'(U,\{C\in \mathcal{F}, C\subset F\};E^*)$ for  $F=\overline{\cup_{u\in B}\text{supp}(u)}$. Since $\mathcal{D}(U;E^*)\subset\mathcal{E}_i(U,\mathcal{O};E^*)$ continuously, $B$ is of course bounded in $\mathcal{D}'(U,\{C\in \mathcal{F}, C\subset F\};E^*)\subset\mathcal{D}'(U;E^*)$, since the bornology is the one induced from the strong topology by the strong bornology of $\mathcal{D}'(U;E^*).$ It thus suffices to prove $F\in \mathcal{O}^{o}$ to get the stated conclusion. Let $G=\cup_{u\in B}\text{supp}(u)$. Let us first check it suffices to show that for any $W\in \mathcal{O}$, $\overline{G\cap W}$ is compact. Indeed, if $F\not\in \mathcal{O}^{o}$ there is such a $W$ with $\overline{F\cap W}$ not compact, thus there exists a sequence $x_n\in F\cap W$ which has no converging sub-sequence in $U$, i.e. ,by $\sigma$-compactness of $M$ and modulo extraction, either is unbounded or converges outside $U$ in $M$. Modulo taking a sub-sequence there is a sequence $B(x_n,\epsilon_n)\subset W$ of disjoint balls, with $\epsilon_n\to 0$ and  $x_n$ either with $||x_n||\to \infty$ or converging outside $U$. But there is $y_n\in B(x_n,\epsilon_n)\cap G$ by definition of a closure, thus $(y_n)\in W\cap G$ is a sequence with no converging sub-sequence in $U$ and $\overline{G\cap W}$ is not compact as expected.


Thus take $W\in \mathcal{O}$ and $\varphi$ as before. By definition, if $C$ is bounded in $\mathcal{E}(U;E^*)$, $\varphi C$ is bounded in $\mathcal{E}_i(U,\mathcal{O};E^*)$ (actually even in the inductive limit bornology because of the induction of the bornology of the pieces of the inductive limit by the one of $\mathcal{E}(U,E^*)$). We deduce from the definition of the strong bornology that $\sup_{u\in B,f\in C}|\langle u,f\varphi\rangle|<\infty$ i.e. $B\varphi$ is bounded in $\mathcal{E}'(U;E).$ 
 Thus there is a given compact set $K$ such that $\text{supp}(u)\cap W\subset \text{supp}(u\varphi)\subset K$ for all $u\in B,$ and thus we deduce $G\cap W\subset K$ as expected.



Similarly, take now a continuous linear form $T$ on $\mathcal{D}_i'(U,\mathcal{C};E)\supset \mathcal{E}'(U;E)$ defining by restriction an element $\varphi\in \mathcal{E}(U;E^*)=(\mathcal{E}'(U;E))'$ and similarly, it suffices to check the support condition on $\varphi$, i.e. $(Ker^\infty\varphi)^c\in \mathcal{C}^o$. Assume for contradiction that there is $F\in \mathcal{C}$ with $F\cap (Ker^\infty\varphi)^c$ not relatively compact so that there exists $x_n=\varphi_{i_n}(y_n)\in  F, \partial_{\alpha_n}(\varphi\circ \varphi_{i_n})(y_n)=\xi_n\neq 0$ with either $(x_n)$ tending to infinity or $x_n$ converges outside $U$ in $M$. Fix $\eta_n\in E_{x_n}$  with $\langle\eta_n, \xi_n\rangle= 1$ and write $\delta_{y_n,\eta_n}\in \mathcal{D}'(\varphi_{i_n}(U_{i_n});E\circ \varphi_n^{-1})$ the distribution defined by $\delta_{y_n,\eta_n}(f)=\langle \eta_n,f(y_n)\rangle$. Then the sum $\sum_n(-1)^{|\alpha_n|}\varphi_{i_n}^*(\partial_{\alpha_n}\delta_{y_n,\eta_n})$ has support in $\cup_n\{x_n\}\subset F$ and converges in  $\mathcal{D}'(U,\{C\in \mathcal{F}, C\subset F\};E)$ since we put the topology induced by $\mathcal{D}'(U;E)$ and the sum is finite against compactly supported sections in $U$. However, $T(\sum_{n\leq N}(-1)^{|\alpha_n|}\varphi_{i_n}^*(\partial_{\alpha_n}\delta_{y_n,\eta_n}))$ cannot converge when $N\to \infty$ since it equals $\langle \sum_{n\leq N}(-1)^{|\alpha_n|}\varphi_{i_n}^*(\partial_{\alpha_n}\delta_{y_n,\eta_n})),\varphi\rangle= \sum_{n\leq N}\langle \eta_{n}, \partial_{\alpha_n}(\varphi\circ\varphi_{i_n})(y_n)\rangle=N$ by construction, contradicting the continuity of $T$. As a conclusion $(Ker^\infty\varphi)^c\in  \mathcal{C}^{o}.$
We thus concluded to the two remaining duality equalities  $(\mathcal{D}_i'(U,\mathcal{C};E))'=\mathcal{E}(U,\mathcal{C}^o;E^*)$ and $(\mathcal{D}_p'(U,\mathcal{O}^{oo};E)'=\mathcal{E}(U,\mathcal{O}^{oo};E^*)$.

Assume finally given a strongly bounded set $B$ in $(\mathcal{D}_i'(U,\mathcal{C};E))'_b$, it is of course bounded in $\mathcal{E}(U,\{P\in \mathcal{T}, P\subset W\};E^*)$ with $W=\cup_{\varphi_\in B}(Ker^\infty\varphi)^c$. It suffices to show $W\in  \mathcal{C}^o.$
For, take $F\in \mathcal{C}$ and assume for contradiction $\overline{W\cap F}$ is not compact and get a sequence $x_n=\varphi_{i_n}(y_n)\in F$ as before diverging to infinity or converging outside $U$ with $\varphi^{(n)}\in B,\alpha_n$ such that $\partial_{\alpha_n}(\varphi^{(n)}\circ\varphi_{i_n})(y_n)=\xi_n\neq 0,$ and then $\eta_n$ as before.

Take $C=\{v_n=n(-1)^{|\alpha_n|}\varphi_{i_n}^*(\partial_{\alpha_n}\delta_{y_n,\eta_n}), n\in \N\}$. Then all elements of $C$ have support in $F$ and $C$ is bounded in $\mathcal{D}'(U;E)$ since on any compact, only finitely many $x_n$ appear and thus $v_n$ are non zero. Thus $C$ is bounded in $\mathcal{D}_i'(U,\mathcal{C};E)$.
But $\langle\varphi^{(n)},v_n\rangle=n$ thus $\sup_{v\in C,\varphi\in B}|\langle\varphi,v\rangle|=\infty$ giving the desired contradiction.
\end{proof}

\subsection{Functional analytic properties}
We refer to \cite{Domanski} and \cite[p.~96]{Wengenroth} for (PLN) and (PLS) spaces.
\begin{proposition}\label{FApropertiesSupport}
With the assumptions of the previous lemma, we have identifications of strong and inductive topologies $(\mathcal{E}_i(U,\mathcal{O};E^*))_b'=\mathcal{D}_i'(U,\mathcal{O}^o;E)$ and $(\mathcal{D}_i'(U,\mathcal{C};E))_b'=\mathcal{E}_i(U,\mathcal{C}^o;E^*)$ so that $\mathcal{E}_i(U,\mathcal{O}^{oo};E^*),\mathcal{D}_i'(U,\mathcal{C}^{oo};E)$ are always reflexive. 

 Moreover, 
  $\mathcal{D}_i'(U,\mathcal{C}^{oo};E)=\mathcal{D}_p'(U;\mathcal{C}^o;E)$ (resp. $\mathcal{E}_i(U,\mathcal{O}^{oo};E^*)=\mathcal{E}_p(U,\mathcal{O}^o;E^*)$.) so that they are ultrabornological complete nuclear 
  spaces.  If $\mathcal{C}^o$ is countably generated $\mathcal{D}_i'(U,\mathcal{C}^{oo};E)$ is a (PLN) space.
  
  We will write them $\mathcal{D}'(U,\mathcal{C}^{oo};E)$ and $\mathcal{E}(U,\mathcal{O}^{oo};E^*)$ (without mentioning the two agreeing topologies).


\end{proposition}
\begin{proof}
Note that once the duality result will be proved $(\mathcal{E}_i(U,\mathcal{C}^o;E^*))_b'=\mathcal{D}_i'(U,\mathcal{C}^{oo};E)$ will be a strong (and even Mackey) dual of a countable inductive limit of Fréchet nuclear spaces, thus a countable projective limit of their Mackey=strong duals which are (DFN) spaces, namely, we will have the (PLN) property by definition, in the stated case. 

We show that the identity map of $\mathcal{D}'(U,\mathcal{O}^o;E)$ with the inductive limit topology to $(\mathcal{E}_i(U,\mathcal{O};E^*))_b'$, i.e. itself with the strong topology is continuous. Thus take $B$ bounded in $(\mathcal{E}_i(U,\mathcal{O};E^*))$, it suffices to show $\sup_{\varphi\in B}|\langle u, \varphi\rangle|$ is bounded on each $\mathcal{D}'(U,\{C\in \mathcal{F}, C\subset F\};E),$ for $F\in \mathcal{O}^o$ by a seminorm of $\mathcal{D}'(U;E^*)$ since we put on this space the induced topology.
But $\mathcal{O}\subset \mathcal{O}^{oo}\subset \{F\}^o$, thus $\mathcal{E}_i(U,\mathcal{O};E^*)\to \mathcal{E}_i(U,\{F\}^o;E^*)$ is continuous, so that $B$ is bounded in the latter. By example 
\ref{countgenclosed}, $\{F\}^o$ is countably generated, so that the inductive limit defining $\mathcal{E}_i(U,\{F\}^o;E^*)$ is regular as we already noted, thus $B$ is bounded in some $\mathcal{E}(U,\{P\in \mathcal{T}, P\subset F^c\cup V\};E^*)$ with some $V\in \mathcal{B}\cap\mathcal{T}$, and especially in $\mathcal{E}(U;E^*).$ Take $W= F^c\cup V$ so that $\overline{W\cap F} \subset \overline{V}$ is compact, and we are in the situation of the beginning of the proof of lemma \ref{DualitySupport}, and we can take $f$, uniformly for all $\varphi\in B.$

Thus $\sup_{\varphi\in B}|\langle u, \varphi\rangle|\leq \sup_{\psi\in fB}|\langle u,\psi\rangle|$ which is a seminorm in $\mathcal{D}'(U;E)$ since $fB$ is bounded in $\mathcal{D}(U;E^*)$. This concludes to the stated continuity and thus the inductive limit topology on $\mathcal{D}'(U,\mathcal{O}^o)$ is stronger than the strong topology, which is stronger than Mackey, which is by lemma \ref{DualitySupport}, stronger than the inductive limit topology, and thus they all coincide.

One shows similarly that the identity map of $\mathcal{E}(U,\mathcal{C}^o;E^*)$ with the inductive limit topology to $(\mathcal{D}'_i(U,\mathcal{C};E))_b'$, i.e. itself with the strong topology is continuous and thus that both topologies coincide. 


Considering now the polar case, by identifying the strong topologies in the dual, we saw that the bounded sets are the same in 
$\mathcal{D}_i'(U,\mathcal{O}^o;E),(\mathcal{E}_i(U,\mathcal{O};E^*))'_b$, thus from lemma \ref{DualitySupport}, $\mathcal{D}_{i}'(U,\mathcal{O}^o)$ is a regular inductive limit so that we can even write another equivalent regular inductive limit description since $\mathcal{K}\subset \mathcal{O}^o$ :
$$\mathcal{D}_{i}'(U,\mathcal{O}^o)=\text{ind}\lim_{C\in \mathcal{O}^{o}} (\mathcal{D}_i'(U,\{F\in \mathcal{F}, F\subset C\}\cup \mathcal{K};E)).$$

By a general result on regular inductive limits (see e.g. the discussion of \cite[p292-293]{Kothe} or \cite[p57]{Bierstedt}) on deduces :\begin{align*}(\mathcal{D}_{i}'(U,\mathcal{O}^o;E))'_b&=\text{proj}\lim_{C\in \mathcal{O}^{o}} (\mathcal{D}_i'(U,\{F\in \mathcal{F}, F\subset C\}\cup \mathcal{K};E))'_b)\\&=\text{proj}\lim_{C\in \mathcal{O}^{o}} (\mathcal{E}_i(U,\{C\}^o;E^*))=\mathcal{E}_p(U,\mathcal{O}^{o};E^*)\end{align*}
with the second equality coming from the first part of the proof, applicable since $\{F\in \mathcal{F}, F\subset C\}\cup \mathcal{K}$ is stable by subset and contain compact sets.
Likewise, from the coincidence of bounded sets in $\mathcal{E}_i(U,\mathcal{C}^o;E^*),(\mathcal{D}'_i(U,\mathcal{C};E))'_b$, the inductive limit $\mathcal{E}_i(U,\mathcal{C}^o;E^*)$ is regular and~:
$$(\mathcal{E}_{i}'(U,\mathcal{C}^o,E^*
))'_b
=\text{proj}\lim_{O\in \mathcal{C}^{o}} (\mathcal{E}_i(U,\{O\}^o;E))=\mathcal{D}'_p(U;\mathcal{C}^{o};E)$$
For the functional analytic properties, we already saw the projective limits were complete and nuclear, their strong dual are thus ultrabornological \cite[p15]{HogbeNlendMoscatelli}, and barrelled as strong dual of semi-reflexive spaces  \cite[$\S 23.3.(4)$]{Kothe}.
\end{proof}

\subsection{Continuous dense injections : openly generated and enlargeable support conditions}
\begin{remark}\label{opengen}
The asymmetry between the support condition on distributions and functions in this section would be technically annoying for us when we interpolate between functions and distributions by putting wave front set conditions. Instead of developing vanishing condition for distribution similar to those appearing in the duality above for functions, we restrict again to a class of $\mathcal{C}$ where the problem won't appear and that will be sufficient for our purposes.  Let $\mathscr{O}_{\mathcal{C}}=\{U=Int(C) : C\in\mathcal{C}\}, \mathscr{C}_{\mathcal{O}}=\{C=\overline{U}:  U\in \mathcal{O}\}.$ We say a polar family of closed cone is \emph{openly generated} if $\mathcal{C}= (\mathscr{C}_{\mathscr{O}_{\mathcal{C}}})^{oo}.$\footnote{Note that a countably generated family $\mathcal{C}=\{F_i\}^{oo}$ is openly generated if 
 it is countably generated by regular closed sets, i.e. if $F_i=\overline{Int(F_i)}$. Indeed, if $F_i$ verifies this property $F_i\in\mathscr{C}_{\mathscr{O}_{\mathcal{C}}}$ i.e. $(\mathscr{C}_{\mathscr{O}_{\mathcal{C}}})^{oo}\supset \mathcal{C}$ and this is the only non obvious inclusion.
} Now if $C\in \mathscr{C}_{\mathscr{O}_{\mathcal{C}}}$ and $W$ open, if $\overline{W\cap C}$ is not compact, there is a sequence $x_n\in  W\cap C$ converging outside of $U$ maybe to infinity and since $C=\overline{Int(C)}$ by assumption (and a standard result namely $C\mapsto\overline{Int(C)}$ is involutive), then one can find $y_n\in W\cap Int(C)$ of the same kind and thus $\overline{W\cap Int(C)}$ is not compact. As a consequence, $\mathcal{C}^o=\{W\in \mathcal{T}:\ \forall U\in \mathscr{O}_{\mathcal{C}} ,\overline{W\cap U}\in \mathcal{K}\},$ when $\mathcal{C}$ is \emph{openly generated}. But if $W\in \mathcal{C}^o$ and since $ \overline{W}\cap U\subset \overline{W\cap U}$ when $U$ is open, we have $\overline{\overline{W}\cap U}=\overline{W\cap U},$ and thus :
$$\mathcal{C}^o=\{W\in \mathcal{T}\ :\ \forall U\in \mathscr{O}_{\mathcal{C}} ,\overline{\overline{W}\cap U}\in \mathcal{K}\}=\{W\in \mathcal{T}, \overline{W}\in (\mathscr{O}_{\mathcal{C}})^o\}$$
from which one deduces half of the next : 
\end{remark}
\begin{lemma}\label{opengenlemma} If $\mathcal{C}$ is \emph{openly generated}, so is $(\mathscr{O}_{\mathcal{C}})^o$ and we have :
$$\mathcal{E}(U,\mathcal{C}^o;E^*)= \mathcal{E}(U;E^*)\cap\mathcal{D}'(U,(\mathscr{O}_{\mathcal{C}})^o;E^*)\ \ \text{and}\ \ \mathcal{C}=(\mathscr{O}_{(\mathscr{O}_{\mathcal{C}})^o})^o.$$

Finally, we have continuous embeddings $\mathcal{E}(U,\mathcal{C}^o;E^*)\to \mathcal{D}'(U,(\mathscr{O}_{\mathcal{C}})^o;E^*),\mathcal{E}(U,(\mathscr{O}_{\mathcal{C}})^{oo};E^*)\to \mathcal{D}'(U,\mathcal{C};E^*).$
\end{lemma}
\begin{proof}Indeed, for the last equality, if $U\in \mathscr{O}_{(\mathscr{O}_{\mathcal{C}})^o},$ $U=Int(C)$ with $C\in (\mathscr{O}_{\mathcal{C}})^o$, then $\overline{U}\subset C$ since $C$ closed thus since $(\mathscr{O}_{\mathcal{C}})^o$ is polar thus stable by subset $\overline{U}\in (\mathscr{O}_{\mathcal{C}})^o$ and thus $U\in \mathcal{C}^o,$ by the previously established equation and by polarity $\mathcal{C}\subset (\mathscr{O}_{(\mathscr{O}_{\mathcal{C}})^o})^o.$
Conversely if $W\in \mathcal{C}^o$ then  $\overline{W}\in (\mathscr{O}_{\mathcal{C}})^o$ and $W\subset Int(\overline{W})\in \mathscr{O}_{(\mathscr{O}_{\mathcal{C}})^o}$, thus $W\in (\mathscr{O}_{(\mathscr{O}_{\mathcal{C}})^o})^{oo}$ and $(\mathscr{O}_{(\mathscr{O}_{\mathcal{C}})^o})^o\subset \mathcal{C}^{oo}=\mathcal{C}.$ It remains to check $(\mathscr{O}_{\mathcal{C}})^o$ is openly generated, if $C\in \mathscr{C}_{\mathscr{O}_{(\mathscr{O}_{\mathcal{C}})^o}},$ $C=\overline{W}$ with $W\in \mathscr{O}_{(\mathscr{O}_{\mathcal{C}})^o}\subset \mathcal{C}^o,$ 
and for $U\in \mathscr{O}_{\mathcal{C}}$ thus $U=Int(D),$ $D\in \mathcal{C}$, then $\overline{\overline{W}\cap U}=\overline{{W}\cap U}\subset \overline{{W}\cap D}$ which is compact by definition, thus $C\in (\mathscr{O}_{\mathcal{C}})^o$, i.e. $\mathscr{C}_{\mathscr{O}_{(\mathscr{O}_{\mathcal{C}})^o}}\subset (\mathscr{O}_{\mathcal{C}})^o,$ and $(\mathscr{C}_{\mathscr{O}_{(\mathscr{O}_{\mathcal{C}})^o}})^{oo}\subset (\mathscr{O}_{\mathcal{C}})^o$. Conversely take $V\in (\mathscr{C}_{\mathscr{O}_{(\mathscr{O}_{\mathcal{C}})^o}})^o$, so that $\overline{V\cap C}=\overline{V\cap W}=\overline{\overline{V}\cap W}$ is compact for any $C$ as above, i.e. $\overline{V}\in (\mathscr{O}_{(\mathscr{O}_{\mathcal{C}})^o})^o=\mathcal{C}$. Thus $V\subset Int(\overline{V})\in \mathscr{O}_{\mathcal{C}}$ and any $F\in (\mathscr{O}_{\mathcal{C}})^o$ satisfies $\overline{F\cap Int(\overline{V})}\supset \overline{F\cap V}$ compact, i.e. $F\in (\mathscr{C}_{\mathscr{O}_{(\mathscr{O}_{\mathcal{C}})^o}})^{oo}$ which concludes.

For the continuous embedding, since we checked that the projective limit and injective limit are the same, we choose to prove $\mathcal{E}_i(U,\mathcal{C}^o;E^*)\to \mathcal{D}_p'(U,\mathscr{O}_{\mathcal{C}};E^*),$ is continuous. From general properties of inductive and projective limits, it suffices to show for $O\in\mathcal{C}^o,W=Int(F)\in \mathscr{O}_{\mathcal{C}}, F\in \mathcal{C}$ the continuity of  $\mathcal{E}(U,\{P\in \mathcal{T}, P\subset O\};E^*)\to \mathcal{D}_i'(U,\{W\}^o;E^*).$ But since $\overline{\overline{O}\cap W}=\overline{O\cap W}\subset \overline{O\cap F}$ is compact, we have $\overline{O}\in  \{W\}^o$ and an embedding : $$\mathcal{E}(U,\{P\in \mathcal{T}, P\subset O\};E^*)\to \mathcal{D}'(U,\{C\in \mathcal{F}, C\subset \overline{O}\};E^*)\to \mathcal{D}_i'(U,\{W\}^o;E^*)$$ and since the topologies of the first two  spaces are induced by  $\mathcal{E}(U;E^*),\mathcal{D}'(U;E^*),$ the continuity follows. By the equality, the second stated continuous map is a special case with $\mathcal{C}$ replaced by $(\mathscr{O}_{\mathcal{C}})^{o}$. From this also follows $\mathcal{E}(U,\mathcal{C}^o;E^*)\subset \mathcal{E}(U;E^*)\cap\mathcal{D}'(U,(\mathscr{O}_{\mathcal{C}})^o;E^*)$. For the reversed inclusion, note that if $\varphi\in \mathcal{E}(U;E^*)\cap\mathcal{D}'(U,(\mathscr{O}_{\mathcal{C}})^o;E^*)$, then $\overline{Ker^\infty(\varphi)^c}\in (\mathscr{O}_{\mathcal{C}})^o$, and from  $(Ker^\infty(\varphi))^c\subset Int(\overline{Ker^\infty(\varphi)^c})$ one deduces $(Ker^\infty(\varphi))^c\in (\mathscr{O}_{(\mathscr{O}_{\mathcal{C}})^o})^{oo}=\mathcal{C}^o$ as expected.
\end{proof}
{Our next task is to obtain some density of smooth functions with controlled support in a convenient way. We start by a first attempt we call uniform enlargeability, which depends strongly on a riemannian metric and won't be such convenient to give interesting examples on globally hyperbolic spacetimes. We then adapt it with a notion of enlargeability to get a more local version  that will enable to reach our motivating example \ref{enlargeable3} and various stability properties.}

\begin{remark}\label{uniformlyenlargeablermk}
Even with an openly generated family, the support condition could constrain some wave front sets and thus interact with the computations of duals in a complicating way. In the typical examples we are interested in compact support, full support, future or past compact support in globally hyperbolic space time, and products of those conditions on product spaces, these difficulties don't appear. We introduce a condition taking care of suppressing these effects. 
Recall we fixed a complete  Riemannian metric $D$ (with compact closed balls). Then  the uniform epsilon enlargement of a closed set $C$, is the closed set $C_{u \epsilon}=\{x\in U,\ \exists y\in C,\ D(x,y)\leq \min(\epsilon,D(y,U^c)(1-e^{-\epsilon})/2)\}$ (closedness follows from compactness of balls). Note also that $C_{u \epsilon}$ is compact when $C$ is.
We say a family of closed sets is \emph{uniformly enlargeable} (for the metric $D$) if $$(C\in \mathcal{C} \Rightarrow \exists \epsilon>0,\ C_{u\epsilon}\in \mathcal{C}).$$
Of course a polar uniformly enlargeable family $\mathcal{C}$ is openly generated  since $\overline{Int(C_{u\epsilon})}=C_{u\epsilon}\in \mathscr{C}_{\mathscr{O}_{\mathcal{C}}}$ if $\epsilon>0$.
Note also that, using geodesics, one checks $C_{u(2\epsilon)}\subset ((C_{u\epsilon})_{u\epsilon})_{u\epsilon}.$ 

Let us check that if $\mathcal{C}$ is uniformly enlargeable (not necessarily polar)
 then  $(\mathscr{O}_{\mathcal{C}})^o$ is also uniformly enlargeable and thus so is $\mathcal{C}^{oo}$. Take $V\in (\mathscr{O}_{\mathcal{C}})^o$, thus for any $C\in \mathcal{C}, \epsilon>0$ with $C_{2\epsilon}\in \mathcal{C}$ we have $\overline{C_{\epsilon}\cap V}\subset \overline{Int(C_{2\epsilon})\cap V}$ is compact. Note that if  $x\in C\cap V_{u\epsilon}$, we have $y\in V$, with $D(x,y)\leq \eta=\min(\epsilon,D(y,U^c)(1-e^{-\epsilon})/2)$, and $B(y,\eta)\subset U$ so that for $z\in U^c$, $D(x,z)\geq D(y,z)-D(x,y)\geq  D(y,U^c)-D(y,U^c)(1-e^{-\epsilon})/2=D(y,U^c)(1+e^{-\epsilon})/2$ so that $D(x,U^c)\geq D(y,U^c)(1+e^{-\epsilon})/2$, and finally $D(x,y)\leq \min(\epsilon,2D(x,U^c)(1-e^{-\epsilon})/2(1+e^{-\epsilon}))$ so that for $\epsilon < \ln(3)$, $D(x,y)\leq \min(\epsilon,D(x,U^c)(1-e^{-\delta(\epsilon)})/2)$ with $\delta(\epsilon)=\ln\left(\frac{1+e^{-\epsilon}}{3e^{-\epsilon}-1}\right)\to_{\epsilon\to 0} 0$ so that $y\in C_{\delta(\epsilon)}$ and thus, to sum up, for $\epsilon < \ln(3)$, $\overline{C\cap V_{u\epsilon}}\subset (\overline{C_{u\delta(\epsilon)}\cap V})_{u\epsilon}$ which is again compact for $\epsilon$ small enough (since an $\epsilon$ enlargements of a compact is compact since $D$ has compact closed balls). 
Thus indeed $V_{u\epsilon}\in (\mathscr{O}_{\mathcal{C}})^o$ as expected for $\epsilon$ small enough, implying  that $(\mathscr{O}_{\mathcal{C}})^o$ is uniformly enlargeable.
Moreover, if $F\in \mathcal{C}^{oo}$, $V\in (\mathscr{O}_{\mathcal{C}})^o$, the reasoning above implies $Int(V)\in \mathcal{C}^{o}$ and 
$\overline{F\cap Int(V)}$ compact, i.e. $\mathcal{C}^{oo}\subset (\mathscr{O}_{(\mathscr{O}_{\mathcal{C}})^o})^o.$ Conversely, if $F\in (\mathscr{O}_{(\mathscr{O}_{\mathcal{C}})^o})^o$ and let $V\in \mathcal{C}^{o}$, then for any $C\in \mathcal{C}$, for $\epsilon$ small enough $\overline{V\cap C_\epsilon}$ and ${C\cap \overline{V}}\subset{C\cap (\overline{V})_\epsilon}\subset (C_{\delta(\epsilon)}\cap \overline{V})_\epsilon\subset (\overline{C_{2\delta(\epsilon)}\cap V})_{2\epsilon}$ are compact, thus $\overline{V}\in (\mathscr{O}_{\mathcal{C}})^o$ and $\overline{V\cap F}\subset \overline{Int(\overline{V})\cap F}$ is compact and thus $F\in \mathcal{C}^{oo}$. We deduce $\mathcal{C}^{oo}= (\mathscr{O}_{(\mathscr{O}_{\mathcal{C}})^o})^o$ and its uniform enlargeability from the first result.

Finally, we also note that for any closed $C$, $\{C_\epsilon, \epsilon>0\}^{oo}$ is a polar uniformly enlargeable family, since it is the bipolar of an uniformly enlargeable family (the family  of closed subsets of $C_\epsilon$).


Indeed one checks that if $\epsilon_2\leq e^{-\epsilon_1}\ln(4)/9,$ we have $(C_{u\epsilon_1})_{u\epsilon_2}\subset C_{u(\epsilon_1+3e^{\epsilon_1}\epsilon_2)}.$ 


\end{remark}

\begin{ex}\label{enlargeable1}
On any $U$, $\mathcal{F},\mathcal{K}=(\mathscr{O}_\mathcal{F})^o$ are uniformly enlargeable, and if $\mathcal{C}_1,\mathcal{C}_2$ are polar uniformly enlargeable families on $U_1\subset M_1,U_2\subset M_2$, so is  $(\mathcal{C}_1\times \mathcal{C}_2)^{oo}$ on $U=U_1\times U_2$, with the product Riemannian metric on $M=M_1\times M_2$. Indeed from example \ref{ProductCase}, if $C\in (\mathcal{C}_1\times \mathcal{C}_2)^{oo},$ $C\subset C_1\times C_2$, $C_i\in \mathcal{C}_i$, and $(C_1\times C_2)_{u\epsilon}\subset (C_1\times U_2)_\epsilon\cap (U_1\times C_2)_{u\epsilon}$ and if $x=(x_1,x_2)\in (C_1\times U_2)_{u\epsilon}$, let $y=(y_1,y_2)\in C_1\times U_2$  as in the definition with  $D(x,y)\leq \min(\epsilon,D(y,U^c)(1-e^{-\epsilon})/2)$ 
note that $D(x,y)
\geq D_1(x_1,y_1)$ and $D(y,U^c)\leq D_1(y_1,U_1^c)$, using $U_1^c\times \{y_2\}\subset U^c$, 
and thus $D_1(x_1,y_1)\leq \min(\epsilon,D_1(y_1,U_1^c)(1-e^{-\epsilon})/2)$ i.e. $(C_1\times U_2)_{u\epsilon}\subset (C_1)_{u\epsilon}\times U_2$, and thus by symmetry we have the concluding inclusion~:$(C_1\times C_2)_{u\epsilon}\subset(C_1)_{u\epsilon}\times(C_2)_{u\epsilon}.$
\end{ex}

If uniform enlargeability is natural if $M=\R^n$, it is not local enough in general to prove it easily for our motivating example \ref{enlargeable3}, we will thus use a variant we now discuss.
 The  $\epsilon$ enlargement, for $\epsilon\in ]0,\infty[^I$ of a closed set $C$ (depending of the metric $D$ and of the locally finite covering $(U_i)$), is the closed set $$C_{\epsilon}=\bigcup_{i\in I} \overline{U_i}\cap 
 (\overline{U_i}\cap C)_{u\epsilon_i}$$  (closedness follows from local finiteness of the covering). Note also that $C_{\epsilon}$ is compact when $C$ is (since then the union is finite by local finiteness of $(\overline{U_i})$). The same is true for $C_{w\epsilon}=\bigcup_{i\in I}  
 (\overline{U_i}\cap C)_{u\epsilon_i}\supset C_{\epsilon}.$
\begin{definition}\label{enlargeable} We say a family of closed sets is \emph{enlargeable} (for the metric $D$) if $$(C\in \mathcal{C} \Rightarrow \exists \epsilon\in ]0,\infty[^I,\ C_{\epsilon}\in \mathcal{C}).$$
\end{definition}
Note that changing the locally finite cover does not change the notion as soon as the collection $\mathcal{C}$ is stable by subsets. Indeed for each term of the new cover  $V_i$ consider  the terms of the previous cover $U_j, j\in J_i$ finite, intersecting its closure i.e. $\overline{V_i}\subset \cup_{j\in J_i}U_j$ thus $\overline{V_i}\cap C\subset\cup_{j\in J_i}U_j\cap C\subset W_i=\cup_{j\in J_i}U_j\cap Int((\overline{U_j}\cap C)_{u\epsilon_i})$ which is open so that we let $\eta_i=d(\overline{V_i}\cap C,W_i^c)/2>0$ by compactness of $\overline{V_i}\cap C$ and thus $(\overline{V_i}\cap C)_{u\eta_i}\subset W_i\subset \cup_{j\in J_i}\overline{U_j}\cap((\overline{U_j}\cap C)_{u\epsilon_i})\subset C_\epsilon$ and $C_{\eta}^{(V)}$, computed with the new covering, is in  $C_\epsilon$ thus in $C_{\eta}^{(V)}\in\mathcal{C}$ by stability by subset. This argument shows basically how more local the notion of enlargeability is, as we will use later.

\begin{proposition}\label{EnlargeableStable}
A polar enlargeable family $\mathcal{C}$ is openly generated. Moreover, for any enlargeable family $\mathcal{C},$ $(\mathscr{O}_{\mathcal{C}})^o$ is also enlargeable and thus so is $\mathcal{C}^{oo}.$
\end{proposition}
\begin{proof}
Of course a polar enlargeable family $\mathcal{C}$ is openly generated  since $\overline{Int(C_{\epsilon})}=C_{\epsilon}\in \mathscr{C}_{\mathscr{O}_{\mathcal{C}}}$.

Let us now check that if $\mathcal{C}$ is enlargeable (not necessarily polar) 
 then  $(\mathscr{O}_{\mathcal{C}})^o$ is also enlargeable and thus so is $\mathcal{C}^{oo}$. Take $V\in (\mathscr{O}_{\mathcal{C}})^o$, thus for any $C\in \mathcal{C}, \epsilon_i>0$ with $C_{(2\epsilon)}\in \mathcal{C}$ we have $\overline{C_{\epsilon}\cap V}\subset \overline{Int(C_{(2\epsilon)})\cap V}$ is compact. 
 Let $\epsilon_i < \ln(3)$, using the uniform case above 
 one gets :\begin{align*}\overline{C\cap V_\epsilon}&=\overline{\bigcup_i\overline{(\overline{U_i}\cap C)\cap (\overline{U_i}\cap V)_{u\epsilon_i}}}\\&\subset\overline{\bigcup_i(\overline{(C\cap \overline{U_i})_{u\delta(\epsilon_i)}\cap (\overline{U_i}\cap V)})_{u\epsilon_i}}\\&\subset\overline{\bigcup_i(\overline{C_{\delta(\epsilon)}\cap (\overline{U_i}\cap V)})_{u\epsilon_i}}\\&= (C_{\delta(\epsilon)}\cap V)_{w\epsilon}\end{align*} which is again compact for $\epsilon$ small enough. 
Thus indeed $V_\epsilon\in (\mathscr{O}_{\mathcal{C}})^o$ as expected for $\epsilon$ small enough, implying  that $(\mathscr{O}_{\mathcal{C}})^o$ is enlargeable.
Moreover, if $F\in \mathcal{C}^{oo}$, $V\in (\mathscr{O}_{\mathcal{C}})^o$, the reasoning above implies $Int(V)\in \mathcal{C}^{o}$ and 
$\overline{F\cap Int(V)}$ compact, i.e. $\mathcal{C}^{oo}\subset (\mathscr{O}_{(\mathscr{O}_{\mathcal{C}})^o})^o.$ Conversely, if $F\in (\mathscr{O}_{(\mathscr{O}_{\mathcal{C}})^o})^o$ and let $V\in \mathcal{C}^{o}$, then for any $C\in \mathcal{C}$, for $\epsilon$ small enough, $\overline{V\cap C_\epsilon}$ and ${C\cap \overline{V}}\subset{C\cap (\overline{V})_\epsilon}\subset (C_{\delta(\epsilon)}\cap \overline{V})_{w\epsilon}\subset (\overline{C_{2\delta(\epsilon)}\cap V})_{w(2\epsilon)}$ are compact, thus $\overline{V}\in (\mathscr{O}_{\mathcal{C}})^o$ and $\overline{V\cap F}\subset \overline{Int(\overline{V})\cap F}$ is compact and thus $F\in \mathcal{C}^{oo}$. We deduce $\mathcal{C}^{oo}= (\mathscr{O}_{(\mathscr{O}_{\mathcal{C}})^o})^o$ and its enlargeability from the first result.
\end{proof}
Finally, we also note that for any closed $C$, $\eta\in]0,1[^I$, $\{C_\epsilon, \epsilon<\eta\}^{oo},\{C_{(1-1/n)\eta}, n\in \N\}^{oo}$ are polar enlargeable families, since they are the bipolar of an enlargeable family (the family  of closed subsets of some $C_\epsilon$).


Indeed one deduces from the uniform case that if $\epsilon_2\leq e^{-\epsilon_1}\ln(4)/9,$ we have $(C_{\epsilon_1})_{\epsilon_2}\subset C_{\epsilon_1+3e^{\epsilon_1}\epsilon_2}$ (all the operations applied pointwise in $i\in I$). 



\begin{ex}\label{enlargeable2}
Of course uniformly enlargeable families are enlargeable.
If $\mathcal{C}_1,\mathcal{C}_2$ are polar (uniformly) enlargeable families on $U$, so are  $\mathcal{C}_1\cap\mathcal{C}_2,\mathcal{C}_1\vee\mathcal{C}_2.$ Recall the last notation has been introduced in example \ref{IntersectionCase}, and we use the relation proved there and proposition \ref{EnlargeableStable} to deduce the second case from the easy first case. 

Moreover if $\mathcal{C}_1,\mathcal{C}_2$ are  polar enlargeable families on $U_1\subset M_1,U_2\subset M_2$, so is  $(\mathcal{C}_1\times \mathcal{C}_2)^{oo}$ on $U=U_1\times U_2$, with the product Riemannian metric on $M=M_1\times M_2$. It suffices to consider the product covering say $U_1\times V_j$ indexed by $I_1\times I_2$ and note that for $\eta_i\leq 1,\epsilon_j\leq 1$, if we write $(\epsilon\eta)_{(i,j)}=\epsilon_i\eta_j$ and use mostly the uniform case in the third line, one gets : \begin{align*}(C_1\times C_2)_{\epsilon\eta}&=\cup_{(i,j)\in I_1\times I_2}\overline{U_i\times V_j}\cap (\overline{U_i\times V_j}\cap C_1\times C_2)_{u(\epsilon_{i}\eta_j)}\\&=\cup_{(i,j)\in I_1\times I_2}\overline{U_i}\times \overline{V_j}\cap ( (\overline{U_i}\cap C_1)\times (\overline{V_j}\cap C_2))_{u(\epsilon_{i}\eta_j)} \\&\subset\cup_{(i,j)\in I_1\times I_2}\overline{U_i}\times \overline{V_j}\cap((\overline{U_i}\cap C_1)_{u\epsilon_i\eta_j}\times(\overline{V_j}\cap C_2)_{u\epsilon_i\eta_j})\\&\subset\cup_{(i,j)\in I_1\times I_2}[\overline{U_i}\cap (\overline{U_i}\cap C_1)_{u\epsilon_i}]\times[ \overline{V_j}\cap(\overline{V_j}\cap C_2)_{u\eta_j}]
\\& =(C_1)_\epsilon\times (C_2)_\eta .\end{align*}  
Thus for any $C_i\in\mathcal{C}_i$, for $\epsilon, \eta$ small enough $(C_1\times C_2)_{\epsilon\eta}\in (\mathcal{C}_1\times \mathcal{C}_2)^{oo}.$ This concludes.
\end{ex}

\begin{ex}\label{enlargeable3}
 If $U=M$ is globally hyperbolic for some time oriented Lorentzian metric (and connected 
  $\sigma$-compact as before), with $\mathfrak{C}(M)$ the set of Cauchy surfaces, then the following family of closed sets are all polar enlargeable families (for any complete Riemannian metric): Timelike-compact closed sets $\mathcal{K}_T=\mathcal{K}_F\cap \mathcal{K}_P$, future-compact closed sets $$\mathcal{K}_F=\{F\in\mathcal{F}, \exists \Sigma\in \mathfrak{C}(M),\  F\subset J^-(\Sigma)\}=\{F\in\mathcal{F}, \forall K\in\mathcal{K} ,\  F\cap J^+(K)\in \mathcal{K}\},$$ past-compact closed sets $$\mathcal{K}_P=\{F\in\mathcal{F}, \exists \Sigma\in \mathfrak{C}(M),\  F\subset J^+(\Sigma)\}=\{F\in\mathcal{F}, \forall K\in\mathcal{K} ,\  F\cap J^-(K)\in \mathcal{K}\},$$ spacelike-compact closed sets $$\mathcal{SK}=\{F\in\mathcal{F}, \forall \Sigma\in \mathfrak{C}(M),\  F\cap \Sigma\in \mathcal{K}\}=\{F\in\mathcal{F}, \exists K\in\mathcal{K} ,\  F\subset J^+(K)\cup J^-(K) \},$$ future-spacelike-compact 
   closed sets $$\mathcal{SK}_F=\{F\in\mathcal{F}, \forall \Sigma\in \mathfrak{C}(M),\  F\cap J^+(\Sigma)\in \mathcal{K}\}=\{F\in\mathcal{F}, \exists K\in\mathcal{K} ,\  F\subset J^-(K)\}$$ and  past-spacelike-compact 
    closed sets $$\mathcal{SK}_P=\{F\in\mathcal{F}, \forall \Sigma\in \mathfrak{C}(M),\  F\cap J^-(\Sigma)\in \mathcal{K}\}=\{F\in\mathcal{F}, \exists K\in\mathcal{K} ,\  F\subset J^+(K)\}$$ (cf. \cite{Sanders} for the equivalent definitions). Indeed, from the first definitions above one gets : $\mathcal{SK}_F= (\mathscr{O}_{\mathcal{K}_P})^o$, $\mathcal{SK}_P= (\mathscr{O}_{\mathcal{K}_F})^o.$
 From the second formulas, we also have $\mathcal{K}_P= (\mathscr{O}_{\mathcal{SK}_F})^o$, $\mathcal{K}_F= (\mathscr{O}_{\mathcal{SK}_P})^o.$
 
  Note that by $\sigma$-compactness of $M$, one sees from the second definition that $\mathcal{SK},\mathcal{SK}_F,\mathcal{SK}_P$ are all countably generated and thus from example \ref{IntersectionCase}, one deduces with the notation there that $\mathcal{SK}=\mathcal{SK}_F\vee \mathcal{SK}_P,$ $(\mathscr{O}_{\mathcal{SK}})^{oo}=(\mathscr{O}_{\mathcal{SK}_F})^{oo}\vee (\mathscr{O}_{\mathcal{SK}_P})^{oo}$ (using $(\mathscr{O}_{\mathcal{SK}_P})^{oo}=\{O\in\mathcal{T}, \exists K\in\mathcal{K} ,\  T\subset I^+(K)\}$ and similar variants giving countable generatedness). We also have (using \cite[Th 2.2]{Sanders}) $(\mathscr{O}_{\mathcal{K}_T})^{oo}=(\mathscr{O}_{\mathcal{K}_F})^{oo}\cap (\mathscr{O}_{\mathcal{K}_P})^{oo}$. 
   Thus again from example \ref{IntersectionCase}, $(\mathscr{O}_{\mathcal{SK}})^{o}=(\mathscr{O}_{\mathcal{SK}_F})^{o}\cap (\mathscr{O}_{\mathcal{SK}_P})^{o}=\mathcal{K}_T,(\mathscr{O}_{\mathcal{K}_T})^{o}=\mathcal{SK}_F\vee \mathcal{SK}_P=\mathcal{SK}.$

From these identities, the previous example and proposition \ref{EnlargeableStable}, it only remains to show that $\mathcal{K}_F$ (and the analogous $\mathcal{K}_P$) is enlargeable. Indeed, take $F \subset J^{-}(\Sigma)$ and a Cauchy hypersurface $\Sigma'\subset I^{+}(\Sigma) $ (by global hyperbolicity of this set).
By compactness of $\overline{U_i}\cap F\subset I^{-}(\Sigma')$, this set (when non empty) is at distance $2\epsilon_i$ of $\Sigma'$ (in the empty case one can take $\epsilon_i=1/2$) and thus $(\overline{U_i}\cap F)_{u\epsilon_i}\subset J^{-}(\Sigma')$ thus by union, $F_\epsilon\subset J^{-}(\Sigma')$ so that $F_\epsilon\in\mathcal{K}_F$ as expected.

Note that in this case our proposition \ref{FApropertiesSupport} (and lemma \ref{opengenlemma} for the identification of definitions) recover \cite[Th 4.3]{Sanders}, with a corrected proof. Indeed, without mentioning the sketchy treatment of topologies, it is claimed in the proof there that if one takes $\varphi\in (\mathcal{E}(M,\mathcal{K}_P,E))'$, say for instance the one given from a smooth function non vanishing  on some $J^{-}(K)$ supported on some $J^{-}(K')$, then its support should satisfy $\text{supp}(\varphi)\cap I^+(\Sigma)$ compact for any Cauchy surface $\Sigma$, which is clearly not the case in our example since the intersection is even often not closed and contains $I^+(\Sigma)\cap J^{-}(K)$ ! The issue is that the ``restriction" to $I^+(\Sigma)$ of a distribution does not have the relation it claims to have with the restriction of functions, one has to work with multiplication by smooth functions. Anyways our statement confirms 
the topological properties of the intuitive and true statement of \cite{Sanders}.
\end{ex}

Technically, this condition gives a convenient density result using a standard convolution trick we will detail more later for distributions with wave front set conditions. It suffices to approximate $u$ by $\sum (\rho_i*[(uf_i)\circ{\varphi_i}^{-1}])\circ \varphi_i$ in choosing the support of $\rho_i$ small enough so that the convolution does not leave $\varphi_i(U_i)$ and moreover the support stays in some $C_\epsilon.$ We thus got :

\begin{lemma}\label{enlargeable}
If $\mathcal{C}$ is enlargeable, then the inclusions $\mathcal{E}(U,\mathcal{C}^o;E^*)\to \mathcal{D}'(U,(\mathscr{O}_{\mathcal{C}})^o;E^*)$,

\noindent$\mathcal{E}(U,(\mathscr{O}_{\mathcal{C}})^{oo};E^*)\to \mathcal{D}'(U,\mathcal{C};E^*),$ have sequentially dense images.
\end{lemma}

\section{Wave front set and Dual wave front set}
Let $\Gamma$  a closed cone (resp $\Lambda$ an open cone) in $T^*U-\{0\}$, we introduce H\"ormander spaces with conditions on wave front set (cf. e.g. \cite{BrouderDabrowski}) : $\mathcal{D}'_\Gamma(U;E)=\{u\in \mathcal{D}'(U;E)\ | \ WF(u)\subset \Gamma\}$,  
$\mathcal{E}'_\Lambda(U;E)=\{u\in \mathcal{E}'(U;E)\ | \ WF(u)\subset \Lambda\}$. 

Recall that $WF(u)=\bigcup_{i\in I}d\varphi_i^*[\cup_{j=1}^eWF([(\varphi_i^{-1})^*(f_iu)]_j)]\subset \dot{T}^*U,$ i.e. the union of all wave front sets of coordinates  $[(\varphi_i^{-1})^*(f_iu)]_j$ of the distributional section when pushed forward by the trivializing chart, and then pulled back to the manifold.
(Note we may also write $(\varphi_i^{-1})^*(f_iu)=(f_iu)\circ \varphi_i^{-1}$.)

 More generally, for $\mathcal{C}=\mathcal{C}^{oo}$ a polar family of closed sets, $\Gamma,\Lambda$  a closed or open cone $$\mathcal{D}'_\Gamma(U,\mathcal{C};E)=\{u\in \mathcal{D}'(U,\mathcal{C};E)\ | \ WF(u)\subset \Gamma\}.$$

As in \cite{BrouderDabrowski} we want to put strong variants of H\"ormander topology on them. On $\mathcal{D}'_\Gamma(U;E)=\mathcal{D}'_\Gamma(U,\mathcal{F};E)$, for $\Gamma$ closed, we consider the locally convex topology $\mathcal{I}_H$ defined by the following family of seminorms : For any $B\subset \mathcal{D}(U;E^*)$ a strongly bounded subset we consider the seminorm :
$$P_B(u)=\sup_{f\in B}|\langle u,f\rangle|$$
and as usual for $i\in I, k\in \N^*,f\in  \mathcal{D}(U), \text{supp} f\subset U_i, d\varphi_i^*(\text{supp} (f\circ\varphi_i^{-1})\times V)\cap \Gamma=\emptyset$:
$$P_{i,k,f,V}(u)=\sup_{\xi\in V}(1+|\xi|)^k\sum_{l=1}^e|\mathcal{F}[[(f u)\circ \varphi_i^{-1}]_l](\xi)|.$$
where $\mathcal{F}$ is the Fourier transform on $\mathcal{E}'(\R^d).$ There won't be any confusion with the set of all closed sets $\mathcal{F}$ as used in the previous section.




Then we define on $\mathcal{D}'_\Gamma(U,\mathcal{C};E)$ the inductive  topology $\mathcal{I}_{H,i}$ as the inductive limit topology~:
$$(\mathcal{D}'_\Gamma(U,\mathcal{C}),\mathcal{I}_{H,i})=\text{ind}\lim_{C\in \mathcal{C}} (\mathcal{D}_{\Gamma}'(U,\{F\in \mathcal{F}, F\subset C\};E),\mathcal{I}_{H}).$$
The topology $\mathcal{I}_{H}$ is the induced topology from the one above. As in subsection \ref{support}.2, we have a projective variant\footnote{ Note that, contrary to the definition before lemma \ref{DualitySupport}, we don't index the set by $\mathcal{C}^o$ as the projective limit to simplify notations and since we by now only use polar families for which $\mathcal{C}^o,\mathcal{C}$ are determined by one another. This is thus more important to write with the same notation the same space, and put only topological notations in the topology $\mathcal{I}_{H,p}$.}~:
$$(\mathcal{D}'_\Gamma(U,\mathcal{C};E),\mathcal{I}_{H,i})\to (\mathcal{D}'_\Gamma(U,\mathcal{C};E),\mathcal{I}_{H,p}):= \text{proj} \lim_{O\in \mathcal{C}^o}(\mathcal{D}'_\Gamma(U,\{O\}^o;E),\mathcal{I}_{H,i}).$$
These topologies will be generalized in the next section.

In order to study a class of spaces with good stability properties for (topological and bornological) duality 
we are led to introduce a refinement 
of the wave front set, we call 
it dual wave front set for reasons that will become clear later.

\begin{definition}\label{DWFDef}
For $u\in\mathcal{D}'(U;E)$, we call dual wave front set of $u$ and write $DWF(u)\subset\dot{T}^*U$ the set such that  $(x,\xi)\in (DWF(u))^c$ if and only if for any $k\in \N,i\in I :x\in U_i$, there is an open neighborhood of $x$, $V_k\subset U_i$ (a defining open of a chart of $U$)  and an open conic neighborhood $\Gamma_k$ of $(d(\varphi_i)^{-1})_{\varphi_i(x)})^*(\xi)$ such that : for any $f\in \mathcal{D}(U)$ with $\text{supp}(f)\subset V_k$ we have 
$P_{i,k,f,\Gamma_k}(u)<\infty.$
\end{definition}
Said otherwise $DWF(u)=\cup_{s>0}, WF_s(u)$ is the union of $H^s$-wave front sets (see e.g. \cite[p 8]{Delort}).

\begin{remark}\label{DWF}
Fix a bump function $\chi\in\mathcal{D}(\R^d)$ such that $\chi(x)=1$ if $|x|\leq 1/2$, $\text{supp}(\chi)\subset B(0,1)$ and $\chi:\R^n\to [0,1]$. Let $\chi_{x,n}(y)=\chi(n(y-x))$ with $\text{supp}(\chi_{x,n})\subset B(x,1/n)$. Then, using the usual tricks of wave front set theory to control the wave front set of a product with a smooth function (see e.g. \cite[(8.1.3)' p 253]{Hormander}) it is easy to see $(x,\xi)\in (DWF(u))^c$ if and only if for any $k\in \N,i\in I: x\in U_i$, there is an integer $n_k$ and an open conic neighborhood $\Gamma_k$ of $(d(\varphi_i)^{-1})_{\varphi_i(x)})^*(\xi)$ such that  we have 
$P_{i,k,\chi_{\varphi_i(x),n_k}\circ \varphi_i,\Gamma_k}(u)<\infty.$ Note also that for $f\in\mathcal{D}(U)$, $DWF(fu)\subset DWF(u).$

Of course, if $(DWF_k(u))^c$ is the set satisfying the above property for a fixed $k$ instead of all $k$, it is the union of some $V_k\times [(d\varphi_i)_x]^*\Gamma_k$ which are open, so that $(DWF(u))^c$ is a $G_\delta$ set, i.e. a countable intersection of the above open sets. In this way $DWF(u)$ is an $F_\sigma$ cone and moreover,  $$DWF(u)\subset WF(u)= \overline{DWF(u)} .$$ The first inclusion is obvious and the second using compactness to cover a product in a open cone at each order. One deduces equality from the characterization of the closure.

\end{remark}

We now introduce the various spaces of distributions generalizing H\"ormander ones we will be interested in.
Thus let $\Lambda$ be a cone, 
 $\gamma\subset \Lambda\subset \overline{\gamma}$ 
 another cone 
 and  $\mathcal{C}=\mathcal{C}^{oo}$ an enlargeable (thus openly generated) polar family of closed sets in $U$, we call :
$$\mathcal{D}'_{\gamma,\Lambda}(U,\mathcal{C};E)=\{u\in \mathcal{D}'(U;E)\ | \ WF(u)\subset \Lambda , DWF(u)\subset \gamma, \text{supp}(u)\in \mathcal{C}\}.$$ 

Note that the assumption $\gamma\subset \Lambda\subset \overline{\gamma}$ is not restrictive by replacing $\gamma$ by $\gamma\cap\Lambda$ 
and $\Lambda$ by $\Lambda\cap \overline{\gamma}$. 
As an example for $\Gamma$ a closed cone we recover the spaces defined before as  $\mathcal{D}'_\Gamma(U;E)=\mathcal{D}'_{\Gamma,\Gamma}(U,\mathcal{F};E)$, and $\mathcal{E}'_\Lambda(U;E)=\mathcal{D}'_{\Lambda,\Lambda}(U,\mathcal{K};E),\mathcal{D}'_\Gamma(U,\mathcal{C};E)=\mathcal{D}'_{\Gamma,\Gamma}(U,\mathcal{C};E)$.

\section{Various topologies on generalized H\"ormander spaces of distributions}

For $\Lambda$ a cone intersection of a closed and an open cone, we start by writing it as an increasing union of closed cones. This will reduce some uncountable inductive limits to countable ones.  {What matters is not so much that one can define a topology with a countable inductive limit (as we could do taking any $F_\sigma$-cone) but rather that the canonical uncountable inductive limit will be equivalent to a countable one.

 More generally, most of the definitions of topologies in this subsection will be motivated by either having lots of inductive limits enabling some duality computation, or having only countable inductive limits and huge projective limits, as required by preservation of nuclearity, of course.

With all our projective and inductive limits, we will reduce general controlling cones for $DWF$, $WF$ at the end to a pair $\Lambda,\overline{\Lambda}$ with $\Lambda$ open, and in a intermediate step to $\lambda_1, \overline{\lambda_1}\cap \lambda_2$ with both $\lambda_i$ open. This explains why we need to study the intersection of an open and closed cone.}

\begin{lemma}\label{openclosed}
Any  $\Lambda\subset T^*U-\{0\}$  which is the intersection of an open cone and a closed cone can be written as a union of a sequence of closed cone $\Lambda_n$ with compact first projection $p_1(\Lambda_n)$ and such that $\Lambda_{n}$ is contained  in the interior of $\Lambda_{n+1}$ and any closed cone $\Gamma\subset \Lambda$ (closed for the induced topology) with its projection $p_1(\Gamma)\subset U$ compact is included in one $\Lambda_n.$ 
\end{lemma}
\begin{proof}
Let $\Lambda$ be
the intersection of a closed cone $\gamma$ and an open cone $\lambda$.

To reduce  to the open case, consider $\Lambda_n=\gamma\cap \lambda_n$, where $\lambda_n$ built for $\lambda$. If $\Gamma$ is a closed cone in $\gamma\cap \lambda$ with compact projection, it is in $\lambda$, thus $\lambda_n$, and $\Gamma$, thus in $\Lambda_n$ as expected. 
In the open case, the proof is detailed at the beginning of section 3.1 in \cite{BrouderDabrowski}.
\end{proof}

We start by defining a topology $\mathcal{I}_H$ on $\mathcal{D}'_{\gamma,\overline{\gamma}}(U,\mathcal{F};E).$
In this case, we consider the family $\Delta(\gamma)=\{\delta=(\Gamma_n)_{n\in \N}, \ \Gamma_n\ \text{closed cone}, \ \Gamma_n\subset\Gamma_{n+1}, \cup_{n\in \N}\Gamma_n{\subset}\gamma\}$.

For $\delta=(\Gamma_n)\in \Delta(\gamma)$, we consider the space $$\mathcal{D}'_{\overline{\gamma}}(U,\mathcal{F}:\delta;E)=\{u\in \mathcal{D}'(U;E)\ | \ 
\forall n ,DWF_n(u)\subset \Gamma_n \}.$$ 
We put on this space the natural topology with the seminorm of $\mathcal{D}'_{\overline{\gamma}}(U;E)$ and the seminorms $P_{i,k,f,V_k}(u)$ for any $i\in I,f\in \mathcal{D}(U)$ with $\text{supp} f\subset U_i, d\varphi_i^*(\text{supp} (f\circ\varphi_i^{-1})\times V_k)\cap \Gamma_k=\emptyset.$

Finally, we thus call inductive topology of $\mathcal{D}'_{\gamma,\overline{\gamma}}(U,\mathcal{F};E)$ the following inductive limit topology with $\Delta(\gamma)$ of course ordered by simultaneous inclusions of all the cones ($\delta=(\Gamma_n)\subset \delta'=\Gamma'_n$ iff $\forall n \Gamma_n\subset \Gamma'_n$) : $$(\mathcal{D}'_{\gamma,\overline{\gamma}}(U,\mathcal{F};E),\mathcal{I}_{i})=\underrightarrow{\lim}_{\delta\in \Delta(\gamma)}(\mathcal{D}'_{\overline{\gamma}}(U,\mathcal{F}:\delta;E),\mathcal{I}_H).$$

All the continuous inclusions required to define the above inductive limit are obvious. Note that if $\gamma=\overline{\gamma}$ then $\delta_{\gamma}=(\gamma)_{n\in \N}\in \Delta(\gamma)$ and the inductive limit is trivial : $\mathcal{D}'_{\overline{\gamma},\overline{\gamma}}(U,\mathcal{C};E)=\mathcal{D}'_{\overline{\gamma}}(U,\mathcal{C}:\delta_\gamma;E)$
and the topology is the same as defined before.
We also define a projective variant, involving a projective limit over $\Lambda\supset \gamma$ open cone ordered by 
inclusion :
$$(\mathcal{D}'_{\gamma,\overline{\gamma}}(U,\mathcal{F};E),\mathcal{I}_{p})=\underleftarrow{\lim}_{\Lambda\supset \gamma}(\mathcal{D}'_{\Lambda,\overline{\Lambda}}(U,\mathcal{F};E),\mathcal{I}_{i}).$$

It will thus be crucial for us that $(\mathcal{D}'_{\Lambda,\overline{\Lambda}}(U,\mathcal{F};E),\mathcal{I}_{i})$ is a basic object we will interpret as a completion and for which we will know plenty of functional analytic properties we will then extend to the non-open case.


Towards the general support case, we fix a point $x_0\in U$  once and for all, we define for a cone $\Gamma$ and $R> 0$, its $R$-compact variant, $$\Gamma_{(R)}:=\Gamma_{(R,1)}\cup T_{(R,2)},\Gamma_{(R,1)}= \{(x,\xi)\in \dot{T}^*U, D(x_0,x)\leq R, D(x,U^c)\geq 1/R, (x,\xi)\in \Gamma\},$$ with $T_{(R,2)} =\{(x,\xi)\in \dot{T}^*U, D(x_0,x)\geq R \ \text{or}\  D(x,U^c)\leq 1/R\}.$
So that we have $\Gamma=\cap_{n>0}\Gamma_{(n)}$, and $\Gamma_{(n)}$ is a closed (resp. $F_\sigma$) cone, as soon as $\Gamma$ is. $\Gamma_{(R,1)}$ is also the intersection of a closed and an open cone as soon as $\Gamma$ is.

We consider a mixed topology 
on $\mathcal{D}'_{\gamma,\Lambda}(U,\mathcal{F};E)$ a projective limit over $n>0$ of inductive limits over $\Pi$ closed subcone (directed by inclusion) for $\alpha=p$ (resp. $\alpha=i$
):$$(\mathcal{D}'_{\gamma,\Lambda}(U,\mathcal{F};E),\mathcal{I}_{\alpha m})=\underleftarrow{\lim}_{n\to \infty }\left(\underrightarrow{\lim}_{\Pi\subset \Lambda}(\mathcal{D}'_{(\gamma\cap \Pi)_{(n)},(\overline{\gamma\cap \Pi})_{(n)}}(U,\mathcal{F};E),\mathcal{I}_{\alpha})\right).$$

 The continuity of inclusions is detailed at the beginning of the next section. The projective limit could have been avoided, but in this way, the second inductive limit is equivalent to a countable one when $\Lambda$ is the intersection of a closed and an open cone using our lemma \ref{openclosed}.
Indeed, for any $\Pi\subset \Lambda$ if we take an exhaustion $\Lambda_m$ of $\Lambda_{(n,1)}$ as in lemma \ref{openclosed}, $\Lambda_{(n,1)}$ being also the intersection of a closed and an open cone as noted after its definition, and $\Pi_{(n,1)}\subset\Lambda_{(n,1)}$ which is a closed subcone with compact first projection, we have $\Pi_{(n,1)}\subset \Lambda_m$ and thus $\Pi_{(n)}\subset (\Lambda_m)_{(n)}$ and likewise $(\gamma\cap \Pi)_{(n)}\subset (\gamma\cap\Lambda_m)_{(n)}$ $(\overline{\gamma\cap \Pi})_{(n)}\subset (\overline{\gamma\cap \Lambda_m})_{(n)}$ so that one gets  (using the continuous inclusions proved at the beginning of the next section) a map $$\left(\underrightarrow{\lim}_{\Pi\subset \Lambda}(\mathcal{D}'_{(\gamma\cap \Pi)_{(n)},(\overline{\gamma\cap \Pi})_{(n)}}(U,\mathcal{F};E),\mathcal{I}_{\alpha}))\right)\to \left(\underrightarrow{\lim}_{m\to \infty}(\mathcal{D}'_{(\gamma\cap \Lambda_m)_{(n)},(\overline{\gamma\cap \Lambda_m})_{(n)}}(U,\mathcal{F};E),\mathcal{I}_{\alpha}))\right).$$ Since the converse map giving an isomorphism is obvious using $\Lambda_m\subset \Lambda_{(n,1)}\subset \Lambda$ is a closed cone, we conclude to the stated equivalence to a countable inductive limit.

 We thus define 
an inductive 
topology of $\mathcal{D}'_{\gamma,\Lambda}(U,\mathcal{F};E)$ using only an inductive limit on closed cones avoiding the use of the projective limit above, (for $\alpha\in\{p,i\}$)$$(\mathcal{D}'_{\gamma,\Lambda}(U,\mathcal{F};E),\mathcal{I}_{\alpha i})=\underrightarrow{\lim}_{\Pi\subset \Lambda}\left(\mathcal{D}'_{\gamma\cap \Pi,\overline{ \gamma\cap\Pi}}(U,\mathcal{F};E),\mathcal{I}_{\alpha}\right),$$
 Note again that when $\Lambda=\overline{\gamma}$ the inductive limits disappear and $\mathcal{I}_{ii}=\mathcal{I}_{i}$, and in the case $\Lambda=\overline{\gamma}=\gamma=\Gamma$, we even have on $\mathcal{D}'_{\Gamma,{\Gamma}}(U,\mathcal{F};E)$, $\mathcal{I}_{im}=\mathcal{I}_{ii}=\mathcal{I}_{H}$ since there is no more inductive limits and in this way the added intermediate projective limit is equivalent to the original one. 
 Since the mixed variant above didn't really achieved the goal of having only countable inductive limits in all cases, we exploit those cases to define 
 projective 
topologies of $\mathcal{D}'_{\gamma,\Lambda}(U,\mathcal{F};E)$ using  projective limit on open cones or pairs of open cones, so for $\alpha\in\{p,i\}$~:
$$(\mathcal{D}'_{\gamma,\Lambda}(U,\mathcal{F};E),\mathcal{I}_{\alpha P})=\underleftarrow{\lim}_{\lambda\supset \Lambda}\left(\mathcal{D}'_{\gamma,\overline{\gamma}\cap \lambda}(U,\mathcal{F};E),\mathcal{I}_{\alpha m}\right),$$
$$(\mathcal{D}'_{\gamma,\Lambda}(U,\mathcal{F};E),\mathcal{I}_{\alpha p})=\underleftarrow{\lim}_{ \lambda_1\supset \gamma}\underleftarrow{\lim}_{\lambda_2\supset \lambda_1\cup \Lambda}\left(\mathcal{D}'_{\lambda_1,\overline{\lambda_1}\cap \lambda_2}(U,\mathcal{F};E),\mathcal{I}_{\alpha m}\right).$$
Note that the projective limit over open cones indeed forces the wave front set to be in the intersection of $\overline{\gamma}\cap \lambda$, i.e. $\overline{\gamma}\cap \Lambda=\Lambda$ as expected. Moreover, $\overline{\gamma}\cap \lambda$ is the intersection of a closed and an open cone so that the inductive limits involved in the definition of $\left(\mathcal{D}'_{\gamma,\overline{\gamma}\cap \lambda}(U,\mathcal{F};E),\mathcal{I}_{\alpha m}\right)$
 are equivalent to countable inductive limits as explained above. If $\Lambda$ is the intersection of a closed cone $\Gamma$ and an open cone $\lambda$, so that $\overline{\gamma}\subset \Gamma$ we can even assume $\overline{\gamma}= \Gamma$ (since $\Lambda\subset \overline{\gamma}$)  so that   the projective limit is terminated at $\lambda$ since if $\Lambda\subset \lambda'\subset \lambda$ $\Lambda=\overline{\gamma}\cap \Lambda\subset \overline{\gamma}\cap \lambda'\subset \overline{\gamma}\cap \lambda\subset \Gamma\cap \lambda=\Lambda$ and we have equality so that $(\mathcal{D}'_{\gamma,\Lambda}(U,\mathcal{F};E),\mathcal{I}_{\alpha P})=(\mathcal{D}'_{\gamma,\Lambda}(U,\mathcal{F};E),\mathcal{I}_{\alpha m}).$ 
Note also that \begin{equation}\label{EqualConep}(\mathcal{D}'_{\gamma,\gamma}(U,\mathcal{F};E),\mathcal{I}_{\alpha p})=\underleftarrow{\lim}_{ \lambda\supset \gamma}\left(\mathcal{D}'_{\lambda,\lambda}(U,\mathcal{F};E),\mathcal{I}_{\alpha m}\right),\end{equation} which is similar to the dual wave front set analogue used to define $\mathcal{I}_{p}.$
 
At this stage, we also define a variant motivated by the explicit computation of various bornologifications. As will be intuitive to any expert, we require a polynomial boundedness of Fourier transforms and take appropriately inductive limits on the degree (and then require this only uniformly on some compact and take a projective limit in this compact).
$$(\mathcal{D}'_{\gamma,\Lambda}(U,\mathcal{F};E),\mathcal{I}_{\alpha b})=\underleftarrow{\lim}_{K\in \mathcal{K}}\underrightarrow{\lim}_{k\in\N}(\mathcal{D}_{\gamma,\Lambda}^{\prime(K,k)}(U,\mathcal{F};E),\mathcal{I}_{\alpha m}^{+}),$$
where $(\mathcal{D}_{\gamma,\Lambda}^{\prime(K,k)}(U,\mathcal{F};E),\mathcal{I}_{\alpha m}^{+}):=\{u\in (\mathcal{D}_{\gamma,\Lambda}(U,\mathcal{F};E), \forall f\in \mathcal{D}(U),\text{supp}(f)\subset K\cap U_i, \sup_{\xi}(1+|\xi|)^{-k}\sum_{l=1}^e|\mathcal{F}[[(f u)\circ \varphi_i^{-1}]_l](\xi)|<\infty\}$ is given the topology induced by the seminorms of $((\mathcal{D}_{\gamma,\Lambda}(U,\mathcal{F};E),\mathcal{I}_{\alpha m})$
and moreover $P_{i,-k,f,\R^{d}}(u)=\sup_{\xi}(1+|\xi|)^{-k}\sum_{l=1}^e|\mathcal{F}[[(f u)\circ \varphi_i^{-1}]_l](\xi)|.$
 
 
 
Finally, we introduce topologies for spaces with support conditions with $\alpha=ii$ or $\alpha=im$ or $\alpha=iP$ or $\alpha=ip$ or $\alpha=ib$ 
or $\alpha=pi$ or $\alpha=pm$ or $\alpha=pP$ or $\alpha=pp$ or $\alpha=pb$~:  
 \begin{align*}(\mathcal{D}'_{\gamma,\Lambda}(U,\mathcal{C};E),\mathcal{I}_{\alpha i})&=\underrightarrow{\lim}_{C\in \mathcal{C}} (\mathcal{D}'_{\gamma,\Lambda}(U,\{F\in \mathcal{F}, F\subset C\};E),\mathcal{I}_{\alpha})
 \\&\to (\mathcal{D}'_{\gamma,\Lambda}(U,\mathcal{C};E),\mathcal{I}_{\alpha p}):=\underleftarrow{\lim}_{O\in \mathcal{C}^o}(\mathcal{D}'_{\gamma,\Lambda}(U,\{O\}^o;E),\mathcal{I}_{\alpha i}).\end{align*}

The first one is for $\mathcal{C}$ any family of closed sets stable by subset, the second one a polar family of closed sets.

To summarize, we introduced first inductive limits for dual wave front sets and a projective variant 
 then inductive/projective limits on wave front sets (and a mixed variant not really achieving the goal of only involving countable inductive limits, it is thus only an intermediate to define the projective one, really only involving countable inductive limits), and then finally, on support which was easier but still requires inductive/projective variants. At each step, there are two main schemes, one of which containing only huge inductive limits, and the second one smaller (often countable) inductive limits associated to projective limits.

Part of the work later will be to identify various of these topologies to benefit from stability of functional analytic properties by inductive or projective limits.
The most important ones, when they are  not identified, will be emphasized in the main results later.

\section{Continuous injections, densities and known results}

We start by detailing the various continuous injections among the previous spaces. Note we already used them to define the inductive and projective limits appearing before. We then explain a density result obtained as usual by truncation and convolution. They will play a crucial role to reduce most of the basic properties of the above spaces to the best known cases for which we then explain what is already known mainly from \cite{BrouderDabrowski}.  Note that, for two topologies, $\mathcal{I},\mathcal{J}$ we write as usual $\mathcal{I}\supset\mathcal{J}$ is the first $\mathcal{I}$ is finer than the second $\mathcal{J}$. We also denote $L(E,F)$ the space of continuous linear maps between two locally convex spaces $E,F$.

\begin{lemma}\label{contuinuousdenseInj}
Let $\gamma\subset \gamma',$ 
$\Lambda\subset \Lambda'$ 
cones with $\gamma\subset \Lambda\subset\overline{\gamma},\gamma'\subset \Lambda'\subset\overline{\gamma'}$ and $\mathcal{C}\subset\mathcal{C}'$ 
 families of closed sets on $U$ (polar or stable by subset depending of what is needed in the definition) then  $(\mathcal{D}'_{\gamma,\Lambda}(U,\mathcal{C};E),\mathcal{I}_{\alpha})$   embeds continuously in $(\mathcal{D}'_{\gamma',\Lambda'}(U,\mathcal{C}';E),\mathcal{I}_{\alpha})$ (for $\alpha=\alpha_1\alpha_2\alpha_3, \alpha_1,\alpha_3\in\{
i,p\},\alpha_2\in\{i,m,P,p,b\}$). 
 Moreover we always have for any possible $\alpha,\beta,$ $$
 \mathcal{I}_{i\alpha}\supset \mathcal{I}_{p\alpha}, \ \  \mathcal{I}_{\alpha i\beta}\cap \mathcal{I}_{\alpha b\beta}\supset \mathcal{I}_{\alpha m\beta}\supset\mathcal{I}_{\alpha P\beta}\supset \mathcal{I}_{\alpha p\beta}\ \textrm{and} \ \ \mathcal{I}_{\alpha i}\supset \mathcal{I}_{\alpha p}.$$
 
 Moreover, for any $\chi\in \mathcal{D}(U)$, if we define $L_\chi:u\mapsto \chi u$ then for any $\alpha$ $$ L_\chi\in L((\mathcal{D}'_{\gamma,\Lambda}(U,\mathcal{C};E),\mathcal{I}_{\alpha p}), (\mathcal{D}'_{\gamma,\Lambda}(U,\{F, F\subset \text{supp}(\chi)\};E),\mathcal{I}_{\alpha})).$$
 
Likewise, fix for each $i\in I$, $\rho_i\in \mathcal{D}(\R^d)$ positive   supported in $B(0,\delta_i)\subset \R^d$ equal to 1 on $B(0,\delta_i/2)$ with  $\text{supp}(f_i\circ \varphi_i^{-1})+B(0,2\delta_i)\subset \varphi_i(U_i)$. Thus for $\rho\in \mathcal{D}(\R^d)$ one can define $S_\rho(u)=S(\rho,u)=\sum_{i\in I}([\rho_i\rho]*[(uf_i)\circ \varphi_i^{-1}])\circ\varphi_i $, then  $$S_\rho\in L((\mathcal{D}'_{\gamma,\Lambda}(U,\mathcal{K};E),\mathcal{I}_{\alpha}), \mathcal{D}(U;E)).$$

 Finally, for any  polar family of closed sets $\mathcal{C}$, $i\in I$ and $g\in\mathcal{D}(U_i)$ the application $M_g:u\mapsto ([(gu)\circ\varphi_i^{-1}]_l)$ gives a continuous map for any possible $\alpha$ $$M_g:(\mathcal{D}'_{\gamma,\Lambda}(U,\mathcal{C};E),\mathcal{I}_{\alpha})\to (\mathcal{D}'_{(d\varphi_i^{-1})^*\gamma_{|U_i},(d\varphi_i^{-1})^*\Lambda_{|U_i}}(\varphi_i(U_i),\mathcal{K};\R^e),\mathcal{I}_{\alpha})$$ and conversely, $N_i:(v_l)\mapsto (v_l)\circ\varphi_i$ gives a continuous map, for any $\alpha$ in the trivial case $|I|=1$ or when $\gamma$ is open or closed, or for $\alpha=p\beta$ and any $\beta$ otherwise,  $$N_i:(\mathcal{D}'_{(d\varphi_i^{-1})^*\gamma_{|U_i},(d\varphi_i^{-1})^*\Lambda_{|U_i}}(\varphi_i(U_i),\mathcal{K};\R^e),\mathcal{I}_{\alpha})\to(\mathcal{D}'_{\gamma,\Lambda}(U,\mathcal{C};E),\mathcal{I}_{\alpha})$$
\end{lemma}
We will write later $M_i=M_{f_i}.$
\begin{proof}
The comparisons of topologies come from general results on inductive and projective limits 
sometimes using the continuous injections built below and stated in the first part.

We start with the case $\Lambda=\overline{\gamma},\Lambda'=\overline{\gamma'}$. 
It suffices to note $\Delta(\gamma)\subset \Delta(\gamma')$.
Thus, for $\delta\in \Delta(\gamma)$  the continuity of the injection  $\mathcal{D}'_{\overline{\gamma}}(U,\mathcal{F}:\delta;E)
\to \mathcal{D}'_{\gamma',\Lambda'}(U,\mathcal{F};E)$ is obvious, one concludes by the definition of an inductive limit in the case $\mathcal{C}=\mathcal{C}'=\mathcal{F}$. Then applying inductive or projective limits 
on both sides, the more general case $\mathcal{C}\neq\mathcal{C}'$ is easy and also the case for $ \alpha\neq iii.$

For the general case note that  for a closed cone $\Pi\subset\Lambda\subset \Lambda'$, the previous result gives continuity of $\mathcal{D}'_{\gamma\cap \Pi,\overline{\gamma\cap \Pi}}(U,\mathcal{F};E)\to \mathcal{D}'_{\gamma'\cap \Pi,\overline{\gamma'\cap\Pi}}(U,\mathcal{F};E),$ for $\mathcal{I}_{\alpha}, \alpha\in\{i,p\}.$ The definitions of inductive  limits conclude again in the case $\mathcal{C}=\mathcal{C}'=\mathcal{F}$.
For the case with mixed normal topologies note that  for a closed cone $\Pi\subset\Lambda\subset \Lambda'$, the previous result also gives continuity of $\mathcal{D}'_{(\gamma\cap \Pi)_{(n)},(\overline{\gamma\cap\Pi})_{(n)}}(U,\mathcal{F};E)\to \mathcal{D}'_{(\gamma'\cap \Pi)_{(n)},(\overline{\gamma'\cap\Pi})_{(n)}}(U,\mathcal{F};E).$ Again, the definitions of inductive and projective limits conclude in the case $\mathcal{C}=\mathcal{C}'=\mathcal{F}$. Finally, the support inductive/projective limits are dealt with similarly. The cases with $\alpha_2\in\{b,p,P\}$ are dealt with similarly starting from $\alpha_2=m.$

For the continuity of $L_\chi$ since the support condition is known, arguing similarly with inductive and projective limits, we are mostly reduced to prove for $\delta=(\Gamma_n)\in \Delta(\gamma)$ the continuity of $L_\chi :(\mathcal{D}'_{\overline{\gamma}}(U,\mathcal{F}:\delta;E),\mathcal{I}_H)\to (\mathcal{D}'_{\overline{\gamma}}(U,\mathcal{F}:\delta;E),\mathcal{I}_H).$
For $B\in \mathcal{D}(U;E^*)$ strongly bounded, it is well known that so is $\chi B$ thus $P_B(L_\chi(u))\leq P_{\chi B}(u)$ and there is of course no problem with these seminorms. Moreover $P_{i,k,f,V}(L_\chi(u))\leq P_{i,k,\chi f,V}(u)$ and since $\text{supp}(\chi f)\subset \text{supp}( f)$, the second is a seminorm when the first is. Formally, we also have to treat the supplementary seminorms for $\mathcal{I}_{\alpha b\beta}$ similarly.

For $S_\rho$, it suffices to note using the previous continuous injection that $S_\rho\in L((\mathcal{E}'(U;E)), \mathcal{D}(U;E))$ and by the inductive limit topology on the source it suffices for $\chi$ fixed to get  $S_\rho L_\chi\in L((\mathcal{D}'(U;E)), \mathcal{D}(U;E))$, but since this is a finite sum, it suffices to note $u\mapsto ([\rho_i\rho]*[(u\chi f_i)\circ \varphi_i^{-1}])\circ\varphi_i \in L((\mathcal{D}'(U_i;E)), \mathcal{D}(U_i;E))$ i.e. via the isomorphism induced by $\varphi_i$,  $u\mapsto ([\rho_i\rho]*[u[(\chi f_i)\circ \varphi_i^{-1}]]) \in L((\mathcal{D}'(\varphi_i(U_i);E)), \mathcal{D}(\varphi_i(U_i);E))$ which comes for instance from \cite[lemma 7.2]{BrouderDangHelein}.

For the definition of $M_g$, the case of the topology $\mathcal{I}_{iii}$ in the case $\mathcal{C}=\mathcal{F}$, $\Lambda=\overline{\gamma}$ one gets a map at each level of the inductive limit $(d\varphi_i^{-1})^*(\Gamma_k)_{|U_i})\in \Delta(d\varphi_i^{-1})^*\gamma_{|U_i})$ are indeed closed cones for $\delta=(\Gamma_k)\in \Delta(\gamma)$, since the continuity for the strong topology is well-known (definition of distributions on manifolds) and the extra seminorms of the target composed with $M_g$, e.g. $P_{k,f,V_k}\circ M_g$ for any $f\in \mathcal{D}(\varphi_i(U_i))$ $ (\text{supp} (f)\times V_k)\cap (d\varphi_i^{-1})^*(\Gamma_k)_{|U_i}=\emptyset,$  are explicitly seminorms of the source  $P_{k,f,V_k}\circ M_g= P_{i,k,g(f\circ\varphi_i),V_k}$. Note that varying $g$ enables to get any seminorm of the source in this way.

For the definition of $N_i$ in the same case, the only problematic seminorms are those with index $j\neq i,$ (when they exist, i.e. when $U\not\subset U_i$ is not trivial). 
Using the remark above, fix $f\in\mathcal{D}(U_j)$, $g\in \mathcal{D}(U_i)$ equal to $1$ on some compact $\varphi_i^{-1}(K)$ on which we first restrict the support of elements of the domain of $N_i$. It suffices to prove $M_{gf}\circ N_i$ is continuous $(\mathcal{D}'_{(d\varphi_i^{-1})^*\gamma_{|U_i},(d\varphi_i^{-1})^*\Lambda_{|U_i}}(\varphi_i(U_i\cap U_j),\mathcal{K};\R^e), \mathcal{I}_{iii})\to (\mathcal{D}'_{(d\varphi_j^{-1})^*\gamma_{|U_i},(d\varphi_i^{-1})^*\Lambda_{|U_i}}(\varphi_j(U_i\cap U_j),\mathcal{K};\R^e), \mathcal{I}_{iii})$. But $M_{gf}\circ N_i(v)= (gf (v \circ \varphi_i))\circ \varphi_j^{-1}= (\varphi_i\circ\varphi_j^{-1})^*((gf)\circ \varphi_i^{-1} v)$. The pullback being relevant on  $\varphi_i(U_j\cap U_i)\supset \text{supp}([(gf)\circ \varphi_i^{-1}] v)$. Since this only involves the open set case that does not use this argument (trivial open case above and below), we can use proposition \ref{pullback} below in this case (proved first in the open case without the manifold case) when $\gamma$ open in which case $\mathcal{I}_{iii}=\mathcal{I}_{ppp},$ case for which this proposition applies (note it is mostly based on \cite{BrouderDangHelein} and some functional analysis). When $\Gamma$ closed, one can apply \cite{BrouderDangHelein} directly with $\mathcal{I}_{iii}$ giving the same conclusion. Note also that for the supplementary seminorms $P_{i,-k,f,\R^d}$ when there is an index $b$, one needs to use some supplementary basic argument for $N_i$ (the case $M_i$ is similar). As above, it boils done to checking a bound for pull-back by smooth isomorphisms $\varphi=\varphi_i\circ\varphi_j^{-1}$, i.e. to bound $P_{-k-d-1,f,\R^d}(u\circ \varphi)$ by $P_{-k,f,\R^d}(u)$. As usual, writing $\mathcal{F}(f  u\circ \varphi)(\xi)=\frac{1}{(2\pi)^d}\int d\eta T_{f,\varphi}(\eta,\xi)\mathcal{F}(u)(\eta)$ with $T_{f,\varphi}(\eta,\xi)=\int dx f(x)e^{i\langle\varphi(x),\eta\rangle- \langle(x),\xi\rangle}$ thus the bound is obtained in using the standard bound on integral operators  (see e.g. \cite[lemma 4]{stromaier} role of $\xi,\eta$ exchanged with respect to the usual application):\begin{align*}|T_{f,\varphi}(\eta,\xi)|&\leq C||f e^{i\langle .,\xi\rangle}||_{C^{k+d+1}}(1+|\eta|)^{-(k+d+1)}\sup_{x\in \text{supp}(f)} ||(d\varphi(x))^{-1}||^{2(k+d+1)} \\&\leq C'(1+|\xi|)^{k+d+1}(1+|\eta|)^{-(k+d+1)},\end{align*}
using  that $\varphi$ is a diffeomorphism.

 Since $(d\varphi_i^{-1})^*:\dot{T}^*U_i\to \dot{T}^*(\varphi_i(U_i))$ is an isomorphism of cotangent bundles, it preserves open cones and closed cones thus the inductive/projective limits on wave front sets or dual wave front sets are dealt with by universal properties. (Note that for $N_i$ because of the restriction $\gamma$ open above, the topology index  needs to start with a $p$, except of course when $\gamma$ open or closed or when $|I|=1.$ Moreover, for inductive limits on closed cones, ones needs to restrict first to a compact support 
 for which the inductive limit restricts to $\Pi$ well inside $U_i$ and thus having an inverse image which is also a closed cone in $M$ and not only $U_i$). The inductive/projective limits on support are dealt with easily since $\mathcal{C}$ is polar and thus contains $\mathcal{K}.$
\end{proof}

The next lemma gives for now a density result but will be also the source of all approximation properties we will use in the second paper of this series \cite{Dab14b}.
\begin{lemma}\label{normality}
Let $\gamma\subset \Lambda\subset\overline{\gamma}$ 
cones  and $\mathcal{C}$ 
 polar enlargeable family of closed sets on $U$  then there exists a sequence $(L_n)_{n\in \N},$ $L_n\in L((\mathcal{D}'_{\gamma,\Lambda}(U,\mathcal{C};E),\mathcal{I}_{iii}), \mathcal{D}(U;E))$   
 such that for any $u\in  \mathcal{D}'_{\gamma,\Lambda}(U,\mathcal{C};E)$, $L_n(u)$  has a subsequence Mackey-converging to $u$ in $(\mathcal{D}'_{\gamma,\Lambda}(U,\mathcal{C};E),\mathcal{I}_{iii})$. Moreover, for any $B$ in the inductive limit bornology associated to $\mathcal{I}_{iii}$, there is a subsequence $L_{k_n(B)}$ such that $L_{k(B)}(B)=\{L_{k_n(B)}(u), u\in B, n\in \N\}$ is bounded in this same bornology and  $L_{k_n(B)}(u)$ converges to $u$ for any $u\in B.$  Finally, if $\mathcal{C}$ is countably generated, the subsequence can be taken independent of $B$.

Especially $\mathcal{D}(U;E)\hookrightarrow(\mathcal{D}'_{\gamma,\Lambda}(U,\mathcal{C};E),\mathcal{I}_{\alpha})\hookrightarrow\mathcal{D}'(U;E)$ is a strictly normal system for any $\alpha$ considered above.
\end{lemma}
\begin{proof}
This is an easy variant of \cite[Th 8.2.3 p. 262]{Hormander} inspired by \cite{Schwarz}. We fix the map $S_\rho$ of the previous lemma. Take $\chi_n\in \mathcal{D}(U)$, 
$\chi_n$ equal to $1$ on each compact set for $n$ large enough. Take also $\bar{\rho}_n\in \mathcal{D}(R^d)$ of integral $1$, non-negative,  with $\text{supp}(\bar{\rho}_n)\subset B(0,1/n) $ and let $$\mathcal{N}=\{ (n,m)\in \N^2 :\text{supp}(\bar{\rho}_n)\subset B(0,\min\{\delta_i/3: i\in I_m\})\}$$ with $I_m=\{i\in I: \overline{U_i}\cap \text{supp}(\chi_m)\neq \emptyset\}$ (note that for an $m$, there exists $N_m$, such that for $n\geq N_m, (n,m)\in \mathcal{N}$). Recall $\delta_i$ is defined in lemma \ref{contuinuousdenseInj} and is chosen such that the $\rho_i$ (made for compatibility with the chart system) in the definition of $L_\rho$, are such that $\rho_i$ are equal to $1$ on $B(0,\delta_i/2)$ explaining the $\delta_i/3$ appearing above. Thus we have $\rho_i\bar{\rho}_n=\bar{\rho}_n$, for $(n,m)\in \mathcal{N}, i\in I_m$ and we don't have to consider $\rho_i$ anymore.

Fixing a bijection $b:\N\to \mathcal{N}$ we then let $L_n=S_{\bar{\rho}_{b_1(n)}}\circ L_{\chi_{b_2(n)}}.$ The continuity follows from the previous lemma.
Thus take a bounded set as stated, i.e. it is uniformly supported in $C\in\mathcal{C}$, $\forall u\in B,\ WF(u)\subset \overline{\gamma\cap \Pi}$, for $\Pi\subset \Lambda$ closed and there is $\delta=(\Gamma_n)\in \Delta(\gamma\cap \Pi)$ such that $B$ is bounded in $(\mathcal{D}'_{\overline{\gamma}}(U,\mathcal{F}:\delta;E).$ By enlargeability let $\epsilon $ with $C_\epsilon\in \mathcal{C}$, define $\mathcal{N}_\epsilon=\{ (n,m)\in \mathcal{N}: \forall i\in I_m,([\text{supp}(\bar{\rho}_n)]+\varphi_i(C\cap \text{supp}(f_i)))\subset \varphi_i(\overline{U_i}\cap (\overline{U_i}\cap C)_{u\epsilon_i})\}$ which contains again for each $m$ eventually all $n$ since $I_m$  finite 
  $ \varphi_i(C\cap \text{supp}(f_i)))$ are compact contained in the open
$\varphi_i({U_i}\cap Int((\overline{U_i}\cap C)_{u\epsilon_i/2})$ thus at positive distance of its complement. In this way for $(n,m)\in \mathcal{N}_\epsilon$, $S_{n}\circ L_{m}(B)$ is uniformly supported in $C_\epsilon.$ Take $\{k_n,n\in \N\}\subset b^{-1}(\mathcal{N}_\epsilon)$ increasing with $b_i(k_n)$ also increasing. We want to show that this $k(B)$ is satisfactory, i.e. first $L_{k}(B)$ bounded in $\mathcal{D}'_{\overline{\gamma}}(U,\mathcal{F}:\delta;E).$ First to check it is bounded in $\mathcal{D}'(U,E)$ it suffices to check it is weakly bounded, i.e. $L_{k_n}^*(v)$ bounded in $\mathcal{D}(U,E^*)$ for any $v$ in this space,this is well known since our conditions guaranty $L_{k_n}^*(v)\to v$ by standard results on truncation and convolution. This also shows $L_{k_n}(u)\to u$ weakly in $\mathcal{D}'(U,E).$ Since the choice of $k_n$ depends only of $C$, if $\mathcal{C}$ is countably generated, one can extract diagonally, to work for an exhausting sequences of $C$ and obtain the same $k_n$ working for any $B$, i.e. which is for $n$ large enough with $b(k_n)\in \mathcal{N}_\epsilon$ above.

It remains to bound $P_{i,k,f,V_k}(L_{k_n}(u))\leq\sum_{l=1}^eP_{k,f\circ\varphi_i^{-1},V_k}([\rho_i\overline{\rho}_{b_1(k_n)}]*[(\chi_{b_2(k_n)}uf_i)\circ \varphi_i^{-1}]_l)$ for any $i\in I,f\in \mathcal{D}(U)$ with $\text{supp} f\subset U_i, d\varphi_i^*(\text{supp} (f\circ\varphi_i^{-1})\times V_k)\cap \Gamma_k=\emptyset.$ We can now argue as in \cite[Th 8.2.3 p. 262]{Hormander}, take $g\in \mathcal{D}(U_i)$ equal to 1 on a neighborhood of $\text{supp} f$ and $W_k\supset Int(W_k)\supset V_k$ with  $d\varphi_i^*(\text{supp} (g\circ\varphi_i^{-1})\times W_k)\cap \Gamma_k=\emptyset.$ Then for $n$ large enough, uniform in $u\in B$ since depending only on the support condition, $(f\circ\varphi_i^{-1})([\rho_i\rho_{b_1(k_n)}]*[(\chi_{b_2(k_n)}uf_i)\circ \varphi_i^{-1}]_l)=(f\circ\varphi_i^{-1})([\rho_{b_1(k_n)}]*[(guf_i)\circ \varphi_i^{-1}]_l)$ and use his estimate (8.1.3)' to conclude to the expected boundedness uniform in $u\in B$ of $P_{i,k,f,V_k}(L_{k_n}(u))$, using only the bound on $P_{i,k,g,W_k}(u)$ and the uniform polynomial boundedness of $u\in B$. Then the finite number of small $n$ are treated by continuity of each $L_{k_n}$ and thus boundedness of each $L_{k_n}(B).$

Finally, if $\delta'=(\Gamma_k'), \Gamma_{k}'=\Gamma_{k+1}\supset \Gamma_k$, 
 the convergence  $L_{k_n}(u)\to u$ in $\mathcal{D}'_{\overline{\gamma}}(U,\mathcal{F}:\delta';E),$ follows from the strong convergence in $\mathcal{D}'(U,E)$ implying uniform convergence on compacts of Fourier transforms, the boundedness in $\mathcal{D}'_{\overline{\gamma}}(U,\mathcal{F}:\delta;E),$ and the argument showing equivalence of condition (ii) and (ii)' in Hormander's definition 8.2.2 p262 \cite{Homander}.


 The Mackey convergence (after another extraction) follows as in \cite[lemma 22]{BrouderDabrowski} because the topology of $\mathcal{D}'_{\overline{\gamma}}(U,\mathcal{F}:\delta';E),$ is based on the strong seminorm of $\mathcal{D}'(U;E)$ and countably many supplementary seminorms.

The condition on strict normality is easy since the continuous injections have been obtained in the previous lemma and we just checked  the only missing strict density since $\mathcal{I}_{iii}$ is the strongest topology (or rather strongest bornology to include also $\mathcal{I}_{ibi}$ and using the established Mackey density).
\end{proof}

Here is a  summary of the main results in
\cite{BrouderDabrowski}, slightly generalized to vector bundles and general support conditions. For quasi-LB-spaces (which are strictly webbed) we refer to \cite{ValdiviaQuasiLB}.

Recall the Arens topology on a dual is the topology of uniform convergence on balanced compact convex subsets.
\begin{theorem}\label{DualityBD}
If $\Gamma$ is a closed cone and $\Lambda=-\Gamma^c$, $\mathcal{C}$ an enlargeable polar family of closed cones in $U$, then $(\mathcal{D}'_{\Gamma,\Gamma}(U,\mathcal{C};E),\mathcal{I}_{\alpha})$ and $(\mathcal{D}'_{\Lambda,\Lambda}(U,(\mathscr{O}_\mathcal{C})^o;E^*),\mathcal{I}_{\alpha})$ are duals of one another (for $\alpha=iii,imi,iPi,iPp,imp,iip,$). $\mathcal{D}'_{\Gamma,\Gamma}(U,\mathcal{C};E)$ is Hausdorff and  semi-reflexive  in all cases, and complete, nuclear for the projective topology $\mathcal{I}_{iip}=\mathcal{I}_{iPp}=\mathcal{I}_{ipp}=\mathcal{I}_{imp}$, and for the inductive $\mathcal{I}_{iii}=\mathcal{I}_{ipi}=\mathcal{I}_{imi}=\mathcal{I}_{iPi}$ too if $\mathcal{C}$ is countably generated (see better results at the end of lemma \ref{bornologyIdentities} below), 
$\mathcal{D}'_{\Lambda,\Lambda}(U,(\mathscr{O}_\mathcal{C})^o;E^*)$ is Hausdorff, nuclear, ultrabornological and barrelled for its topologies $\mathcal{I}_{iii}=\mathcal{I}_{iip}=\mathcal{I}_{iPp}=\mathcal{I}_{iPi}=\mathcal{I}_{ipp}=\mathcal{I}_{ipi}=\mathcal{I}_{imi}=\mathcal{I}_{imp}$ which coincide with strong, Arens and Mackey topologies from its dual with topology $\mathcal{I}_{iii}$. 
Finally, if either $\mathcal{C}$ or $(\mathscr{O}_\mathcal{C})^o$ is countably generated, then $(\mathcal{D}'_{\Lambda,\Lambda}(U,(\mathscr{O}_\mathcal{C})^o;E^*),\mathcal{I}_{iii})$ is a quasi-LB-space 
 and so is $(\mathcal{D}'_{\Gamma,\Gamma}(U,\mathcal{C};E),\mathcal{I}_{\alpha})$ with $\alpha=iii$ in the first case and $\alpha=iip$ in the second case.
\end{theorem}

\begin{proof}We already explained in the previous section why $\mathcal{I}_{imp}=\mathcal{I}_{iPp},\mathcal{I}_{imi}=\mathcal{I}_{iPi}$ in this case (when the condition on wave front is intersection of closed and open, especially when it is either closed or open). From equation \eqref{EqualConep}, one also deduces 
$\mathcal{I}_{imp}=\mathcal{I}_{ipp},\mathcal{I}_{imi}=\mathcal{I}_{ipi}$ in the open cone case from a terminated projective limit and in the closed cone case from the identification of $(\mathcal{D}'_{\Lambda,\Lambda}(U,\mathcal{F};E),\mathcal{I}_{im})$ with a strong=Mackey dual as proved later here below and from the identification of $(\mathcal{D}'_{\Gamma,\Gamma}(U,\mathcal{F};E),\mathcal{I}_{ii})$ as a bornological dual topology of the bornological inductive limit $\underrightarrow{\lim}_{\Pi\subset -\Gamma^c}\mathcal{D}'_{ \Pi,\Pi}(U, \mathcal{K};E^*)$  in theorem \ref{duals} below or in \cite{BrouderDabrowski}. Indeed, the inductive limit over $\Pi\subset -\Gamma^c$ converts into the expected projective limit on $\lambda=\lambda_1=\lambda_2=-\Pi^c\supset \Gamma$, namely   $\underleftarrow{\lim}_{\lambda\supset \Gamma}[\left(\mathcal{D}'_{ \lambda,\lambda}(U, \mathcal{F};E\right),\mathcal{I}_{im}]$.
We thus don't consider the topologies with middle index $p,P$ in this proof anymore.

 When $|I|=e=1,$ $\mathcal{C}=\mathcal{F}$ and thus $(\mathscr{O}_\mathcal{C})^o=\mathcal{K}$ and for the normal topologies the result is exactly \cite[Prop 7, 9,27, 28]{BrouderDabrowski} and all topologies $\mathcal{I}_{iii},\mathcal{I}_{iip},\mathcal{I}_{imi},\mathcal{I}_{imp}$  coincide with those called normal/inductive topology in \cite{BrouderDabrowski}. We explain the agreement of those topologies even when we don't have $|I|=e=1.$ Indeed, we already treated the case $\mathcal{F}$ with closed cones in the remarks above, and first note that for inductive limits over compacts $\mathcal{I}_{\alpha i}=\mathcal{I}_{\alpha p}$ for $\alpha$ having two letters, since the projective limit is then terminated because any $C\in\mathcal{F}$ is included in the complement of a compact $K^c$, for which $\{Int(K^c)\}^o=\mathcal{K}$. We then note that the topologies $\mathcal{I}_{ii},\mathcal{I}_{im}$ coincide on $\mathcal{D}'_{\gamma,\Lambda}(U,\{F\in \mathcal{F}, F\subset C\};E^*)$ for $C$ compact, since we can compare the induced topologies of the projective limit and the projective limit of induced topologies, and this one is terminated at some $n$ after which $C\subset K_n=\{x\in U: d(x,U^c)\leq 1/n, |x|\leq n\}$. After this same $n$ or a larger one, fix an $f\in\mathcal{D}(U)$ equal to $1$ on $C$ and equal to $0$ on $K_n^c$, multiplication by $f$ define the identity map  on $\mathcal{D}'_{\gamma,\Lambda}(U,\{F\in \mathcal{F}, F\subset C\};E^*)$ and moreover defines a continuous map $\underrightarrow{\lim}_{\Pi\subset \Lambda}(\mathcal{D}'_{(\gamma\cap \Pi)_{(n)},(\overline{\gamma\cap\Pi})_{(n)}}(U,\mathcal{F};E^*),\mathcal{I}_{i}))\to \underrightarrow{\lim}_{\Pi\subset \Lambda}(\mathcal{D}'_{(\gamma\cap \Pi),(\overline{\gamma\cap\Pi})}(U,\mathcal{F};E^*),\mathcal{I}_{i}))$ (See in lemma \ref{contuinuousdenseInj} above). Thus one gets  by restriction the expected missing  map $(\mathcal{D}'_{\gamma,\Lambda}(U,\{F\in \mathcal{F}, F\subset C\};E^*),\mathcal{I}_{im})\to (\mathcal{D}'_{\gamma,\Lambda}(U,\{F\in \mathcal{F}, F\subset C\};E^*),\mathcal{I}_{ii})$. This concludes to the equality $\mathcal{I}_{ii}=\mathcal{I}_{im}$ on $\mathcal{D}'_{\gamma,\Lambda}(U,\{F\in \mathcal{F}, F\subset C\};E^*)$ for $C$ compact and thus to all the identifications of topologies stated above in this case (the identification of $\mathcal{I}_{iii}$ with the topology of \cite{BrouderDabrowski} is obvious by inductive limit generalities). 

For the manifold/vector bundle case (i.e. when we don't have $|I|=e=1$) we still have to prove the expected properties for $\mathcal{I}_{iii}$. 
 The duality pairing is built in adding one more partition of unity for the covering at the beginning. Thus fix $f_i'\in\mathcal{D}(U_i)$  equal to one on a neighborhood of $\text{supp}{f_i}$. From this and $\sum f_i=1$ the duality pairing needs to be for $v\in (\mathcal{D}'_{\Lambda,\Lambda}(U,(\mathscr{O}_\mathcal{C})^o;E^*),\mathcal{I}_{iii}),u\in(\mathcal{D}'_{\Gamma,\Gamma}(U,\mathcal{C};E),\mathcal{I}_{iii})$:
 
 \begin{equation}\label{dualitymanifold}\langle u,v\rangle=\sum_{i\in I}\langle f_iu,f_i'v \rangle=\sum_{i\in I}\langle M_{f_i}(u),M_{f_i'}(v) \rangle.\end{equation}
 From the continuity of $M_{f_i}$ in lemma \ref{contuinuousdenseInj}, finiteness of the sum by compactness of one support, one deduces the continuity of the pairing from the one of the last bracket, which is the one from \cite{BrouderDabrowski}. (We will still refer in this way about this slightly new bracket.) One also deduces from that the conclusion of \cite[lemma 10]{BrouderDabrowski} namely $\sup_{u\in B}|\langle u, v\rangle|$ is continuous in $v$ and we will use this later. To compute the duals take $T\in (\mathcal{D}'_{\Lambda,\Lambda}(U,(\mathscr{O}_\mathcal{C})^o;E^*)'$ (resp. $T\in (\mathcal{D}'_{\Lambda,\Lambda}(U,(\mathscr{O}_\mathcal{C})^o;E^*)'$) giving a distribution $v\in\mathcal{D}'(U,\mathcal{C};E)$ (resp $u\in \mathcal{D}'(U,\mathcal{C};E^*)$) and it suffices to check their wave front set. But $T\circ L_{f_i}\circ N_i$ is a continuous linear form on the corresponding space on $\varphi_i(U_i)$ and applying the results of \cite[Prop 7, 9]{BrouderDabrowski}, they correspond to distributions $u_i$ in the domain of $N_i$ in such a way that $vf_i=N_i(u_i)$ (resp. $uf_i=N_i(u_i)$) so that summing over $i$, one gets the conclusion on the wave front sets of $u,v.$ 
 
 We will explain functional analytic properties of this case with the general support case below.
 

Next we deal with the case $\mathcal{C}=\mathcal{K}$ and thus $(\mathscr{O}_\mathcal{C})^o=\mathcal{F}$. One can note that in the present case $\mathcal{I}_{\alpha i}=\mathcal{I}_{\alpha p}$, for any $\alpha=ii,im$ for obvious reasons (essentially terminated inductive or projective limits). 
A continuous linear functional on $(\mathcal{D}'_{\Lambda,\Lambda}(U,\mathcal{F};E^*),\mathcal{I}_{ii})$ induces, by definition of inductive limit, a  continuous linear functional on $(\mathcal{D}'_{\Pi,\Pi}(U,\mathcal{F};E^*),\mathcal{I}_{i})$, so that $(\mathcal{D}'_{\Lambda,\Lambda}(U,\mathcal{F};E^*),\mathcal{I}_{i})'$ embeds into $\mathcal{D}'_{-\Pi^c,-\Pi^c}(U,\mathcal{K};E)$ for any $\Pi\subset \Lambda$ closed cone. Since the intersection of these spaces is $\mathcal{D}'_{\Gamma,\Gamma}(U,\mathcal{K};E),$ one gets one direction of the computation of the dual. 

Likewise, if $T$ is a continuous linear form on $(\mathcal{D}'_{\Gamma,\Gamma}(U,\mathcal{K};E),\mathcal{I}_{iii})$ (the stronger and thus only relevant topology) then for $f\in\mathcal{D}(U)$, $T_f=T\circ L_f:u\mapsto T(uf)$ is a continuous linear map on $\mathcal{D}'_{\Gamma,\Gamma}(U,\mathcal{F};E),$
(See again lemma \ref{contuinuousdenseInj} above)
Thus $T_f\in \mathcal{D}'_{\Lambda,\Lambda}(U,\mathcal{K};E^*).$ Since $T$ defines by restriction a distribution, taking $f\in \mathcal{D}(U)$ constant equal to $1$ on large open sets, one checks by locality of the wave front set that $T\in \mathcal{D}'_{\Lambda,\Lambda}(U,\mathcal{F};E^*).$ The density of smooth functions then implies that the restriction map gives indeed an injection from $(\mathcal{D}'_{\Gamma,\Gamma}(U,\mathcal{K};E),\mathcal{I}_{iii}))'\to \mathcal{D}'_{\Lambda,\Lambda}(U,\mathcal{F};E^*).$ 

But the duality pairing between these spaces is obvious to build, for $v\in \mathcal{D}'_{\Lambda,\Lambda}(U,\mathcal{F};E^*)$ and $u\in \mathcal{D}'_{\Gamma,\Gamma}(U,\mathcal{K};E)$ with support in $K$, take $f\in \mathcal{D}(U)$ constant equal to $1$ on a neighborhood of $K$, $v\in \mathcal{D}'_{\Lambda,\Lambda}(U,\mathcal{F};E^*)$, then define $\langle u,v\rangle =\langle u,vf\rangle,$ with the bracket of the duality built in \cite{BrouderDabrowski}. Since the identity $(\mathcal{D}'_{\Gamma,\Gamma}(U,\mathcal{K};E),\mathcal{I}_{\alpha})\to (\mathcal{D}'_{\Gamma,\Gamma}(U,\mathcal{F};E),\mathcal{I}_{H})$ is continuous (for any $\alpha=iii,iip,imi,imp$ which coincides in this case), and the former is induced by an inductive limit on compact support (for each element of which on can fix $f$) the continuity in $u$ follows for the inductive topology by the results in \cite{BrouderDabrowski}. 
The continuity in $v$ follows from the one of $L_f:(\mathcal{D}'_{\Lambda,\Lambda}(U,\mathcal{F};E^*),\mathcal{I}_{\alpha})\to (\mathcal{D}'_{\Lambda,\Lambda}(U,\mathcal{K};E^*),\mathcal{I}_{\alpha}),$ (for any $\alpha=iii,iip,imi,imp$ which coincides at the target space in this case). 

\medskip

 We are now ready to deal with the general case of the duality for the inductive topologies. From the continuous maps $(\mathcal{D}'_{\Gamma,\Gamma}(U,\mathcal{K};E),\mathcal{I}_{\alpha})\to (\mathcal{D}'_{\Gamma,\Gamma}(U,\mathcal{C};E),\mathcal{I}_{\alpha}),(\mathcal{E}(U,(\mathscr{O}_\mathcal{C})^{oo};E)\to (\mathcal{D}'_{\Gamma,\Gamma}(U,\mathcal{C};E),\mathcal{I}_{\alpha})$
with dense ranges (for any $\alpha=iii,iip,imi,imp$ which coincides at the source space in this case), a continuous linear functional on  $(\mathcal{D}'_{\Gamma,\Gamma}(U,\mathcal{C};E),\mathcal{I}_{\alpha}) $ restricts to a distribution 
  in $(\mathcal{D}'_{\Lambda,\Lambda}(U,(\mathscr{O}_\mathcal{C})^o;E^*)$ (from the case above and from proposition \ref{FApropertiesSupport}). 
 
  Similarly, from the injections $$(\mathcal{D}'_{\Lambda,\Lambda}(U,\mathcal{K};E^*),\mathcal{I}_{\alpha})\to (\mathcal{D}'_{\Lambda,\Lambda}(U,(\mathscr{O}_\mathcal{C})^o;E^*),\mathcal{I}_{\alpha}),(\mathcal{E}(U,\mathcal{C}^{o};E^*)\to (\mathcal{D}'_{\Lambda,\Lambda}(U,(\mathscr{O}_\mathcal{C})^o;E^*),\mathcal{I}_{\alpha})$$ (for any $\alpha=iii,iip,imi,imp$ which coincides at the source space in this case too) one deduces $(\mathcal{D}'_{\Lambda,\Lambda}(U,(\mathscr{O}_\mathcal{C})^o;E^*),\mathcal{I}_{\alpha})'\hookrightarrow\mathcal{D}'_{\Gamma,\Gamma}(U,\mathcal{C};E).$ 

It remains, to finish the computation of duals, to build the duality pairing in the general case 
 and prove continuity with respect to the topologies $\mathcal{I}_{\alpha p}$ ($\alpha=ii,im$).  Take $v\in \mathcal{D}'_{\Lambda,\Lambda}(U,(\mathscr{O}_\mathcal{C})^o;E^*))$ with support in $D\in (\mathscr{O}_\mathcal{C})^o$ and $u\in \mathcal{D}'_{\Gamma,\Gamma}(U,\mathcal{C};E)$ with support in $C\in \mathcal{C}$. Since by enlargeability $C_\epsilon\in \mathcal{C},D_\epsilon\in (\mathscr{O}_\mathcal{C})^o$ for some $\epsilon_i<1$ we have by proposition \ref{EnlargeableStable} $\overline{C_\epsilon\cap D_\epsilon}$ is compact and especially in this way $Int(D_\epsilon)\in  \mathcal{C}^o,Int(C_\epsilon)\in (\mathscr{O}_\mathcal{C})^{oo}.$ 

Thus fix $f\in \mathcal{D}(U)$  equal to $1$ on $\overline{C_\epsilon\cap D_\epsilon}\supset \text{supp}(u)\cap\text{supp}(v).$ By H\"ormander's Thm. $uv$ exists and $uv=uvf$ and thus we define the duality pairing $\langle u,v\rangle=\langle u,vf\rangle,$ and the above pairing is the one from  \cite{BrouderDabrowski}. To prove continuity in $u$ (resp. $v$), it suffices to find an $O\in \mathcal{C}^o$ (resp. $(\mathscr{O}_\mathcal{C})^{oo}$)  say $O=Int(D_\epsilon)$ (resp. $O=Int(C_\epsilon)$) such that the bracket is continuous in the inductive topology $\mathcal{I}_{\alpha i}$ with $\mathcal{C}$ replaced by $\{O\}^o\ni C$ (resp. $(\mathscr{O}_\mathcal{C})^{o}$ replaced by $\{O\}^o\ni D$). Thus it suffices, by definition of the inductive limit as soon as the support is fixed by $C$ (resp. $D$)
, to prove continuity for $(\mathcal{D}'_{\Gamma,\Gamma}(U,\mathcal{F};E),\mathcal{I}_{\alpha })$ (resp. $(\mathcal{D}'_{\Lambda,\Lambda}(U,\mathcal{F};E^*),\mathcal{I}_{\alpha })$) which is done in \cite{BrouderDabrowski} or in the paragraph above.\footnote{\label{TechnicalNote}Note that if we replace $\mathcal{C}$ by $\{F\in \mathcal{F}, F\subset C\}$ and $(\mathscr{O}_\mathcal{C})^o$ by $Int(C_\epsilon)^o$ for $\epsilon_i >0$ we still have $\overline{C\cap D}\subset \overline{Int(C_\epsilon)\cap D}$ compact. Thus we can still build the duality map for the topologies  $\mathcal{I}_{\alpha i}$.}



To prove $\mathcal{D}'_{\Gamma,\Gamma}(U,\mathcal{C};E)$ semi-reflexive, we compute its strong dual as $(\mathcal{D}'_{\Lambda,\Lambda}(U,(\mathscr{O}_\mathcal{C})^o;E^*),\mathcal{I}_{imi})$ (from this we also deduce $\mathcal{I}_{imi}=\mathcal{I}_{iii}$ and equality with strong and Arens topologies from Arens-Mackey Thm and computations of duals). First note that the bounded sets are the same in both projective and inductive topologies $\mathcal{I}_{iii}=\mathcal{I}_{imi},\mathcal{I}_{iip}=\mathcal{I}_{imp}$. Indeed, since the inductive topology is stronger, take a bounded set $B$ for the projective topology $\mathcal{I}_{iip}$. Thus for any $O\in \mathcal{C}^o$ it is bounded in $(\mathcal{D}'_{\Gamma,\Gamma}(U,\{O\}^o;E),\mathcal{I}_{iii}),$ but this (countable) inductive limit is strict with each embedding having closed image, thus it is regular, and thus $B$ is bounded in $\mathcal{D}'_{\Gamma}(U;E)$ with support uniformly in some $C_{O,B}\in\{O\}^o.$ Thus the support of the bounded set is in $C_B=\cap_{O\in \mathcal{C}^o}C_{O,B}\in \mathcal{C}^{oo}=\mathcal{C}$ by definition of the polar, and $B$ is thus bounded in $(\mathcal{D}'_{\Gamma,\Gamma}(U,\mathcal{C};E),\mathcal{I}_{iii})$ as expected. It thus suffices to compute the strong dual of the (regular) inductive limit $(\mathcal{D}'_{\Gamma,\Gamma}(U,\mathcal{C};E),\mathcal{I}_{iii})$. Reasoning as in proposition \ref{FApropertiesSupport} using Arens-Mackey theorem and the previous computation of duals, we only prove that the identity map from $(\mathcal{D}'_{\Lambda,\Lambda}(U,(\mathscr{O}_\mathcal{C})^o;E^*),\mathcal{I}_{imi})$ to $(\mathcal{D}'_{\Gamma}(U,\mathcal{C};E),\mathcal{I}_{iii})_b'$ is continuous. By definition of the inductive limit, it suffices to prove continuity on $(\mathcal{D}'_{\Lambda,\Lambda}(U,\{F\in\mathcal{F}, F\subset D\};E^*),\mathcal{I}_{im})$ for $D\in (\mathscr{O}_\mathcal{C})^o$. But a bounded set $B$ in $(\mathcal{D}'_{\Gamma,\Gamma}(U,\mathcal{C};E),\mathcal{I}_{iii})$ is in some $(\mathcal{D}'_{\Gamma}(U,\{F\in\mathcal{F}, F\subset C\});E)$, $C\in \mathcal{C}$, thus we can fix $f$ as above to compute the duality pairing and we have to bound $\sup_{u\in B}|\langle u,vf\rangle|$, but by \cite[Lemma 10]{BrouderDabrowski} (and its bundle entension explained above) this is uniformly bounded by the seminorms of $vf$ in $\mathcal{D}'_{\Lambda,\Lambda}(U,\mathcal{K};E^*)$ and the continuity of $L_f:v\mapsto vf,$ from $(\mathcal{D}'_{\Lambda,\Lambda}(U,\mathcal{F};E^*),\mathcal{I}_{im})\to \mathcal{D}'_{\Lambda,\Lambda}(U,\mathcal{K};E^*)$ we already checked concludes.

Moreover, note that from polarity of $\mathcal{C},$  $(\mathcal{D}'_{\Gamma,\Gamma}(U,\mathcal{C};E),\mathcal{I}_{iii}))=\underrightarrow{\lim}_{C\in \mathcal{C}} (\mathcal{D}'_{\Gamma,\Gamma}(U,\{C\}^{oo};E),\mathcal{I}_{iii})$ which is a regular inductive limit and
 we can compute the same strong dual by the general result for regular inductive limits so that we obtain $$(\mathcal{D}'_{\Gamma,\Gamma}(U,\mathcal{C};E),\mathcal{I}_{iii}))'_b=\underleftarrow{\lim}_{C\in \mathcal{C}} ((\mathcal{D}'_{\Gamma,\Gamma}(U,\{C\}^{oo};E),\mathcal{I}_{iii}))_b'.$$
 
For $C\in  \mathcal{C}$ openly generated, let $O=Int(C)$ and $F=\overline{V}, V\in \mathscr{O}_{(\mathscr{O}_\mathcal{C})^o}$, thus since $O,V$ open: $\overline{O\cap F}=\overline{O\cap V}=\overline{\overline{O}\cap V}$ which is compact since $\overline{O}\subset C\in \mathcal{C}$ and $V\in (\mathscr{O}_{(\mathscr{O}_\mathcal{C})^o})^{oo}=\mathcal{C}^o$ by lemma \ref{opengenlemma}, i.e. we got $O=Int(C)\in  (\mathscr{C}_{\mathscr{O}_{(\mathscr{O}_\mathcal{C})^o}})^o=(\mathscr{O}_\mathcal{C})^{oo}$
 since $(\mathscr{O}_\mathcal{C})^{o}$ is openly generated.

 We now want to build a continuous canonical map, with $O'=Int(C_\epsilon),$ from $$(\mathcal{D}'_{\Lambda,\Lambda}(U,\{O'\}^o;E^*),\mathcal{I}_{imi})\to  (\mathcal{D}'_{\Gamma}(U,\{C\}^{oo};E),\mathcal{I}_{iii})_b'.$$ 
 
 Note that if $\text{supp}(u)\subset F\in \{C\}^{oo},\text{supp}(v)\subset V\in \{O'\}^o$, $u\in \mathcal{D}'_{\Gamma}(U,\{C\}^{oo};E), v\in \mathcal{D}'_{\Lambda,\Lambda}(U,\{O'\}^o;E^*)$, we have for $\epsilon$ small enough (by the proof of proposition \ref{EnlargeableStable})  $V_{\eta}\cap C\subset (\overline{V\cap C_{\delta(\eta)}})_{w\eta}\subset (\overline{V\cap Int(C_\epsilon)})_{w\eta}$ is compact so that $Int(V_{\eta})\in \{C\}^{o}$ and $V\cap F\subset \overline{Int(V_{\eta})\cap F}$ is compact. We thus define $f\in\mathcal{D}(U)$ equal to $1$ in a neighborhood of $\overline{Int(V_{\eta})\cap F}.$ The duality bracket is thus defined as above by $\langle u,v\rangle=\langle u,vf\rangle$ (even though we are not in an enlargeable case as before, the flexibility in terms of choice of $O'$, using enlargeability of $\mathcal{C}$ took care of the issue). To get the stated continuity, one needs to consider $B$ of support in $F$, bounded in $\mathcal{D}'_{\Gamma,\Gamma}(U;\mathcal{F};E)$ (from the regularity of the inductive limit), and thus the continuity in $v$ follows similarly as before. Gathering the result in a projective limit we got continuous maps :
\begin{align*}((\mathcal{D}'_{\Lambda,\Lambda}(U,(\mathscr{O}_\mathcal{C})^{o};E^*),\mathcal{I}_{imp}))&=\underleftarrow{\lim}_{O\in(\mathscr{O}_\mathcal{C})^{oo}} ((\mathcal{D}'_{\Lambda,\Lambda}(U,\{O\}^{o};E^*),\mathcal{I}_{imi}))\\&\to \underleftarrow{\lim}_{C\in \mathcal{C}} ((\mathcal{D}'_{\Gamma,\Gamma}(U,\{C\}^{oo};E),\mathcal{I}_{iii}))_b'=(\mathcal{D}'_{\Gamma,\Gamma}(U,\mathcal{C};E),\mathcal{I}_{iii}))'_b
\\&=((\mathcal{D}'_{\Lambda,\Lambda}(U,(\mathscr{O}_\mathcal{C})^{o};E^*),\mathcal{I}_{iii}))\to ((\mathcal{D}'_{\Lambda,\Lambda}(U,(\mathscr{O}_\mathcal{C})^{o};E^*),\mathcal{I}_{imp})),\end{align*}
and thus  $\mathcal{I}_{imp}=\mathcal{I}_{iii}=\mathcal{I}_{iip}=\mathcal{I}_{imi}$ as stated (we proved equality of the stronger and the weaker topologies among the list). 
Since distributions are separated by smooth compactly supported functions, which are in the computed duals, all the topology involved are clearly Hausdorff.

We already noted that $\mathcal{I}_{imp}$ is built with projective limits and inductive limits equivalent to countable ones in our too cases, all the stated nuclearity then follows from the one of $\mathcal{D}'_\Gamma(U;E)$ proved in \cite{BrouderDabrowski} by stability properties of nuclearity, see e.g. \cite{Treves}.

On $\mathcal{D}'_{\Gamma,\Gamma}(U,\mathcal{C};E)$,  $\mathcal{I}_{imp}$ is a projective limit of countable strict inductive limits of  closed subspaces of $\mathcal{D}'_{\Gamma}(U;E)$ with the complete nuclear (normal) topology studied in \cite[proposition 12]{BrouderDabrowski} (the bundle case is similar, the topology its still defined as a closed subspace of a countable product of complete nuclear spaces. Again the completeness can be checked independently by the bornological dual computation  in section 2). Since the inductive limit involved is strict countable and consist of complete spaces, it is complete by \cite[§ 19.6.(3)]{Kothe}, since it based on nuclear spaces, it is nuclear (see e.g. \cite{Treves}). Then completeness is kept by projective limits \cite[§ 19.10.(2)]{Kothe}, and nuclearity too. Completeness and nuclearity of 
   $\mathcal{I}_{imi}$ follows similarly when $\mathcal{C}$ is countably generated.

Finally, $(\mathcal{D}'_{\Lambda,\Lambda}(U,(\mathscr{O}_\mathcal{C})^{o};E^*),\mathcal{I}_{iii})$ is barrelled as the strong dual of a semi-reflexive space, ultrabornological, as the strong dual of a complete nuclear space.

For the results about quasi-LB-spaces, $(\mathcal{D}'_{\Gamma,\Gamma}(U,\mathcal{F};E),\mathcal{I}_{H})$ is known from \cite[Corollary 13]{BrouderDabrowski} to be a closed subspace of a countable product of Fr\'echet Schwartz spaces and dual Fr\'echet Schwartz spaces and since those spaces are quasi-LB-spaces, and quasi-LB-spaces are stable by countable products, countable sums, closed subspaces and separated quotients we have the result for this case (again in the bundle case the argument is identical). Then, in all the cases, we can reduce all inductive/projective limits of $\mathcal{I}_{imp}$ or $ \mathcal{I}_{imi}$ to countable ones and the same stability properties concludes. 
  
\end{proof}

\part{Duality among Generalized H\"ormander spaces of distributions and functional analytic consequences.}
\medskip
 \section{General bornological and topological duality results}
The goal of this section is to compute topological and bornological duals in the general case. We start by checking that our notion of dual wave front set is the right one to compute bornological duals.  
 
\begin{lemma}\label{bornImbedding}
Let 
$\gamma\subset \Lambda\subset \overline{\gamma}$ cones 
and  $\mathcal{C}=\mathcal{C}^{oo}$ an enlargeable polar family of closed sets in $U$. Let $\lambda=-\gamma^c$ then the bornological dual
$(\mathcal{D}'_{\gamma,\Lambda}(U,\mathcal{C};E),\mathcal{I}_{iii})^{\times}$ embeds in $\mathcal{D}'_{\lambda,\overline{\lambda}}(U,(\mathscr{O}_\mathcal{C})^o;E^*)$, and the topological dual  $(\mathcal{D}'_{\gamma,\Lambda}(U,\mathcal{C};E),\mathcal{I}_{iii})'$ embeds in $\mathcal{D}'_{\lambda,{\lambda}}(U,(\mathscr{O}_\mathcal{C})^o;E^*)$. 
\end{lemma} 
 Note there is no topological/bornological statement about the dual at this stage, the ``embedding" is merely an injective linear map. In general, when we don't specify a topology in a statement (except when there is a unique one we always consider in a section our in the whole paper as for $\mathcal{D}'(U,E)$), this means that the statement is purely algebraic. In the proofs though, we will sometimes not repeat the topology specified in the statements.
\begin{proof}
Take $(x,\xi)\in \gamma$, we will reduce everything to arguing with the closed cone $X=(x,\R_+^*\xi)$. Then 
$(\mathcal{D}'_{X,X}(U,\mathcal{C};E),\mathcal{I}_{iii})\hookrightarrow(\mathcal{D}'_{\gamma,\Lambda}(U,\mathcal{C};E),\mathcal{I}_{iii})$, thus a continuous linear form on $(\mathcal{D}'_{\gamma,\Lambda}(U,\mathcal{C};E),\mathcal{I}_{iii})$
 when restricted,  gives  a distribution $u$ in $\mathcal{D}'_{-X^c,-X^c}(U,(\mathscr{O}_\mathcal{C})^o;E^*)$ by theorem \ref{DualityBD} for any $(x,\xi)\in \gamma$, thus  $WF(u)\subset\{(x,\R_-^*\xi)\}^c$ (the condition on $DWF$ is a consequence) so that $WF(u)\subset \cap_{(x,\xi)\in\gamma}\{(x,\R_-^*\xi)\}^c
 \subset -\gamma^c.$ 

The injection property comes from (Mackey)-sequential density of $\mathcal{D}(U)$ by lemma \ref{normality}.

Since lemma \ref{contuinuousdenseInj} says $\mathcal{E}(U,(\mathscr{O}_\mathcal{C})^{oo};E)$ embeds continuously in $(\mathcal{D}'_{\gamma,\Lambda}(U,\mathcal{C};E),\mathcal{I}_{iii}),$ (from the identification of $(\mathcal{D}'_{\emptyset,\emptyset}(U,\mathcal{C};E),\mathcal{I}_{iii})\simeq \mathcal{E}(U,(\mathscr{O}_\mathcal{C})^{oo};E)$, see e.g. \cite[section 10.2]{BrouderDangHelein}) the restriction map gives a map $(\mathcal{D}'_{\gamma,\Lambda}(U,\mathcal{C};E),\mathcal{I}_{iii})^{\times}\to (\mathcal{E}(U,(\mathscr{O}_\mathcal{C})^{oo};E))^{\times}=\mathcal{D}'(U,(\mathscr{O}_\mathcal{C})^o;E^*)$ by Proposition \ref{FApropertiesSupport} since we computed there the dual and showed the space is bornological. This is an embedding by the sequential density of the inclusion $\mathcal{E}(U,(\mathscr{O}_\mathcal{C})^{oo};E)$ from lemma \ref{normality} since a bounded functional is Mackey-sequentially continuous. 

{We will reason similarly as before for the bornological duality statement. To check that the dual wave front set of a bounded linear functional is in $\lambda=-\gamma^c$ we will take a direction $X$ as above in $DWF$, get a bounded set $B\subset\mathcal{D}'_{-X,-X}(U,\mathcal{K})$ on which the linear form is not bounded so that $-X$ cannot be in $\gamma$ since otherwise the linear form would be bounded on $B$, implying $DWF\subset -\gamma^c$. We now explain this in more detail.}

\medskip
Thus let $u\in \mathcal{D}'(U,(\mathscr{O}_\mathcal{C})^o;E^*)$ the image of $T\in (\mathcal{D}'_{\gamma,\Lambda}(U,\mathcal{C};E), \mathcal{I}_{i
 ii})^{\times}$ by our previous embedding
 , it remains to check $DWF(u)\subset \lambda.$
Let $(x_0,\xi_0)\in DWF(u)$, say $x_0\in \text{supp}(f_i)$, generating a cone $X=\{(x_0,\mu \xi_0), \mu>0\}$. By definition and remark \ref{DWF}, there exists a $k$ (and a $i$ say the previous one) such that for all $n\geq N$ and open conical neighborhood of $(d(\varphi_i)^{-1})_{\varphi_i(x)})^*(\xi_0)=\xi'_0$, say $|\xi_0'|=1,$ for instance $\Gamma_n$ the cone generated by $B(\xi_0',1/n)$, $P_{i,k,\chi_{\varphi_i(x),n}\circ \varphi_i,\Gamma_n}(u)=\infty$, and especially, if we let $g_n=\chi_{\varphi_i(x),n}\circ \varphi_i$ and write $h\in\mathcal{D}(U_i)$ equal to $1$ on the support of all $g_n$ then there is $\xi_n\in \Gamma_n$, $l_n\in[1,e]$ such that  
$(1+|\xi_n|)^k|\mathcal{F}[(g_n u)\circ \varphi_i^{-1}]_{l_n}(\xi_n)|\geq n e^{n}C_n,$ with $C_n=\max(1,\sup_{\xi}|\mathcal{F}[\chi_{\varphi_i(x),n}](\xi)|(1+|\xi|)^k).$

Define for $u\in \mathcal{D}'(U;E^*),$ $\langle u, e_{\xi_n,l_n}\rangle=\mathcal{F}[(h u)\circ \varphi_i^{-1}]_{l_n}(\xi_n)$ so that $e_{\xi_n,l_n}\in (\mathcal{D}'(U;E^*))'= \mathcal{D}(U;E).$

Consider $B=\{e^{-n}\frac{g_n}{C_n} e_{\xi_n,l_n}(1+|\xi_n|)^k, n\in \N^*\}\subset\mathcal{D}(U;E) $ so that $\sup_{\{\varphi\in B\}}|\langle u,\varphi\rangle|=\infty$ 
and let us check $B$ is bounded in $(\mathcal{D}'_{-X,-X}(U,\mathcal{K};E),\mathcal{I}_{iii}
)$
. First we have uniformly compact support bounded by $
\text{supp}(h)\subset \overline{U_i}$, thus it suffices to check boundedness 
in $(\mathcal{D}'_{-X,-X}(U,\mathcal{F};E),\mathcal{I}_{iii})$.

First consider $C$ bounded in $\mathcal{D}(U;E^*),$ and $g\in C$, we bound as usual :
\begin{align*}|\langle& g_ne_{\xi_n,l_n}, g\rangle|=\mathcal{F}[(g_n g)\circ \varphi_i^{-1}]_{l_n}(\xi_n)=|\frac{1}{(2\pi)^d}\int_{\R^d}\mathcal{F}(g_n\circ \varphi_i^{-1})(\xi_n-\xi) [\mathcal{F}((hg)\circ \varphi_i^{-1})]_{l_n}(\xi)|\\&\leq \frac{C_n}{(2\pi)^d}\sup_{\eta\in \R^d}|[\mathcal{F}((hg)\circ \varphi_i^{-1})]_{l_n}(\eta)(1+|\eta|)^{d+1+k}|\times\int_{\R^d}(1+|\xi_n-\xi|)^{-k}(1+|\xi|)^{-k-d-1}
\\& \leq \frac{C_n}{(2\pi)^d}\sup_{\eta\in \R^d}|[\mathcal{F}((hg)\circ \varphi_i^{-1})]_{l_n}(\eta)(1+|\eta|)^{d+1+k}|\times(1+|\xi_n|)^{-k}\int_{\R^d}(1+|\xi|)^{-d-1}\\&=C_nc_d\ P_{i,k+1+d,h,\R^d}(g)(1+|\xi_n|)^{-k}.
\end{align*}
Thus $P_C(\varphi)$ is indeed bounded on $B.$ Then consider $g\in \mathcal{D}(U)$. 

If $x\not\in\text{supp}(g)\subset U_i$, take any closed cone $V$ (any cone satisfies $(\text{supp}(g)\times V)\cap (-X)=\emptyset$) but for any $n$ large enough $g_ng=0$ (as soon as $\varphi_n^{-1}(B(x,1/n))\cap \text{supp}(g)=\emptyset$) and thus $P_{i,l,g,V}(g_ne_{\xi_n,l_n})=0$ for $n$ large enough explaining why $P_{i,l,g,V}(\varphi)$ is bounded on $B$. A fortiori $P_{k,l,g,V}(g_ne_{\xi_n,l_n})=0$ for $k\neq i,$ if $\text{supp}(g)\subset U_k.$

If $x\in \text{supp}(g)\subset U_i$, we only have to consider $V$ with $-\xi_0\not\in V,$ so that for $n$ large enough, since $\xi_n/|\xi_n|\to \xi_0'$ we have $-\xi_n\not\in V,$ and there even exists $c\in(0,1]$ such that $|\xi_n+\xi|\geq c|\xi|$ for any $\xi\in V$, $n\geq N'$. 
Note that  for $x'=\varphi_i(x),$ 
$|\mathcal{F}[\chi_{x',n}](\xi)|=|\int_{\R^d}\chi(n(y-x'))e_{\xi}(y)dy|=|\frac{1}{n^d}\int_{\R^d}\chi(y)e_{\xi}(y/n+x')dy|=\frac{1}{n^d}|\mathcal{F}[\chi](\frac{\xi}{n})|$ so that 
$\sup_{\xi}|\mathcal{F}[\chi_{x',n}](\xi)|(1+|\xi|)^l\leq n^{l-d} D_l$ where $D_l=\sup_{\xi}|\mathcal{F}[\chi](\xi)|(1+|\xi|)^l.$
We can bound as in the previous boundedness for $P_B$ :
\begin{align*}P_{il,g,V}(f_ne_{\xi_n})&=\sup_{\xi\in V}(1+|\xi|)^l|\mathcal{F}[(g_ng)\circ\varphi_i^{-1}]_{l_n}(\xi+\xi_n)|\\&\leq n^{2k+l-d}D_{2k+l}c_d\ P_{i,2k+l+1+d,h,\R^d}(g)\ \ \sup_{\xi\in V}(1+|\xi|)^l(1+|\xi_n+\xi|)^{-l-2k}
\\&\leq n^{2k+l-d}D_{ 2k+l}c_d\ P_{i, 2k+l+1+d,h,\R^d}(g)c^{-l-k}(1+|\xi_n|)^{-k}\end{align*}
where in the last inequality, one uses $n\geq N$ to get $-\xi_n$ far from $V$ as specified above, and the usual $(1+|\xi_n+\xi|)^{-k}\leq (1+|\xi_n|)^{-k}(1+|\xi|)^{k}$. Multiplying by $e^{-n}/C_n(1+|\xi_n|)^{k}\leq e^{-n}(1+|\xi_n|)^{k}$, one gets the desired bound since the only remaining dependence in $n$ is $e^{-n}n^{2k+l-d}$ which is bounded. Thus we got boundedness of $M_{f_i}(B)$ and thus of $f_iB=N_i(M_{f_i}(B))\in(\mathcal{D}'_{-X,-X}(U,\mathcal{K};E),\mathcal{I}_{iii})$ from lemma \ref{contuinuousdenseInj} (note it is crucial we reduced to $-X$ closed to apply this with $\mathcal{I}_{iii}$). From the support property we also deduce the boundedness of $B$ (only finitely many terms may be changed by multiplication by $f_i$).

This shows $(x_0,\xi_0) \not\in(-\gamma)$ since, using $\mathcal{D}'_{-X,-X}(U,\mathcal{K};E)\subset \mathcal{D}'_{\gamma,\Lambda}(U,\mathcal{C};E)$ continuously, one otherwise contradicts the boundedness of $B$.
\end{proof}

 \medskip
 
 It now only remains to build the duality pairings to identify the duals (again algebraically at this stage).
 
\begin{theorem}\label{duals}
Let $\gamma\subset \Lambda\subset \overline{\gamma}$ cones  and  $\mathcal{C}=\mathcal{C}^{oo}$ an enlargeable polar family of closed sets in $U$. Let $\lambda=-\gamma^c$ then the bornological dual
$(\mathcal{D}'_{\gamma,\Lambda}(U,\mathcal{C};E),\mathcal{I}_{\alpha})^{\times}=\mathcal{D}'_{\lambda,\overline{\lambda}}(U,(\mathscr{O}_\mathcal{C})^o;E^*)$, and the topological dual  $(\mathcal{D}'_{\gamma,\Lambda}(U,\mathcal{C};E),\mathcal{I}_{\alpha})'=\mathcal{D}'_{\lambda,{\lambda}}(U,(\mathscr{O}_\mathcal{C})^o;E^*)$, for any $\alpha=\alpha_1\alpha_2\alpha_3$ $ \alpha_1,\alpha_3\in\{i,p\},\alpha_2\in\{i,m,P,p\}$. 
Moreover, $(\mathcal{D}'_{\gamma,\Lambda}(U,\mathcal{C};E),\mathcal{I}_{pbp})'=(\mathcal{D}'_{\gamma,\Lambda}(U,\mathcal{C};E),\mathcal{I}_{pbp})^\times=\mathcal{D}'_{\lambda,\overline{\lambda}}(U,(\mathscr{O}_\mathcal{C})^o;E^*).$

Finally if $\gamma$ closed, $(\mathcal{D}'_{\gamma,\gamma}(U,\mathcal{C};E), \mathcal{I}_{iii})$ is a bornological dual of the bornological inductive limit :
 $$(\mathcal{D}'_{\lambda,\lambda}(U,(\mathscr{O}_\mathcal{C})^{o};E^*))=\underrightarrow{\lim}_{C\in (\mathscr{O}_\mathcal{C})^{o}} \underrightarrow{\lim}_{\Pi\subset \lambda}\left(\mathcal{D}'_{ \Pi,\Pi}(U,\{F\in \mathcal{F}, F\subset C\};E^*\right)$$
with the bornology of the later space being the 
bornology induced by $(\mathcal{D}'_{ \Pi,\Pi}(U,\mathcal{F};E^*), \mathcal{I}_{H}).$ 
\end{theorem} 
 
\begin{proof}  

From our previous lemma, it only remains to build the duality pairing and prove the  continuity (resp. boundedness) with respect to $\mathcal{I}_{ppp}$ which is the weakest topology. 
Thus take $v\in \mathcal{D}'_{\lambda,\overline{\lambda}}(U,(\mathscr{O}_\mathcal{C})^o;E^*)$, $u\in \mathcal{D}'_{\gamma,\Lambda}(U,\mathcal{C};E)$, with $\text{supp}(v)\subset D\subset D_\epsilon\in (\mathscr{O}_\mathcal{C})^o, \text{supp}(u)\subset C\subset C_\epsilon\in \mathcal{C}$, so that $C\in\{O\}^o,$ $O=Int(D_\epsilon)$, and thus as in the proof of \ref{DualityBD}, one can take $f\in \mathcal{D}(U)$ with $f=1$ on a neighborhood of $\overline{C_\epsilon\cap D_\epsilon}$ and replacing $v$ by $vf,$ we are reduced to the case $v$ with compact support, i.e. $\mathcal{C}=\mathcal{F}$. Indeed, to prove continuity/boundedness in $u$ for $\mathcal{I}_{ppp}$ it suffices to prove the variant for $\mathcal{I}_{ppi}$ on $\mathcal{D}'_{\gamma,\Lambda}(U,\{O\}^o;E)$ with $O$ above, and then, after fixing $C$ by definition of the inductive limit topology, it suffices to prove continuity on $(\mathcal{D}'_{\gamma,\Lambda}(U,\mathcal{F};E),\mathcal{I}_{pp}).$
Fixing also $g\in \mathcal{D}(U)$ equal to $1$ on a neighborhood of the support of $f$, one could also replace $u$ by $ug$ and assume it with compact support.

We now fix $m$ for which $\mathcal{F}(M_{f_i'}(vf))$ is polynomially bounded of order $m$ (same for all finitely many relevant $i\in I$), let $\Gamma=WF(vf)\subset \overline{\lambda}.$

In order to build the continuous pairing, we can assume $\Gamma\subset \lambda$, thus $\gamma\subset -\Gamma^c$ and it suffices to get continuity for $u\in (\mathcal{D}'_{-\Gamma^c,\overline{-\Gamma^c}}(U,\mathcal{F};E),\mathcal{I}_{pm}),$ case $\lambda_1=-\Gamma^c, \lambda_2=\dot{T}^*U$ of the projective limit for $\mathcal{I}_{pp}$.
Fix $n$ such that $\text{supp}(f)\subset\text{supp}(g)\subset K_n=\{x\in Ud(x,U^c)\geq 1/n,|x|\leq n\}$ ($f,g$, and thus $n$ depending only of $v,C$), we will prove definition/ continuity 
on  $(\mathcal{D}'_{(-\Gamma^c)_{(n)},\overline{-\Gamma^c}_{(n)}}(U,\mathcal{F};E),\mathcal{I}_{p})$ which is mapped continuously via ${L_g}$ to $ 
 (\mathcal{D}'_{-\Gamma^c,\overline{-\Gamma^c}}(U,\mathcal{F};E),\mathcal{I}_{i}=\mathcal{I}_{p})$ and it suffices to prove continuity in $gu$ on the later space by definition of the projective topologies. Thus for any $\delta=(\Gamma_n)\in\Delta(-\Gamma^c)$ it suffices to prove continuity in $ug\in (\mathcal{D}'_{-\Gamma^c}(U,\mathcal{F}:\delta),\mathcal{I}_{H})$ by definitions of inductive limits. For $vf$ fixed of order $m$, we are only interested in $\Gamma_{m+d+1}\subset -\Gamma^c$.

Since $ \Gamma\cap -\Gamma_{m+d+1}=\emptyset$, the proof then follows exactly as in \cite[lemma 3]{BrouderDabrowski} since the key estimate (3) there, only depends on $m$, and once $m$ fixed, we only use the seminorms of $u$ of order $M=m+d+1$. The key estimate here we want to reproduce is :(with the notation of \cite{BrouderDabrowski} except for seminorms) :

\begin{align}\label{BDcountinuous}
|\langle M_{f_i}(ug)&,M_{f_i'}(fv)\rangle|  \le  \sum_j\Big(
|\langle\psi_jf_j,M_{f_i}(ug)\rangle|\\&\nonumber 
 + P_{M,{ U_{\alpha j}},\psi_j}(M_{f_i}(ug)) C
I_n^{M-m} 
+ P_{M,{ U_{\alpha j}},\psi_j}( M_{f_i}(ug))
   P_{N,{ U_{\beta j}},\psi_j}( M_{f_i'}(vf)
  I_n^{N+M} \Big).
\end{align}

We thus got the duality $(\mathcal{D}'_{\gamma,\Lambda}(U,\mathcal{C};E),\mathcal{I}_{ppp})'=(\mathcal{D}'_{\gamma,\Lambda}(U,\mathcal{C};E),\mathcal{I}_{\alpha})'=\mathcal{D}'_{\lambda,{\lambda}}(U,(\mathscr{O}_\mathcal{C})^o;E^*)$.
Note for further use that, as in \cite[(9)]{BrouderDabrowski} and replacing the first term by $P_{B_j}(u)$, $B_j$ bounded in $\mathcal{D}(\varphi_i(U_i))$, this can be made uniform in $vf\in B$ bounded set in $\mathcal{D}'_{\Gamma,\Gamma}(U,\mathcal{K};E^*).$

Let us now explain the bornological dual case in which we only have $\Gamma=WF(vf)\subset \overline{\lambda}$, but where we also know $DWF(vf)\subset \lambda$. Fix $l$ the order of polynomial boundedness of $\mathcal{F}(M_{f_i}(ug))$ uniform in $i\in I$, what we can always do in a bounded set $B$ of $\mathcal{D}'(U;E) $ and a fortiori in one of $\mathcal{D}'_{\gamma,\Lambda}(U,\mathcal{C};E)$. We even fix the closed cone $DWF_{l+d+1}(vf)=\lambda_l\subset \lambda$ so that $u\in (\mathcal{D}'_{\gamma,\overline{\gamma}}(U,\mathcal{F};E),\mathcal{I}_{p})\subset (\mathcal{D}'_{-\lambda_l^c,\overline{-\lambda_l^c}}(U,\mathcal{F};E),\mathcal{I}_{i}) $ and from the definition of the projective limit, we only have to prove boundedness on this space, thus for any  $\delta=(\Gamma_n)\in\Delta(-\lambda_l^c)$ it suffices to prove continuity in $u\in (\mathcal{D}'_{\overline{-\lambda_l^c}}(U,\mathcal{F}:\delta;E),\mathcal{I}_{H}).$ We only have to consider $\lambda_l\subset -\Gamma_{m+d+1}^c.$ 

We are in position to reason again as in \cite[lemma 3]{BrouderDabrowski} ($\lambda_l$ replacing $WF(v)$, $\Gamma_{m+d+1}$  replacing $WF(u)$ ) with the first term in the right hand side of the key estimate (3) replaced by the bound on $|I_{2j}|$ in the proof of \cite[proposition 9]{BrouderDabrowski}, giving the expected boundedness on $B$ with $M=m+d+1, N=l+d+1$. The key summarizing estimate, uniform on $u\in B$, in the notation of \cite{BrouderDabrowski} (except for seminorms) is thus :
\begin{align}\label{BDbounded}
&|\langle M_{f_i}(ug),M_{f_i'}(fv)\rangle|   \le C \sum_j\Big(
P_{N,\psi_j,{ U_{\beta j}}}(M_{f_i'}(fv)) P_{-l,g,\R^d}( M_{f_i}(ug)) I^{N-l}_d+ 
 \\&  P_{M,\psi_j,{ U_{\alpha j}}}( M_{f_i}(ug)) P_{-m,f,\R^d}(M_{f_i'}(fv))
I_d^{M-m} 
+ P_{M,\psi_j,{ U_{\alpha j}}}( M_{f_i}(ug))
   P_{N,\psi_j,{ U_{\beta j}}}(M_{f_i'}(fv)) 
  I_d^{N+M} \Big).\nonumber
\end{align}

Note finally that the previous expression also shows $v$ defines a continuous map on $(\mathcal{D}'_{\gamma,\Lambda}(U,\mathcal{C};E),\mathcal{I}_{pbp})$. 
Since bounded sets there are obviously the same as in $(\mathcal{D}'_{\gamma,\Lambda}(U,\mathcal{C};E),\mathcal{I}_{pmp})$, 
a continuous linear form on $(\mathcal{D}'_{\gamma,\Lambda}(U,\mathcal{C}),\mathcal{I}_{pbp})$ is a bounded form on $(\mathcal{D}'_{\gamma,\Lambda}(U,\mathcal{C}),\mathcal{I}_{pmp})$ thus we actually got the expected identification $(\mathcal{D}'_{\gamma,\Lambda}(U,\mathcal{C}),\mathcal{I}_{pbp})'=(\mathcal{D}'_{\gamma,\Lambda}(U,\mathcal{C}),\mathcal{I}_{pbp})^\times=\mathcal{D}'_{\lambda,\overline{\lambda}}(U,(\mathscr{O}_\mathcal{C})^o).$

Note the last bornological duality statement is a slight extension  of \cite[Proposition 24]{BrouderDabrowski}. Since the map from the inductive bornology to the von Neumann bornology of $\mathcal{I}_{iii}$ is bounded and conversely, the maps we used to compute bornological duals in the last lemma are also bounded with value the inductive bornology, we have already identified the bornological dual, it remains to identify topologies.
From equation \eqref{dualitymanifold}, continuity of $M_{f_i}$ and \cite[(9)]{BrouderDabrowski}, one gets continuity from $\mathcal{I}_{iii},$ to the bornological dual. For the converse, it suffices to check that the seminorms $P_{i,k,g(f\circ\varphi_i),V}=P_{k,f,V}\circ M_g$ (as in lemma \ref{contuinuousdenseInj}) are induced by the equibounded convergence topology, and since by the proof of \cite[Proposition 24]{BrouderDabrowski}, $P_{k,f,V}=P_{B'}$, for a bounded set in the corresponding space on $\varphi_i(U_i)$, $P_{i,k,g(f\circ\varphi_i),V}$ is the norm of uniform convergence on $gN_i(B').$ This concludes.
\end{proof} 


\section{Consequences in the open and closed case}

\subsection{General functional analytic properties in the open and closed case}
\begin{corollary}\label{FAClosedOpen2}
Let $\gamma= \Lambda= \overline{\gamma}$ a closed cone  and  $\mathcal{C}=\mathcal{C}^{oo}$ an enlargeable polar family of closed sets in $U$. Let $\lambda=-\gamma^c$
an open cone 
 $(\mathcal{D}'_{\lambda,\overline{\lambda}}(U,(\mathscr{O}_\mathcal{C})^o;E^*),\mathcal{I}_{iii})$
  is bornological, the strong, Arens and Mackey dual of $(\mathcal{D}'_{\gamma,\gamma}(U,\mathcal{C};E),\mathcal{I}_{iii}^{born})$ (the bornologification of the previously introduced topology)  and is also the completion of $(\mathcal{D}'_{\lambda,\lambda}(U,(\mathscr{O}_\mathcal{C})^o;E^*),\mathcal{I}_{iii})$, thus is nuclear and barrelled. Moreover, both $(\mathcal{D}'_{\lambda,\overline{\lambda}}(U,(\mathscr{O}_\mathcal{C})^o;E^*),\mathcal{I}_{iii})$ and $(\mathcal{D}'_{\gamma,\gamma}(U,\mathcal{C};E),\mathcal{I}_{iii}^{born})$ are Montel spaces, the later being also the strong, Arens  and Mackey dual of the former.
\end{corollary}
\begin{proof}
From the computation of bornological duals $(\mathcal{D}'_{\gamma,\overline{\gamma}}(U,\mathcal{C},E),\mathcal{I}_{iii}^{born})'=(\mathcal{D}'_{\gamma,\overline{\gamma}}(U,\mathcal{C};E),\mathcal{I}_{iii})^\times=\mathcal{D}'_{\lambda,\overline{\lambda}}(U,(\mathscr{O}_\mathcal{C})^o;E^*)$ and $(\mathcal{D}'_{\lambda,\overline{\lambda}}(U,(\mathscr{O}_\mathcal{C})^o;E^*),\mathcal{I}_{iii})'=\mathcal{D}'_{\gamma,\overline{\gamma}}(U,\mathcal{C};E^*)$ so that they form a compatible pair, and thus by Mackey-Arens thm, the strong topology induced by this duality on $(\mathcal{D}'_{\lambda,\overline{\lambda}}(U,(\mathscr{O}_\mathcal{C})^o;E^*)$ is stronger than Mackey which is stronger than the original topology $\mathcal{I}_{iii}$. To conclude to the identification of the strong dual, it suffices to prove $\mathcal{I}_{iii}$ is stronger than the strong topology, i.e. identity from $\mathcal{I}_{iii}$ to the strong dual is continuous. 
This identification especially implies $\mathcal{I}_{iii}$ is bornological on $\mathcal{D}'_{\lambda,\overline{\lambda}}(U,(\mathscr{O}_\mathcal{C})^o;E^*)$ which is claimed above since it is the Mackey topology and the previous bornological and ordinary dual are known to agree.


Fix $C\in (\mathscr{O}_\mathcal{C})^o$. Since the inductive limit on support is regular, take a bounded set $B$ in $\mathcal{D}'_{\gamma,\gamma}(U,\mathcal{F};E)$ with for any $u\in B$ $\text{supp}(u)\subset D$,  for some $D\in\mathcal{C}$. By definition and enlargeability as before, $D\cap C$ is compact, thus taking $f$ with compact support equal to $1$ here, $\mathcal{F}(fu)$ are all polynomially bounded of order at most $l$ uniformly in $u\in B.$ 
We have now to prove continuity of $v\mapsto \sup_{u\in B} |\langle u,fv\rangle|$ on $\mathcal{D}'_{\overline{\lambda}}(U,\mathcal{F}:\delta;E^*)$ for $\delta\in \Delta(\lambda).$
 We reason  again as in \cite[lemma 10
]{BrouderDabrowski} ($\gamma$ replacing $WF(u)$, $\Gamma_{l+d+1}\subset -\gamma^c$  replacing $WF(v)$ ). Of course we first decompose as in 
 equation \eqref{dualitymanifold} as a sum of terms  $\langle M_{f_i}(u),M_{f_i'}(fv)\rangle$ which reduces the problem to the trivial case $|I|=1.$ We then  bound as in \cite[(7)
]{BrouderDabrowski}.



What follows then requires various identifications of bornologies we gather in the following lemma (after the end of the proof) which is technical but of independent interest  and we now use.

The identification of the completion, then comes from the already known nuclearity of $(\mathcal{D}'_{\lambda,\lambda}(U,(\mathscr{O}_\mathcal{C})^o;E^*),\mathcal{I}_{iii}),$ from the previous duality computations and from \cite[p.~140 Corollary 3]{HogbeNlendMoscatelli}. Note that one also needs to use that the equicontinuous bornology on $\mathcal{D}'_{\gamma,\gamma}(U,\mathcal{C};E)$ is (by our next lemma) the same as the von Neumann bornology of its original topology $\mathcal{I}_{iii},$ since the bornology in the cited result on the dual is this equicontinuous bornology and we computed the bornological dual for $\mathcal{I}_{iii}$.

Now the completion $(\mathcal{D}'_{\lambda,\overline{\lambda}}(U,(\mathscr{O}_\mathcal{C})^o;E^*),\mathcal{I}_{iii})$ is barrelled ( as any completion of a quasi-barrelled space, see e.g. K\"othe \cite[$\S 27.1.(2)$]{Kothe}), it is nuclear (as any completion of a nuclear space, see e.g. Treves \cite[(50.2) Proposition 50.1 p514]{Treves}) and of course complete, thus it is a Montel space (see e.g. Treves \cite[Corollary 2 p520]{Treves}). 
$\mathcal{D}'_{\gamma,\gamma}(U,\mathcal{C};E)$ with the strong topology coming from the pairing with $(\mathcal{D}'_{\Lambda,\overline{\Lambda}}(U,(\mathscr{O}_\mathcal{C})^o;E^*),\mathcal{I}_{iii})$ 
 is also a Montel space (see e.g. K\"othe \cite[$\S 27.2.(2)$]{Kothe} or Treves \cite[Proposition 34.6 p357]{Treves}). 

It only remains to identify this strong topology with the bornologification $\mathcal{I}_{iii}^{born}$. Since we saw $(\mathcal{D}'_{\lambda,\overline{\lambda}}(U,(\mathscr{O}_\mathcal{C})^o;E^*),\mathcal{I}_{iii})$ is nuclear complete, its strong dual is (ultra)bornological. Indeed by a Theorem of Schwartz (\cite{Schwarz} see also \cite[VII.6 Th 1 p 62]{HogbeNlend}), the strong dual of a complete Schwartz locally convex space is bornological in its strong topology.
In order to identify this strong dual with $\mathcal{I}_{iii}^{born}$, it thus suffices to check $\mathcal{I}_{iii}$ has the same bounded sets as this strong dual, what is again done in the next lemma.

Finally this strong dual coincides with the Mackey dual from Arens-Mackey Theorem, since we computed the dual of this strong dual to be the original space, so that the strong topology is compatible with the duality. The identification of Mackey and Arens duals is known for semi-Montel spaces \cite{Horvath}.

\end{proof}
\subsection{Computation of bornologies and identification of topologies by equicontinuous set computations.}
The proof of our previous result will be completed once done the following computation of bornologies. 
\begin{lemma}\label{bornologyIdentities}With the notation of the previous corollary, especially, $\gamma$ a closed cone. 
 On $\mathcal{D}'_{\gamma,\gamma}(U,\mathcal{C};E)$, the bornology of its strong topology  from its duality with $(\mathcal{D}'_{\lambda,\overline{\lambda}}(U,(\mathscr{O}_\mathcal{C})^o;E^*),\mathcal{I}_{iii})$ or $(\mathcal{D}'_{\lambda,{\lambda}}(U,(\mathscr{O}_\mathcal{C})^o;E^*),\mathcal{I}_{iii})$ equals its equicontinuous parts bornology coming again from either  $(\mathcal{D}'_{\lambda,\overline{\lambda}}(U,(\mathscr{O}_\mathcal{C})^o;E^*),\mathcal{I}_{iii})$ or $(\mathcal{D}'_{\lambda,{\lambda}}(U,(\mathscr{O}_\mathcal{C})^o;E^*),\mathcal{I}_{iii})$ and the von Neumann bornology of its original topology $\mathcal{I}_{iii}$.

 The von Neumann bornology of $\mathcal{D}'_{\lambda,\overline{\lambda}}(U,(\mathscr{O}_\mathcal{C})^o;E^*)$ of its strong topology coming from duality with $(\mathcal{D}'_{\gamma,\gamma}(U,\mathcal{C};E),\mathcal{I}_{iii})$, 
 is also its equicontinuous bornology as a strong dual and its equibounded topology as a bornological dual.

\end{lemma}
 \begin{proof}
 
We start by identifying the von Neumann bornology of  $(\mathcal{D}'_{\gamma,\gamma}(U,\mathcal{C};E),\mathcal{I}_{iii})$ with one of the stated equicontinuous bornologies. From a result of Hogbe-Nlend, we will relate it to an equibounded bornology on a bornological dual which we mostly have to relate to the topology we already considered. This is thus a standard bornological result. 
 From the identification of Theorem \ref{duals} of $(\mathcal{D}'_{\gamma,\gamma}(U,\mathcal{C};E),\mathcal{I}_{iii})$ as a bornological dual of 

 \noindent$\underrightarrow{\lim}_{C\in (\mathscr{O}_\mathcal{C})^{o}} \underrightarrow{\lim}_{\Pi\subset \lambda}\left(\mathcal{D}'_{ \Pi,\Pi}(U,\{C_\eta, \eta<\epsilon(C)\}^{oo};E^*),\mathcal{I}_{iii}^{born}\right)$ (with for every $C\in (\mathscr{O}_\mathcal{C})^{o}$, $\epsilon(C)$ is some fixed value with $C_{\epsilon(C)}\in (\mathscr{O}_\mathcal{C})^{o}$), the von Neumann bornology here is the equibounded bornology from this bornological duality and we note that from a general category theory statement (the associated topology functor being a left adjoint it preserves inductive limit, cf. e.g. \cite{FrolicherKriegel})  the topologification of this bornological inductive limit is merely the von Neumann bornology of the topological inductive limit namely $$(\mathcal{D}'_{\lambda,\lambda}(U,(\mathscr{O}_\mathcal{C})^{o};E^*),\mathcal{I}):=\underrightarrow{\lim}_{C\in (\mathscr{O}_\mathcal{C})^{o}} \underrightarrow{\lim}_{\Pi\subset \lambda}\left(\mathcal{D}'_{ \Pi,\Pi}(U,\{C_\eta, \eta<\epsilon(C)\}^{oo};E^*),\mathcal{I}_{iii}^{born}\right).$$
 We claim that, exploiting the bornologicality in the open cone case, this topology is the same as 
 $$(\mathcal{D}'_{\lambda,\lambda}(U,(\mathscr{O}_\mathcal{C})^o;E^*),\mathcal{I}_{iii})=\underrightarrow{\lim}_{C\in (\mathscr{O}_\mathcal{C})^{o}} \underrightarrow{\lim}_{\Pi\subset \lambda}\left(\mathcal{D}'_{ \Pi,\Pi}(U,\{C_\eta, \eta<\epsilon(C)\}^{oo};E^*),\mathcal{I}_{iii}\right).$$ Indeed, closed cones can be replaced by closed cones of the form $\Pi=\overline{\Lambda_\Pi}$ (since any closed cone is included in such a cone which is still in $\lambda$ 
  and which is the closure of an open cone) and thus even an inductive limit over open cones :\begin{align*}(\mathcal{D}'_{\lambda,\lambda}(U,(\mathscr{O}_\mathcal{C};E^*)^o),\mathcal{I})&=\underrightarrow{\lim}_{C\in (\mathscr{O}_\mathcal{C})^{o}} \underrightarrow{\lim}_{\Lambda\subset\overline{\Lambda}\subset \lambda}\left(\mathcal{D}'_{ \overline{\Lambda},\overline{\Lambda}}(U,\{C_\eta, \eta<\epsilon(C)\}^{oo};E^*),\mathcal{I}_{iii}^{born}\right)\\&\simeq \underrightarrow{\lim}_{C\in (\mathscr{O}_\mathcal{C})^{o}} \underrightarrow{\lim}_{\Lambda\subset\overline{\Lambda}\subset \lambda}\left(\mathcal{D}'_{ {\Lambda},{\Lambda}}(U,\{C_\eta, \eta<\epsilon(C)\}^{oo};E^*),\mathcal{I}_{iii}\right)\end{align*}
  since $\left(\mathcal{D}'_{ {\Lambda},{\Lambda}}(U,\{C_\eta, \eta<\epsilon(C)\}^{oo};E^*),\mathcal{I}_{iii}\right)\to \left(\mathcal{D}'_{ \overline{\Lambda},\overline{\Lambda}}(U,\{C_\eta, \eta<\epsilon(C)\}^{oo};E^*),\mathcal{I}_{iii}^{born}\right)$ is continuous for $\Lambda$ open by lemma \ref{contuinuousdenseInj}  and Theorem \ref{DualityBD} for its bornologicality (one uses also $\{C_\eta, \eta<\epsilon(C)\}^{oo}$ is polar enlargeable by the remark after proposition \ref{EnlargeableStable}). For the converse map, one uses $\overline{\Lambda}$ is a closed cone thus included in some open $\Lambda'$ of the type of $\Lambda$. Reasoning in the same way for $\mathcal{I}_{iii}$, one gets $\mathcal{I}_{iii}=\mathcal{I}$ as claimed, and we now use what we noted that this is the topologification of the bornology whose equibounded dual gives $(\mathcal{D}'_{\gamma,\gamma}(U,\mathcal{C};E),\mathcal{I}_{iii})$ with its von Neumann bornology.

   It follows from \cite[5:2.3 Prop (6) p 74]{HogbeNlend2} that the von Neumann bornology  of $(\mathcal{D}'_{\gamma,\gamma}(U,\mathcal{C};E),\mathcal{I}_{iii})$ is also the equicontinuous bornology (generated by polar of $0$-neighborhoods) from its duality with $(\mathcal{D}'_{\lambda,\lambda}(U,(\mathscr{O}_\mathcal{C})^o;E^*),\mathcal{I}_{iii}).$

\medskip

Since the only used result for identifying the completion above is the one we just proved, we can now use that  $(\mathcal{D}'_{\lambda,\overline{\lambda}}(U,(\mathscr{O}_\mathcal{C})^o;E^*),\mathcal{I}_{iii})$ is the completion of $(\mathcal{D}'_{\lambda,\lambda}(U,(\mathscr{O}_\mathcal{C})^o;E^*),\mathcal{I}_{iii})$ to deduce from \cite[\S 21.4.(5) p261]{Kothe}  they induce the same equicontinuous sets on their common dual $\mathcal{D}'_{\gamma,\gamma}(U,\mathcal{C};E),$ thus identifying the corresponding bornologies.

Finally, the von Neumann bornology of the strong topology of $\mathcal{D}'_{\gamma,\gamma}(U,\mathcal{C};E)$   is a priori coarser than the bornology generated by equicontinuous parts of this space (indeed, being bounded on a zero neighborhood implies being bounded on bounded sets always absorbed by zero neighborhoods). But actually, here, since  $(\mathcal{D}'_{\lambda,\lambda}(U,(\mathscr{O}_\mathcal{C})^o;E^*),\mathcal{I}_{iii}),(\mathcal{D}'_{\lambda,\overline{\lambda}}(U,(\mathscr{O}_\mathcal{C})^o;E^*),\mathcal{I}_{iii})$ are bornological (from Theorem \ref{DualityBD} for the first, from the beginning of corollary \ref{FAClosedOpen2}, independent of the currently proved lemma for the second), using \cite[Corollary of prop (6) \S5.2.3 p 75]{HogbeNlend2} every strongly bounded subset of the dual of a bornological space is equicontinuous, thus the strong bornologies  of $\mathcal{D}'_{\gamma,\gamma}(U,\mathcal{C};E)$ actually coincide with the equicontinuous bornology as claimed.

 Since $(\mathcal{D}'_{\lambda,\overline{\lambda}}(U,(\mathscr{O}_\mathcal{C})^o;E^*),\mathcal{I}_{iii})$ is a nuclear locally convex space
 it is a Schwartz locally convex space.
 By a Theorem of Schwartz (\cite{Schwarz} see also \cite[VII.6 Th 1 p 62]{HogbeNlend}), the strong dual of a complete Schwartz locally convex space is ultrabornological, i.e. so is the strong topology on 
$\mathcal{D}'_{\gamma,\gamma}(U,\mathcal{C};E)$.

Thus similarly to one of our previous results, with the same reasoning, using that $\mathcal{D}'_{\gamma,\gamma}(U,\mathcal{C};E)$ is bornological with its strong topology coming from the completion, on its dual (computed by semi-reflexivity of these Montel spaces) $(\mathcal{D}'_{\lambda,\overline{\lambda}}(U,(\mathscr{O}_\mathcal{C})^o)),\mathcal{I}_{iii})$ , the von Neumann bornology and the equicontinuous bornology also coincide as stated. The identification with the equibounded bornology as a bornological dual with the von Neumann bornology of the strong dual is obvious.
 
\end{proof}
We gather in the next lemma various consequences of computations of equicontinuous sets that are of independent interest and will simplify greatly support condition issues later. Most of the results can be seen as commutation of inductive/projective limits on support with other projective/inductive limits on (dual) wave front set.
\begin{lemma}\label{bornologyIdentitiesEqui}With the notation of the previous corollary, especially for $\gamma$ a closed cone we have on $\mathcal{D}'_{\gamma,\gamma}(U,\mathcal{C};E),$ $ \mathcal{I}_{iii}=\mathcal{I}_{iip}=\mathcal{I}_{ipp}=\mathcal{I}_{ipi}$ and equicontinuous sets of their dual are those bounded in the bornological inductive limit as in theorem \ref{duals}: 
 $$\mathcal{D}'_{\lambda,\lambda}(U,(\mathscr{O}_\mathcal{C})^{o};E^*))=\underrightarrow{\lim}_{C\in (\mathscr{O}_\mathcal{C})^{o}} \underrightarrow{\lim}_{\Pi\subset \lambda}\left(\mathcal{D}'_{ \Pi,\Pi}(U,\{F\in \mathcal{F}, F\subset C\};E^*\right)$$
Likewise, on $\mathcal{D}'_{\lambda,\overline{\lambda}}(U,(\mathscr{O}_\mathcal{C})^o;E^*)$, we have $\mathcal{I}_{iii}= \mathcal{I}_{iip}=\mathcal{I}_{pii}=\mathcal{I}_{pip}=\mathcal{I}_{pmi}=\mathcal{I}_{pmp}$ (we will see identity with $\mathcal{I}_{ppp}$ in proposition \ref{FAGeneral}) and on $\mathcal{D}'_{\gamma,\gamma}(U,\mathcal{C};E)$ we have also
$\mathcal{I}_{iii}=\mathcal{I}_{pii}$ and as a consequence for any $\gamma$ (not necessarily closed) on $\mathcal{D}'_{\gamma,\gamma}(U,\mathcal{C};E)$ we have $\mathcal{I}_{i\alpha}=\mathcal{I}_{p\alpha}$ for any $\alpha\in\{ii,mi,ip,mp,pi,pp\}.$

Moreover, in the general context of the previous theorem,  we have topological identifications (as usual $\lambda_i$ open cones, $\Pi$ closed cones) : \begin{equation}\label{Alternativeppp}(\mathcal{D}'_{\gamma,\Lambda}(U,\mathcal{C};E),\mathcal{I}_{p pp})\simeq\underleftarrow{\lim}_{ \lambda_1\supset \gamma}\underleftarrow{\lim}_{\lambda_2\supset \lambda_1\cup \Lambda}\left(\mathcal{D}'_{\lambda_1,\overline{\lambda_1}\cap \lambda_2}(U,\mathcal{C};E),\mathcal{I}_{pmp}\right),\end{equation}
\begin{equation}\label{Alternativepip}(\mathcal{D}'_{\gamma,\overline{\gamma}}(U,\mathcal{C};E),\mathcal{I}_{p ip})\simeq\underleftarrow{\lim}_{ \lambda_1\supset \gamma}\left(\mathcal{D}'_{\lambda_1,\overline{\lambda_1}}(U,\mathcal{C};E),\mathcal{I}_{pip}\right),\end{equation}
\begin{equation}\label{Alternativeiii}(\mathcal{D}'_{\lambda,\lambda}(U,(\mathscr{O}_\mathcal{C})^o;E^*),\mathcal{I}_{iii})\simeq\underrightarrow{\lim}_{\Pi\subset \lambda}(\mathcal{D}'_{\Pi,\Pi}(U,(\mathscr{O}_\mathcal{C})^o;E^*),\mathcal{I}_{iii}).\end{equation}
As a consequence on $\mathcal{D}'_{\lambda,\lambda}(U,(\mathscr{O}_\mathcal{C})^o;E^*)$ we have  $\mathcal{I}_{iii}=\mathcal{I}_{ppp}$.

Finally, we have a topological isomorphism :
\begin{equation}\label{Alternativei}(\mathcal{D}'_{\lambda,\overline{\lambda}}(U,\mathcal{F};E),\mathcal{I}_{iii})\simeq\underleftarrow{\lim}_n\underrightarrow{\lim}_{\Pi\subset \lambda}(\mathcal{D}'_{n,\Pi}(U,\mathcal{F};E)),\end{equation}
where $(\mathcal{D}'_{n,\Pi}(U,\mathcal{F};E))=\{u\in \mathcal{D}'(U,\mathcal{F};E), DWF_n(u)\subset \Pi\},$ for $\Pi$ a closed cone, is equipped with the projective limit topology from the seminorms of the strong topology of $\mathcal{D'}(U,\mathcal{F})$ and  $P_{i,n,f,V}$ such that  $(d\varphi_i)^*(\text{supp}(f\circ\varphi_i^{-1})\times V)\subset \Pi^c\subset DWF_n(u)^c, \text{supp}(f)\subset U_i$.

\end{lemma}
 \begin{proof}
 As advertised before the statement, all the topological identifications will come from computation of equicontinuous sets. The point is that equicontinuous sets being a bornological notion, they are easy to compute in duals of topological projective and inductive limits (see \cite[\S 22.7]{Kothe}). This is in sharp contrast with computations of topological duals of topological inductive/projective limits which are really limited in general and basically reduced to strong duals of regular inductive limits and Mackey/Arens duals of projective limits (which are bornological inductive/projective limits in disguise). At the end, we will also use the identification of bounded and equicontinuous sets from the previous lemma in the closed cone case that will be our building block for the computations. At the end we are reduced to a concrete computation on our witnesses : supports and cones. We will thus repeat several times the same kind of arguments. 

Let us first determine equicontinuous sets in $\mathcal{D}'_{\lambda,\lambda}(U,(\mathscr{O}_\mathcal{C})^{o};E^*)$ for the stated duality with $(\mathcal{D}'_{\gamma,\gamma}(U,\mathcal{C};E), \mathcal{I}_{iii})$. 
The fact that those bounded sets are equicontinuous is obvious using our first result above in  lemma \ref{bornologyIdentities} 
since such bounded sets are bounded in $(D'_{\Pi,\Pi}(U,(\mathscr{O}_\mathcal{C})^{o};E^*),\mathcal{I}_{iii})$  and thus they are equicontinuous there from the duality with 

\noindent $(D'_{-\Pi^c,-\Pi^c}(U,\mathcal{C};E),\mathcal{I}_{iii})\leftarrow(\mathcal{D}'_{\Gamma,\Gamma}(U,\mathcal{C};E), \mathcal{I}_{iii})$  from the lemma, and, from the above continuous injection, this concludes, and we thus only have to prove the converse.
Note this is a slight extension of \cite[lemma 6.3]{BrouderDangHelein} coming from \cite[Proposition 24]{BrouderDabrowski}. 

  From the continuous inclusion $\mathcal{E}(U,\mathscr{O}_\mathcal{C};E)\to (\mathcal{D}'_{\gamma,\gamma}(U,\mathcal{C};E),\mathcal{I}_{iii})$, such set is also equicontinuous in $\mathcal{D}'(U,(\mathscr{O}_\mathcal{C})^{o};E^*)$ and from the regular inductive limit definition, since an equicontinuous set $B$ is bounded and thus uniformly supported in some $C\in \mathcal{C}$. Moreover, from the continuous inclusion $(\mathcal{D}'_{\gamma,\gamma}(U,\mathcal{K};E),\mathcal{I}_{iii})\to (\mathcal{D}'_{\gamma,\gamma}(U,\mathcal{C};E),\mathcal{I}_{iii})$, it is also equicontinuous in $\mathcal{D}'_{\lambda,\lambda}(U,\mathcal{F};E^*).$ Taking $f_n$ smooth compactly supported equal to $1$ on a large subset $K_n$, since the multiplication $L_f$ is continuous $(\mathcal{D}'_{\gamma,\gamma}(U,\mathcal{F};E),\mathcal{I}_{iii})\to (\mathcal{D}'_{\gamma,\gamma}(U,\mathcal{K};E),\mathcal{I}_{iii})$ by lemma \ref{contuinuousdenseInj}  $L_{f_n}B$ is equicontinuous in $\mathcal{D}'_{\lambda,\lambda}(U,\mathcal{K};E^*),$  from its duality with $(\mathcal{D}'_{\gamma,\gamma}(U,\mathcal{F};E),\mathcal{I}_{iii})\simeq \underleftarrow{\lim}_{\Lambda\supset\gamma} (\mathcal{D}'_{\overline{\Lambda},\overline{\Lambda}}(U,\mathcal{F};E),\mathcal{I}_{iii})$. We are thus reduced to the dual of a projective limit over open cones above. From the projective limit description and \cite[\S 22.7.(5)p 292]{Kothe}, $L_{f_n}(B)$ is also equicontinuous in $(\mathcal{D}'_{\overline{\Lambda_n},\overline{\Lambda_n}}(U,\mathcal{F};E),\mathcal{I}_{iii})'$ for some $\Lambda_n\supset \gamma$ open cone and thus in $(\mathcal{D}'_{{\Lambda_n},{\Lambda_n}}(U,\mathcal{F};E),\mathcal{I}_{iii})'$ from the continuous inclusion $(\mathcal{D}'_{{\Lambda_n},{\Lambda_n}}(U,\mathcal{F};E),\mathcal{I}_{iii})\to (\mathcal{D}'_{\overline{\Lambda_n},\overline{\Lambda_n}}(U,\mathcal{F};E),\mathcal{I}_{iii}).$ We thus concludes since we determined these equicontinuous sets  in lemma \ref{bornologyIdentities} to be bounded sets in $(\mathcal{D}'_{{-\Lambda_n}^c,{-\Lambda_n}^c}(U,\mathcal{K};E^*),\mathcal{I}_{iii})$ as claimed ($\Pi_n={-\Lambda_n}^c\subset \lambda$ closed).
    Taking $\Pi=\cup_{n\in \N} T^*(K_n)\cap \Pi_n$ for an exhaustive increasing sequence of compacts $K_n$ as above, $\Pi\subset \lambda$ is obviously closed (since convergence of a sequence imply staying in $T^*(K_n)$ thus in $\Pi_n$). From the local definition of boundedness in $(\mathcal{D}'_{\Pi,\Pi}(U,\mathcal{F};E),\mathcal{I}_{iii})$ $B$ is bounded there and we already checked the uniform support property, concluding to the stated result.
  
We moreover determine equicontinuous sets in $\mathcal{D}'_{\lambda,\lambda}(U,(\mathscr{O}_\mathcal{C})^{o};E^*)$ for the stated duality with $(\mathcal{D}'_{\gamma,\gamma}(U,\mathcal{C};E), \mathcal{I}_{iip})$. Since such equicontinuous sets are equicontinuous in the previous sense, it only remains to prove that if $B$ is bounded in  $\mathcal{D}'_{\Pi,\Pi}(U,\mathcal{F};E^*)$ for some closed $\Pi\subset \lambda$ and uniformly supported in some $D\in (\mathscr{O}_\mathcal{C})^{o}$ it is also equicontinuous in this new sense. 
Take $\eta$  with $D_\eta\in (\mathscr{O}_\mathcal{C})^{o}$ our previous result proves it is equicontinuous in $(\mathcal{D}'_{\lambda,\lambda}(U,\{D_\epsilon,\epsilon=\eta(1-1/n),n\in \N\}^{oo};E^*)$ for its duality with $(\mathcal{D}'_{\gamma,\gamma}(U,\{Int(D_\epsilon),\epsilon=\eta(1-1/n),n\in \N\}^{o};E),\mathcal{I}_{iii}),$ (using enlargeablity of $\mathscr{D}=\{D_\epsilon,\epsilon=\eta(1-1/n),n\in \N\}^{oo}$ by the remark after proposition \ref{EnlargeableStable}, and $(\mathscr{O}_\mathscr{D})^{o}=\{Int(D_\epsilon),\epsilon=\eta(1-1/n),n\in \N\}^{o}$
)
 but this exactly implies the equicontinuity for the projective limit $(\mathcal{D}'_{\gamma,\gamma}(U,\mathcal{C};E), \mathcal{I}_{iip}),$ since $Int(D_\eta)\in \mathcal{C}^o,$ and thus $(\mathcal{D}'_{\gamma,\gamma}(U,\mathcal{C};E), \mathcal{I}_{iip})\to(\mathcal{D}'_{\gamma,\gamma}(U,\{Int(D_\eta)\}^{o};E),\mathcal{I}_{iii})\to(\mathcal{D}'_{\gamma,\gamma}(U,\{Int(D_\epsilon),\epsilon=\eta(1-1/n),n\in \N\}^{o};E),\mathcal{I}_{iii})$ is continuous so that our set is also equicontinuous in $(\mathcal{D}'_{\gamma,\gamma}(U,\mathcal{C};E), \mathcal{I}_{iip})'.$ 
 Finally, the identification of both kinds of equicontinuous sets imply the remaining equality of topologies to be proved $\mathcal{I}_{iip}=\mathcal{I}_{iii}$ on $\mathcal{D}'_{\gamma,\gamma}(U,\mathcal{C};E).$ The other ones were checked in Theorem \ref{DualityBD}.

\medskip

For the other statements about identification of topologies in the open/closed case, first note that since $\lambda$ is open, we have by definition $\mathcal{I}_{i\alpha}=\mathcal{I}_{p\alpha}$ for any $\alpha$ with two letters. We thus first check $\mathcal{I}_{iii}=\mathcal{I}_{iip}$. For, we reason as at the end of the proof of  theorem \ref{DualityBD} so that using enlargeability, we only have to prove, for $C\in \mathcal{C}$ with $O'=Int(C_\epsilon),$ that we have a continuous map $$(\mathcal{D}'_{\lambda,\overline{\lambda}}(U,\{O'\}^o;E^*),\mathcal{I}_{iii})\to  (\mathcal{D}'_{\gamma,\gamma}(U,\{C\}^{oo};E),\mathcal{I}_{iii}^{born})_b'.$$ 
 
As in this proof we can take $f$ compactly supported  equal to $1$ on a neighborhood of  the compact $V\cap F$ for $F\in \{C\}^{oo},V\in \{O'\}^o$, thus from regularity of the inductive limit on support, we take $B$ bounded in $(\mathcal{D}'_{\gamma,\gamma}(U,\{C\}^{oo};E),\mathcal{I}_{iii})$
uniformly supported in $F$ and of course thus bounded  in  $(\mathcal{D}'_{\gamma,\gamma}(U,\mathcal{F};E),\mathcal{I}_{iii})$ and we only need to bound $\sup_{u\in B}|\langle u,fv\rangle|$. From the continuity of $L_f:v\mapsto fv,$ explained in lemma \ref{contuinuousdenseInj} 
 $(\mathcal{D}'_{\lambda,\lambda}(U,\mathcal{F};E^*),\mathcal{I}_{iii})\to (\mathcal{D}'_{\lambda,\lambda}(U,\mathcal{K};E^*),\mathcal{I}_{iii})$ and from our previous identification of the completion, we also have continuity $(\mathcal{D}'_{\lambda,\overline{\lambda}}(U,\mathcal{F};E^*),\mathcal{I}_{iii})\to (\mathcal{D}'_{\lambda,\overline{\lambda}}(U,\mathcal{K};E^*),\mathcal{I}_{iii})$ and the final continuity stated then follows from the identification of this last target space, where $vf$ lives, with a strong dual.
  Gathering the results in a projective limit we got, similarly to our proof in Theorem \ref{DualityBD},  continuous maps :
\begin{align*}(\mathcal{D}'_{\lambda,\overline{\lambda}}(U,(\mathscr{O}_\mathcal{C})^{o};E^*),\mathcal{I}_{iip})&=\underleftarrow{\lim}_{O\in(\mathscr{O}_\mathcal{C})^{oo}} (\mathcal{D}'_{\lambda,\overline{\lambda}}(U,\{O\}^{o};E^*),\mathcal{I}_{iii})\\&\to \underleftarrow{\lim}_{C\in \mathcal{C}} (\mathcal{D}'_{\gamma,\gamma}(U,\{C\}^{oo};E),\mathcal{I}_{iii}^{born})_b'=(\mathcal{D}'_{\gamma,\gamma}(U,\mathcal{C};E),\mathcal{I}_{iii}^{born})'_b
\\&=(\mathcal{D}'_{\lambda,\overline{\lambda}}(U,(\mathscr{O}_\mathcal{C})^{o};E^*),\mathcal{I}_{iii})\to (\mathcal{D}'_{\lambda,\overline{\lambda}}(U,(\mathscr{O}_\mathcal{C})^{o};E^*),\mathcal{I}_{iip}),\end{align*}
with again the last equality from our computation of the strong dual, and the last obvious maps concludes to the expected equality $\mathcal{I}_{iii}=\mathcal{I}_{iip}$ in the open cone case (we used that the bornologification of a regular inductive limit is the inductive limit of bornologification, using the topology functor is a left adjoint and thus preserves inductive limits see e.g. \cite[Prop 2.1.10]{FrolicherKriegel}).

We are now ready to also identify $\mathcal{I}_{iii}=\mathcal{I}_{pmi},\mathcal{I}_{iip}=\mathcal{I}_{pmp}$ on $\mathcal{D}'_{\lambda,\overline{\lambda}}(U,(\mathscr{O}_{\mathcal{C}})^o;E^*)$. For it suffices to identify $\mathcal{I}_{iii}=\mathcal{I}_{pmi}$ on $\mathcal{D}'_{\lambda,\overline{\lambda}}(U,\mathcal{F};E^*)$ (same inductive/projective limits afterward) but in this case, there is no inductive limit on $\Pi$ for $\mathcal{I}_{pmi}$ which is only \begin{align*}\underleftarrow{\lim}_{n\to \infty }(\mathcal{D}'_{\lambda_{(n)},\overline{\lambda}_{(n)}}(U,\mathcal{F};E^*),\mathcal{I}_{pii})&=\underleftarrow{\lim}_{n\to \infty }\underleftarrow{\lim}_{\Lambda\supset \lambda_{(n)} }(\mathcal{D}'_{\Lambda,\overline{\Lambda}}(U,\mathcal{F};E^*),\mathcal{I}_{iii})\\&=\underleftarrow{\lim}_{n\to \infty }(\mathcal{D}'_{Int(\lambda_{(n)}),\overline{\lambda}_{(n)}}(U,\mathcal{F};E^*),\mathcal{I}_{pii}=\mathcal{I}_{iii}).\end{align*}
Indeed, the first identity is merely the definition of the topology since $\overline{\lambda}_{(n)}=\overline{\lambda_{(n)}}$, the only problem for us being $\lambda_{(n)}$ may not be open, but any open $\Lambda\supset \lambda_{(n)}$ of course contains $Int(\lambda_{(n)})=\lambda_{(n,1)}\cup Int(T_{(n,2)})$ and conversely $Int(\lambda_{(n)})\supset \lambda_{(n+1)}$ proving by cofinality the equality with the second inductive limit. We can then compute equicontinuous sets in the dual using \cite[\S 22.7.(5)]{Kothe} which implies an equicontinuous set in the dual is equicontinuous in some $(\mathcal{D}'_{Int(\lambda_{(n)}),\overline{\lambda}_{(n)}}(U,\mathcal{F};E^*),\mathcal{I}_{pii}=\mathcal{I}_{iii})'=\mathcal{D}'_{-Int(\lambda_{(n)})^c,-Int(\lambda_{(n)})^c}(U,\mathcal{K};E).$ Of course these equicontinuous sets are also equicontinuous in $(\mathcal{D}'_{\lambda,\overline{\lambda}}(U,\mathcal{F};E^*),\mathcal{I}_{iii})'=(\mathcal{D}'_{-\lambda^c,-\lambda^c}(U,\mathcal{K};E),\mathcal{I}_{iii})$  since $-\lambda^c\supset -Int(\lambda_{(n)})^c$. But since conversely an equicontinuous set of this second space is bounded thus uniformly supported in some compact set $K$, using again continuity of multiplication by a smooth test function from lemma \ref{contuinuousdenseInj} that will be identity on this subspace controlled by the support condition on $K$ we get it has to be equicontinuous in some $\mathcal{D}'_{-Int(\lambda_{(n)})^c,-Int(\lambda_{(n)})^c}(U,\mathcal{K};E)$. Thus, having identified equicontinuous sets in the dual, we have identified $\mathcal{I}_{iii}=\mathcal{I}_{pmi}$.

To prove $\mathcal{I}_{iii}=\mathcal{I}_{pii}$ on $\mathcal{D}'_{\gamma,\gamma}(U,\mathcal{C};E)$, again $\mathcal{C}=\mathcal{F}$ is enough, we only use the first topology is nuclear and described also by Theorem \ref{DualityBD} by $\mathcal{I}_{ipi}$ given by formula  \eqref{EqualConep} but since the topology is complete it is equal to its completion, and it suffices to see $\mathcal{I}_{pii}$ describes the completion given by 
\cite[p.~140 Corollary 3]{HogbeNlendMoscatelli}. For, we have to compute the dual of $\mathcal{I}_{ipi}$ with  equicontinuous bornology which is $\underrightarrow{\lim}_{ \Pi\subset -\gamma^c}\left(\mathcal{D}'_{\Pi,\Pi}(U,\mathcal{K};E^*),\mathcal{I}_{iii}\right)$ with equicontinuous sets exactly bounded sets in one of the terms of the inductive limit (i.e. an inductive limit bornology, the equicontinuous sets in each space again come from lemma \ref{bornologyIdentities}). Then the completion is given by a bornological dual, which is nothing but a projective limit of bornological duals, i.e. $\underleftarrow{\lim}_{ \Pi\subset -\gamma^c}\left(\mathcal{D}'_{-\Pi^c,\overline{-\Pi^c}}(U,\mathcal{F};E),\mathcal{I}_{iii}\right)=\underleftarrow{\lim}_{ \lambda\supset \Gamma}\left(\mathcal{D}'_{\lambda,\overline{\lambda}}(U,\mathcal{F};E),\mathcal{I}_{iii}\right)$ the equality in calling $\lambda=-\Pi^c$. This is the definition of $\mathcal{I}_{pii}$ concluding to the equality $\mathcal{I}_{iii}=\mathcal{I}_{pii}.$ The remaining identities are straightforward from the definitions.

\medskip

We now come back to the general $\gamma,\Lambda$ case and want to determine equicontinuous sets in the dual of $(\mathcal{D}'_{\gamma,\Lambda}(U,\mathcal{C};E),\mathcal{I}_{ppp})$.

We start with the case $\mathcal{C}=\mathcal{F}$ and since $(\mathcal{D}'_{\gamma,\Lambda}(U,\mathcal{F};E),\mathcal{I}_{ppp})$ is a projective limit, the equicontinuous bornology on the dual is as usual \cite[\S 22.7.(5)]{Kothe} the inductive limit bornology (on $\lambda_i$ open cones) $$\underrightarrow{\lim}_{ \lambda_1\supset \gamma}\underrightarrow{\lim}_{\lambda_2\supset \lambda_1\cup \Lambda}\left(\mathcal{D}'_{\lambda_1,\overline{\lambda_1}\cap \lambda_2}(U,\mathcal{F};E),\mathcal{I}_{p mi}\right)'.$$
 But $\left(\mathcal{D}'_{\lambda_1,\overline{\lambda_1}\cap \lambda_2}(U,\mathcal{F};E),\mathcal{I}_{p mi}\right)'=\left(\mathcal{D}'_{-\lambda_1^c,-\lambda_1^c}(U,\mathcal{K};E^*)\right)$ is independent of $\lambda_2$ and we want to check the equicontinuous sets are also independent. But from the continuous inclusions $(\mathcal{D}'_{\lambda_1, \lambda_1}(U,\mathcal{F};E),\mathcal{I}_{p mi}=\mathcal{I}_{iii})\to (\mathcal{D}'_{\lambda_1,\overline{\lambda_1}\cap \lambda_2}(U,\mathcal{F};E),\mathcal{I}_{p mi})\to (\mathcal{D}'_{\lambda_1,\overline{\lambda_1}}(U,\mathcal{F};E),\mathcal{I}_{p mi}=\mathcal{I}_{iii})$ an equicontinuous set with respect to the duality with the last space is equicontinuous for the second, which is equicontinuous for the first, but since the last and the first induce the same equicontinuous sets as explained in lemma \ref{bornologyIdentities} (the identification of topologies written above being checked at the beginning of this proof and in Theorem  \ref{DualityBD}). In any case, the equicontinuous dual is thus $\underrightarrow{\lim}_{ \lambda_1\supset \gamma}\left(\mathcal{D}'_{\lambda_1,\overline{\lambda_1}}(U,\mathcal{F};E),\mathcal{I}_{p mi}\right)'.$ 
  For the general support case, we likewise identify equicontinuous sets in the dual $(\mathcal{D}'_{\gamma,\Lambda}(U,\mathcal{C};E),\mathcal{I}_{ppp})'=\mathcal{D}'_{-\gamma^c,-\gamma^c}(U,(\mathscr{O}_\mathcal{C})^o;E^*)$. Since an (absolutely convex) equicontinuous set $B$ is bounded, it is uniformly supported on some $D\in (\mathscr{O}_\mathcal{C})^o$ and if $K_n$ is an exhaustion of compact and $f_n\in\mathcal{D}(U)$ equal to 1 on a neighborhood of $K_n$, $v\to f_nv$ is continuous 
 $(\mathcal{D}'_{\gamma,\Lambda}(U,\mathcal{F};E),\mathcal{I}_{ppp})\to \mathcal{D}'_{\gamma,\Lambda}(U,\mathcal{C};E),\mathcal{I}_{ppp})$, we thus deduce as above $L_{f_n}(B)$ is equicontinuous in $(\mathcal{D}'_{\gamma,\Lambda}(U,\mathcal{F};E),\mathcal{I}_{ppp})'=\mathcal{D}'_{-\gamma^c,-\gamma^c}(U,\mathcal{K};E^*)$. But we have just seen this means it is bounded in some $\mathcal{D}'_{\Pi_n,\Pi_n}(U,\mathcal{K};E^*)$ with $\Pi_n\subset -\gamma^c$ closed. Especially, $B$ has wave front set uniformly in $\Pi=\cup_n \Pi_n\cap T^*K_n\subset -\gamma^c$ which is also closed (see argument in the first part of the proof). But since $B$ is (absolutely convex) equicontinuous it is also absolutely convex compact in $\mathcal{D}'_{-\gamma^c,-\gamma^c}(U,(\mathscr{O}_\mathcal{C})^o;E^*),$ with its Arens dual topology ( cf. \cite{Schwarz} to check an equicontinuous is compact in the Arens dual).
 Thus $B$ is a completant Banach disk and thus by Valdivia's closed graph Theorem in \cite{ValdiviaQuasiLB} $E_B\to (\mathcal{D}'_{\Pi,\Pi}(U,\mathcal{F};E),\mathcal{I}_{iii})$ is continuous, especially $B$ is bounded in $\mathcal{D}'_{\Pi,\Pi}(U,\mathcal{F};E)$ uniformly supported in $D$ for the above closed cone $\Pi\subset - \gamma^c.$ [One can also reason more elementarily as above in making explicit the boundedness in $\mathcal{D}'_{\Pi_n,\Pi_n}(U,\mathcal{F};E)$].

Moreover, it is obvious there is a continuous map, for $\lambda_i$ as in the projective limits below, $(\mathcal{D}'_{\gamma,\Lambda}(U,\mathcal{F};E),\mathcal{I}_{p p})\to \left(\mathcal{D}'_{\lambda_1,\overline{\lambda_1}\cap \lambda_2}(U,\mathcal{F};E),\mathcal{I}_{pp}=\mathcal{I}_{pm}\right)$ (with equality of topologies by a terminated projective limit)  and thus we have a continuous map we want to  be an isomorphism~:
 $$(\mathcal{D}'_{\gamma,\Lambda}(U,\mathcal{C};E),\mathcal{I}_{p pp})\to\underleftarrow{\lim}_{ \lambda_1\supset \gamma}\underleftarrow{\lim}_{\lambda_2\supset \lambda_1\cup \Lambda}\left(\mathcal{D}'_{\lambda_1,\overline{\lambda_1}\cap \lambda_2}(U,\mathcal{C};E),\mathcal{I}_{pmp}=\mathcal{I}_{ppp}\right).$$
 But in the dual of the later projective limit we can compute the  equicontinuous sets as in the case $\mathcal{C}=\mathcal{F}$, and they are equicontinuous for the duality of the first space and they are exactly elements in the class we found any equicontinuous set of the first space to be. We thus identified the equicontinuous set of the two topologies on the same space separated by the map above on the same dual, this means the map above is an isomorphism. This gives \eqref{Alternativeppp}.
 
\medskip

 To check \eqref{Alternativepip} note both spaces are algebraically the same and have algebraically the same dual  $\mathcal{D}'_{-\gamma^c,-\gamma^c}(U,(\mathscr{O}_\mathcal{C})^o;E^*)=\underrightarrow{\lim}_{\Pi\subset -\gamma^c}\mathcal{D}'_{\Pi,\Pi}(U,(\mathscr{O}_\mathcal{C})^o;E^*).$
 We have to check they induce the same equicontinuous sets. Of course from  \cite[\S 22.7.(5)]{Kothe} a space equicontinuous in the inductive limit is equicontinuous in one of the terms thus in $\mathcal{D}'_{-\gamma^c,-\gamma^c}(U,(\mathscr{O}_\mathcal{C})^o;E^*).$ 
  Conversely, take an equicontinuous set $B$ in the first space, as usual it is uniformly supported in some $D\in (\mathscr{O}_\mathcal{C})^o$ and  for any smooth compactly supported map $f_n$ as before, $L_{f_n}(B)$ is equicontinuous in  $\mathcal{D}'_{-\gamma^c,-\gamma^c}(U,\mathcal{K};E^*)=(\mathcal{D}'_{\gamma,\overline{\gamma}}(U,\mathcal{F};E),\mathcal{I}_{pip})'=\underrightarrow{\lim}_{\Pi\subset -\gamma^c}\mathcal{D}'_{\Pi,\Pi}(U,\mathcal{K};E^*)$ (one uses that the identity we want to check is trivial when $\mathcal{C}=\mathcal{F}$ by definition) where now this is again \cite[\S 22.7.(5)]{Kothe} that identifies the equicontinuous set in the later space. Thus  $L_{f_n}(B)$ bounded in $\mathcal{D}'_{\Pi_n,\Pi_n}(U,\mathcal{F};E^*)$  and reasoning as before to gather $\Pi_n$'s in a closed cone $\Pi$ concludes.
 
Of course, to check \eqref{Alternativeiii} note again  both spaces are algebraically the same and have algebraically the same dual (inductive limit on $\Lambda$ open cone as usual) $\mathcal{D}'_{\gamma,\gamma}(U,\mathcal{C};E)=\underleftarrow{\lim}_{\Lambda\supset \gamma}\mathcal{D}'_{\Lambda,\Lambda}(U,\mathcal{C};E).$
Again,  we have to check they induce the same equicontinuous sets. Of course, an equicontinuous set in the first space is equicontinuous in all the terms of the projective limit, thus it is equicontinuous in the projective limit by \cite[\S 22.7.(4)]{Kothe}.
 Conversely a set $B$ equicontinuous in any $\mathcal{D}'_{\Lambda,\Lambda}(U,\mathcal{C};E)$
 is uniformly supported in some $C\in \mathcal{C}$ and moreover   bounded in $(\mathcal{D}'_{\Pi_\Lambda,\Pi_\Lambda}(U,\mathcal{F};E),\mathcal{I}_H)$ for some closed cone $\Pi_\Lambda\subset \Lambda$ using the computation of equicontinuous sets at the beginning of the proof. But consider the closed cone intersection of closed cones $\Pi=\bigcap_{\Lambda\supset \gamma}\Pi_\Lambda\subset \bigcap_{\Lambda\supset \gamma}\Lambda=\gamma,$  $B$ is also bounded in $(\mathcal{D}'_{\Pi,\Pi}(U,\mathcal{F};E),\mathcal{I}_H)$ by the definition of the normal topology (for, any $\text{supp}(f)\times V\subset \Pi^c$ can be cut by compactness in finitely many pieces of the same form in  some $\Pi_\Lambda^c$, alternatively one can again use completant Banach disk and Valdivia's closed graph theorem). But sets in $\mathcal{D}'_{\Pi,\Pi}(U,\mathcal{C};E)$ equicontinuous for the duality of this space with $(\mathcal{D}'_{-\Pi^c,-\Pi^c}(U,(\mathscr{O}_\mathcal{C})^o;E^*),\mathcal{I}_{iii})$ are exactly the bounded sets in $(\mathcal{D}'_{\Pi,\Pi}(U,\mathcal{C};E),\mathcal{I}_{iii})$ by lemma \ref{bornologyIdentities} and thus are exactly those of the form of $B$. But from the continuous map $(\mathcal{D}'_{\lambda,\lambda}(U,(\mathscr{O}_\mathcal{C})^o;E^*),\mathcal{I}_{iii})\to(\mathcal{D}'_{-\Pi^c,-\Pi^c}(U,(\mathscr{O}_\mathcal{C})^o;E^*),\mathcal{I}_{iii})$ one concludes $B$ is equicontinuous in the dual of $(\mathcal{D}'_{\lambda,\lambda}(U,(\mathscr{O}_\mathcal{C})^o;E^*),\mathcal{I}_{iii})$ concluding the proof of the claim. Now the identification of topologies $\mathcal{I}_{iii}=\mathcal{I}_{ppp}$ follows from the computation of duals and the previous identification of equicontinuous sets in those duals in the proofs of \eqref{Alternativeiii} and \eqref{Alternativeppp}.

\medskip

For the identification \eqref{Alternativei}, the continuous dense range inclusion : $(\mathcal{D}'_{\lambda,\overline{\lambda}}(U,\mathcal{F};E),\mathcal{I}_{iii})\to\underleftarrow{\lim}_n\underrightarrow{\lim}_{\Pi\subset \lambda}(\mathcal{D}'_{n,\Pi}(U,\mathcal{F};E))$ is obvious to obtain. Thus to identify the duals with $(\mathcal{D}'_{-\lambda^c,-{\lambda}^c}(U,\mathcal{K};E^*)$ it suffices to show this defines  continuous linear functionals on the right hand side. This is similar to the proof for the left hand side in Theorem \ref{duals}. To identify the topologies, it only remains to identify equicontinuous sets in the duals, and the only non-obvious direction is to prove that if $B$ is equicontinuous in the dual of the left hand side, it is in the dual of the right hand side. But we know $B$ is uniformly supported in a compact set, and bounded in $\mathcal{D}'_{\Gamma,\Gamma}(U,\mathcal{F};E^*)$ for some $\Gamma\subset -\lambda^c$ from the beginning of the proof. From \cite[\S22.7.(4,5)]{Kothe}, it suffices to show it is equicontinuous for some $n$, for every $\Pi\subset \lambda$ in $(\mathcal{D}'_{n,\Pi}(U,\mathcal{F};E^*))'.$ Take $n=m+d+1$ 
 with $m$ the degree of polynomial boundedness of $\mathcal{F}(M_{f_i}(v)),v\in B,i\in I.$ 
Taking $\Pi=\Gamma_{m+d+1}\subset -\Gamma^c$ in the proof of equation \eqref{BDcountinuous} and using the note after it about uniformity in $v$, we conclude.
\end{proof}
\subsection{Quasi-LB representations.}
Our next goal will be to identify, in the next subsection, the remaining Arens/Mackey duals that are missing in our computations. As a technical result to obtain compactness, we will use the notion of quasi-(LB) spaces (this will give results of independent interest anyway).

Recall that for quasi-(LB)-spaces, we refer to \cite{ValdiviaQuasiLB} and for the (dually related) class $\mathfrak{G}$ we refer to \cite{CascalesOrihuela}.
Our 
last identification in equation 
\eqref{Alternativei} of the previous lemma will be motivated by the stability properties of these spaces.
 We start by a general lemma that explain part of this remark. Recall that on $C=\N^{\N}$ the (partial) order is the pointwise ordering induced by the usual order on $\N$. The following lemma that has not been pointed out in \cite{ValdiviaQuasiLB} shows quasi-(LB) spaces have more inductive limit stability than projective limit stability since certain uncountable inductive limits are allowed.

\begin{lemma}\label{QLBlemma}
 Let $(F_{(n,m)})_{(n,m)\in \N^{2}}$ quasi-LB spaces, $F_{(n,m)}\subset F$ with continuous injections for each $(n,m)$ with $F$ an Hausdorff locally convex space, $F_{(n,n_1)}\subset F_{(n,n_2)}$ for $n_1\leq n_2.$ Define for $\alpha \in\N^{\N},$ $G_\alpha=\bigcap_{n\in\N}F_{(n,\alpha(n))}$ with projective limit topology so that one gets continuous injections $G_\alpha\subset G_\beta$ for $\alpha\leq \beta$. Then $\underrightarrow{\lim}_{\alpha\in \N^{\N}}G_{\alpha}$ is a quasi-LB space.
\end{lemma}
\begin{proof}
Let $(A_{(n,m),\alpha})_{\alpha\in \N^{\N}}$ a quasi-LB representation of $F_{(n,m)}$ for which we can assume $A_{(n,n_1),\alpha}\subset A_{(n,n_2),\alpha}$ for $n_1\leq n_2$ (replacing $A_{(n,n_1),\alpha}$ by the absolutely convex hull of $A_{(n,k),\alpha},k\leq n_1$ which is still a Banach disk
). Let $g=g_1\cup g_2:\N\cup \N^2\to \N$ a bijection.
 Let for $\alpha\in \N^{\N}$ then $\alpha\circ g_1\in \N^{\N}, \alpha\circ g_2\in \N^{\N^2}.$ 
Define $B_\alpha=\bigcap_{n\in\N}A_{(n,\alpha\circ g_1(n)),\alpha\circ g_2(n,.)}$ which is a Banach disk in $G_{\alpha\circ g_1}.$ 
Note that if $\alpha\leq \alpha'$ 
then $\alpha\circ g_1\leq \alpha'\circ g_1,\alpha\circ g_2(n,.) \leq \alpha'\circ g_2(n,.)$ so that $A_{(n,\alpha\circ g_1(n)),\alpha\circ g_2(n,.)}\subset A_{(n,\alpha'\circ g_1(n)),\alpha\circ g_2(n,.)}\subset A_{(n,\alpha'\circ g_1(n)),\alpha'\circ g_2(n,.)}$ by the definition of our quasi LB representations. Thus, we have 
 $B_\alpha\subset  B_{\alpha'}.$ But, using distributivity and the quasi-LB representation property for $\supset$ : $$\bigcup_{\alpha}B_\alpha=\bigcap_{n\in\N}\bigcup_{\alpha}A_{(n,\alpha\circ g_1(n)),\alpha\circ g_2(n,.)}\supset \bigcup_{\alpha}\bigcap_{n\in\N}F_{(n,\alpha\circ g_1(n))}=\bigcup_{\alpha}G_{\alpha\circ g_1}$$ concluding the proof that this is indeed a quasi-LB representation.
\end{proof}

\begin{proposition}\label{bornologyIdentitiesQLB}With the notation of the previous corollary, especially for $\gamma$ a closed cone, as soon as either $\mathcal{C}$ or $(\mathscr{O}_\mathcal{C})^{o}$ is countably generated, $(\mathcal{D}'_{\gamma,\gamma}(U,\mathcal{C};E), \mathcal{I}_{iii})$ and its bornologification are quasi-LB-spaces of class $\mathfrak{G}$ and so are $(\mathcal{D}'_{\lambda,\overline{\lambda}}(U,(\mathscr{O}_\mathcal{C})^{o};E^*),\mathcal{I}_{iii}),(\mathcal{D}'_{\lambda,{\lambda}}(U,(\mathscr{O}_\mathcal{C})^{o};E^*),\mathcal{I}_{iii})$.
Moreover, $\mathcal{D}'_{\lambda,\overline{\lambda}}(U,(\mathscr{O}_\mathcal{C})^{o};E^*),\mathcal{D}'_{\lambda,{\lambda}}(U,(\mathscr{O}_\mathcal{C})^{o};E^*)$ are also  quasi-LB spaces when $\lambda$ is a $G_\delta$ -cone for $\mathcal{I}_{iii}$ if $(\mathscr{O}_\mathcal{C})^{o}$ is countably generated and for $\mathcal{I}_{iip}$ if $\mathcal{C}$ countably generated.
\end{proposition}
 \begin{proof}
 
Since a space and its bornologification have the same Banach disks the statement about the quasi-LB bornologification follows from the one in theorem \ref{DualityBD}. Then using \cite[corollary 2.2 (i)]{CascalesOrihuela}, since $(\mathcal{D}'_{\lambda,\overline{\lambda}}(U,(\mathscr{O}_\mathcal{C})^{o};E^*),\mathcal{I}_{iii}),(\mathcal{D}'_{\lambda,{\lambda}}(U,(\mathscr{O}_\mathcal{C})^{o};E^*),\mathcal{I}_{iii}) $
 are barrelled and their strong dual is the above bornologification which is quasi-LB, one concludes they belong to the class $\mathfrak{G}$ (for the computation of the strong dual, one uses the density proven below in proposition \ref{bornologyIdentitiesArens} to show the strong topology induced by the completion is the same). Unfortunately one cannot use directly the same result to get $(\mathcal{D}'_{\gamma,\gamma}(U,\mathcal{C};E),\mathcal{I}_{iii})$ of class $\mathfrak{G}$ since it is not barrelled, but we can still get it from the proofs. Indeed, the equicontinuous sets in its dual are described in the proof of \eqref{Alternativeiii} and the inductive limit is equivalent to a countable one in case $(\mathscr{O}_\mathcal{C})^{o}$ is countably generated. We see from \cite[Prop 4]{ValdiviaQuasiLB} $\left(\mathcal{D}'_{ \Pi,\Pi}(U,\{F\in \mathcal{F}, F\subset C\};E\right)$ is quasi-LB as a closed subspace of a quasi-LB space. But from the proof of the quasi-LB representation of a countable inductive limit in \cite[Prop 1,3]{ValdiviaQuasiLB}, the Banach disks are taken inside the terms of the inductive limit, thus here, from our characterization of equicontinuous sets, they are equicontinuous, and even if there are bounded sets in the dual which are not equicontinuous, the quasi-LB representation is made of equicontinuous sets, giving $(\mathcal{D}'_{\gamma,\gamma}(U,\mathcal{C};E),\mathcal{I}_{iii})$ of class $\mathfrak{G}$ in this case. In the case $\mathcal{C}$ is countably generated, this space is a countable (strict) inductive limit of closed subspaces of $(\mathcal{D}'_{\gamma,\gamma}(U,\mathcal{F};E),\mathcal{I}_{iii})$ which is treated before, thus by the stability properties (countable direct sums, products, subspaces and quotient by closed subspaces) of the class $\mathfrak{G}$ given in \cite[Prop 4,5,6,7]{CascalesOrihuela} we get also this case.

If one proves $(\mathcal{D}'_{\lambda,\overline{\lambda}}(U,(\mathscr{O}_\mathcal{C})^{o};E^*),\mathcal{I}_{iii})$ is a quasi-LB space, then $(\mathcal{D}'_{\gamma,\gamma}(U,\mathcal{C};E), \mathcal{I}_{iii}^{born})$, which has this space as a strong dual for $\lambda$ open and is ultrabornological thus barrelled, will thus be of class $\mathfrak{G}$ by \cite[corollary 2.2 (i)]{CascalesOrihuela} again. When $\lambda$ is a $G_\delta$, write $\lambda^c=\cup_n\Gamma_n$ for $\Gamma_n$ closed cones with $\{(x,\xi)\in\Gamma_n, |\xi|=1\}$ compact. 
Fix for each $n$ a decreasing sequence $\Lambda_{n,m}$ of open cones with $\cap_m \Lambda_{n,m}=\Gamma_n$ as in lemma \ref{openclosed} such that for any open cone $\Lambda\supset \Gamma_n$ we have $\Lambda\supset \Lambda_{n,m}$ for some $m$. Then we claim that $$(\mathcal{D}'_{\lambda,\overline{\lambda}}(U,\mathcal{F};E^*),\mathcal{I}_{iii})\simeq\underleftarrow{\lim}_n\underrightarrow{\lim}_{\alpha\in \N^{\N}}\mathcal{D}'_{n,\Pi(\alpha)}(U,\mathcal{F};E^*),$$

where $\Pi(\alpha)=\cap_{n\in\N} \Lambda_{n,\alpha(n)}^c.$ From \eqref{Alternativei} $\underleftarrow{\lim}_n\underrightarrow{\lim}_{\alpha\in \N^{\N}}\mathcal{D}'_{n,\Pi(\alpha)}(U,\mathcal{F};E^*)\to (\mathcal{D}'_{\lambda,\overline{\lambda}}(U,\mathcal{F};E^*),\mathcal{I}_{iii})$ is obvious from a subinductive limit. To build the converse map note that for $\Pi\subset \lambda$, $\Pi^c\supset \lambda^c\supset \Gamma_n$, thus, since $\Pi^c$ is open, there is $\alpha(n)$ with $\Pi^c\supset \Lambda_{n,\alpha(n)}$, i.e. $\Pi\subset \Pi(\alpha)$, giving the cofinality of the subinductive limit and thus the isomorphism. Note this is where we see why we said before the proof of lemma \ref{openclosed} that what really matters is the equivalence of the general inductive limit to a canonical one, this one being uncountable, but of a form adapted to quasi-LB representations.

Now note that if we call for $k$ fixed $F_{n,m}=\mathcal{D}'_{k,\Lambda_{n,m}^c}(U,\mathcal{F};E^*)$, the assumptions of lemma \ref{QLBlemma} are satisfied, since $\mathcal{D}'_{k,\Lambda_{n,m}^c}(U,\mathcal{F};E^*)$ is a closed subspace of a countable product formed by $\mathcal{D'}(U;E^*)$ known to be quasi-LB and a countable family of Banach spaces (using we can reduce to countably many seminorms as in \cite{BrouderDabrowski}).
Moreover, since $\mathcal{D}'_{k,\Pi(\alpha)}(U,\mathcal{F};E^*)=G_\alpha$ with the notation of the lemma, one gets $(\mathcal{D}'_{\lambda,\overline{\lambda}}(U,\mathcal{F};E^*),\mathcal{I}_{iii})$ is a quasi-LB-space (using also again \cite[Prop 2,4]{ValdiviaQuasiLB} to get a countable projective limit of quasi-LB spaces is quasi-LB). For general support, one takes if $(\mathscr{O}_\mathcal{C})^o$ countably generated, first closed subspaces and countable inductive limits, and if $\mathcal{C}$ countably generated, then countable projective limits, so that in all cases, one remains quasi-LB spaces.
The case $\mathcal{D}'_{\lambda,{\lambda}}(U,(\mathscr{O}_\mathcal{C})^{o};E^*)$ is similar starting, instead, from \eqref{Alternativeiii},  which gives a similar inductive limit :$$(\mathcal{D}'_{\lambda,{\lambda}}(U,\mathcal{F};E^*),\mathcal{I}_{iii})\simeq\underrightarrow{\lim}_{\alpha\in \N^{\N}}\mathcal{D}'_{\Pi(\alpha),\Pi(\alpha)}(U,\mathcal{F};E^*)$$ using this time $F_{n,m}=\mathcal{D}'_{\Lambda_{n,m}^c,\Lambda_{n,m}^c}(U,\mathcal{F};E^*).$
\end{proof}
\subsection{Remaining Mackey duals of non-complete spaces controlled by open cones}
Property $\mathfrak{G}$ will be crucial for computations of Mackey duals but we will also need the following technical lemma with an argument in the spirit of our proof of the non completeness of $(\mathcal{D}'_{\lambda,{\lambda}}(U,(\mathscr{O}_\mathcal{C})^{o}),\mathcal{I}_{iii}) $ in \cite{BrouderDabrowski}.
We will use it several times later, this will be actually the only source of non-trivial computations of Mackey duals, namely not coming from being a Montel space.

We want to produce lots of series like $S$ bellow that converge only in a completion in a systematic way so that some absolutely convex hulls that should contain such limits can not be compact before taking the completion.

\begin{lemma}\label{SequenceWF}Assume $U\subset \R^d$ open.
Let $v_n$ a sequence of distributions such that $P_{k_n,V_n,f_n}(v_n)=\infty$ for all $n$ for $\text{supp}(f_n)\times V_n$ a sequence of closed disjoint cones. Assume $v_k\to 0$ in $\mathcal{D}'(U)$, then there exists a sequence $\lambda_n$ with $\sum_n|\lambda_n|<\infty$ such that the convergent series of distributions $S=\sum_n\lambda_nv_n$ satisfies $P_{k_n,V_n,f_n}(S)=\infty$ for all $n$.

\end{lemma}

\begin{proof}
The convergence of $S$ is assured by $v_k\to 0$ and the choice of $\lambda_k$. We look for $\lambda_n$ complex but the real case is similar. The argument producing $\lambda_n$'s will be slightly probabilistic. We will however only need crude volume bounds starting from $P_{k_n,V_n,f_n}(v_n)=\infty$ saying that the space of parameters $\lambda_n$ such that $P_{k_n,V_n,f_n}(\sum_{n\leq N}\lambda_n v_n)<\infty$ has a smaller dimension and thus most of the choices for $\lambda_n$'s give large values. We then only have to cook-up the conditions on the point where we stop the series $N$ for each estimate so that we can take care of all the conditions either inductively or probabilistically.
 
Technically, we will build inductively $\lambda_n$ 
and various witnesses of the infiniteness above enabling to get the desired limit. More precisely, we will build inductively $\epsilon_n\in ]0,1/2^n[$ and $B_n\subset \prod_{k=1}^n([-\epsilon_k,\epsilon_k]+i[-\epsilon_k,\epsilon_k])$ (a constraint on $(\lambda_1,...,\lambda_n)$) 
with $Leb(B_n) \geq (1-\frac{12}{\pi^2n^2})4^n\prod_{k=1}^n \epsilon_k^2$ ($Leb$ is Lebesgue measure on $\C^n=\R^{2n}$) and a sequence $\Xi=(\xi_{k,m})_{m\geq k\geq 1}$ of witnesses with $\xi_{k,m}\in V_k$ and at the $n$-th induction step, we expect to build  the finite triangle $\Xi_n=(\xi_{k,m})_{n\geq m\geq k\geq 1}.$

We start with any $\xi_{1,1}\in V_1$ with $(1+|\xi_{1,1}|)^{k_1}|\mathcal{F}(v_1f_1)(\xi_{1,1})|\geq 1$ any $\epsilon_1=1/3,$  $B_1=]-\epsilon_1,\epsilon_1[+i]-\epsilon_1,\epsilon_1[.$

At the induction step $n\geq 2$, we take $1/\sqrt{2}\epsilon_n=(2^n+1)(1+\max\{\frac{(1+|\xi_{l,m}|)^{k_l}}{m!}|\mathcal{F}(v_nf_l)(\xi_{l,m})|, n-1\geq m\geq l\geq 1\})$ so that especially 
$\sqrt{2}\epsilon_n(1+|\xi_{l,m}|)^{k_l}|\mathcal{F}(v_nf_l)(\xi_{l,m})|\leq \frac{m!}{(2^n+1)},$ and thus assuming $ \lambda_n\in]-\epsilon_n,\epsilon_n[+i]-\epsilon_n,\epsilon_n[$ thus  $|\lambda_n|\leq \sqrt{2}\epsilon_n$ we will have (once all choices made according to this rule) for $m\geq l$: $$(1+|\xi_{l,m}|)^{k_l}|\mathcal{F}([\sum_{n>m}\lambda_nv_n]f_l)(\xi_{l,m})|\leq m! \sum_{n>m}\frac{1}{2^n}=\frac{m!}{2^m}.$$

Note also that $\sqrt{2}\epsilon_n\leq 1/2^n$ insuring the convergence of $\sum_n|\lambda_n|.$

The choice of $B_n$ and $\Xi$ will now control the other parts of the sum. At this induction step, we have to build $\xi_{1,n},...,\xi_{n,n}$ and $B_n$.

Consider for $l\leq n$ the vector space $E_{l,n}=\{\lambda=(\lambda_1,...,\lambda_n)\subset \C^n : \sup_{\xi\in V_l}(1+|\xi|)^{k_l}|\mathcal{F}([\sum_{k=1}^n\lambda_kv_k]f_l)(\xi)|<\infty\}$
Since the basis vector with the $l$-th digit $e_l\not\in E_{l,n}$ by assumption on $v_l$, $E_{l,n}$ has dimension $d(l,n)\leq n-1$.

Consider $\lambda^{(1)}(\xi)$ reaching $\inf_{\lambda=(\lambda_1,...,\lambda_n)\subset E_{l,n}^\perp, |\lambda|\geq 1} (1+|\xi|)^{k_l}|\mathcal{F}([\sum_{k=1}^n\lambda_kv_k]f_l)(\xi)|$
which is the same by linearity as the inf over $|\lambda|=1$ which thus exists by compactness (and could of course be computed quite explicitly...).

Then define inductively $\lambda^{(k)}(\xi)$ reaching $$\inf_{\lambda=(\lambda_1,...,\lambda_n)\subset Vect(E_{l,n},\lambda^{(1)}(\xi),...,\lambda^{(k-1)}(\xi))^\perp, |\lambda|\geq 1} (1+|\xi|)^{k_l}|\mathcal{F}([\sum_{k=1}^n\lambda_kv_k]f_l)(\xi)|.$$
We continue until $K(n,l)=n-d(l,n)$. By definition $(\lambda^{(1)}(\xi),...,\lambda^{(K(n,l))}(\xi))$ form an orthonormal basis with $(1+|\xi|)^{k_l}|\mathcal{F}([\sum_{k=1}^n\lambda_k^{(m)}(\xi)v_k]f_l)(\xi)|$ non-decreasing in $m$. Assume a second for contradiction that $\sup_{\xi\in V_{l}}(1+|\xi|)^{k_l}|\mathcal{F}([\sum_{k=1}^n\lambda_k^{(K(n,l))}(\xi)v_k]f_l)(\xi)|<\infty$ by increasingness, $\sup_{\xi\in V_{l}}(1+|\xi|)^{k_l}|\mathcal{F}([\sum_{k=1}^n\lambda_k^{(m)}(\xi)v_k]f_l)(\xi)|<\infty$ finite for all $m$.
Then writing $e_l=u_1(\xi)\lambda^{(1)}(\xi)+...+u_{K(n,l)}(\xi)\lambda^{(K(n,l))}(\xi)+\lambda, \lambda=Proj_{E_{n,l}}(e_l)$ one would get,  since $\sum_{i=1}^{K(n,l)}|u_i(\xi)|^2\leq 1$ and thus $|u_i(\xi)|\leq 1$,
$$\sup_{\xi\in V_{l}}(1+|\xi|)^{k_l}|\mathcal{F}(v_lf_l)(\xi)|\leq \sum_{m=1}^{K(n,l)}\sup_{\xi\in V_{l}}(1+|\xi|)^{k_l}|\mathcal{F}([\sum_{k=1}^n\lambda_k^{(m)}(\xi)v_k]f_l)(\xi)|+P_{V_l,k_l,f_l}(\sum_{k=1}^n\lambda_kv_k)<\infty$$ contradicting the assumption. We can thus fix the smallest index $M(n,l)\in[1,K(n,l)]$ with \begin{align*}&\sup_{\xi\in V_{l}}(1+|\xi|)^{k_l}|\mathcal{F}([\sum_{k=1}^n\lambda_k^{(M+1)}(\xi)v_k]f_l)(\xi)|\\&=\sup_{\xi\in V_{l}}\inf_{\lambda=(\lambda_1,...,\lambda_n)\subset Vect(E_{l,n},\lambda^{(1)}(\xi),...,\lambda^{(M)}(\xi))^\perp, |\lambda|\geq 1} (1+|\xi|)^{k_l}|\mathcal{F}([\sum_{k=1}^n\lambda_kv_k]f_l)(\xi)|=\infty.\end{align*} As a consequence if $P_M=Proj_{Vect(E_{l,n},\lambda^{(1)}(\xi),...,\lambda^{(M)}(\xi))^\perp}$
and if $\lambda=P_M(\lambda)+(1-P_M)(\lambda)=\lambda^{(a)}(\xi)+\lambda^{(b)}(\xi)$
then \begin{align*}\inf_{\lambda, |P_M(\lambda)|\geq 1, |\lambda|\leq C}(1+|\xi|)^{k_l}&|\mathcal{F}([\sum_{k=1}^n\lambda_kv_k]f_l)(\xi)|\geq\inf_{\lambda^{(a)}=P_M(\lambda^{(a)}), |(\lambda^{(a)})|\geq 1} (1+|\xi|)^{k_l}|\mathcal{F}([\sum_{k=1}^n\lambda_k^{(a)}v_k]f_l)(\xi)|\\&-\sup_{\lambda^{(b)}=(1-P_M)(\lambda^{(b)}), |(\lambda^{(b)})|\leq C}(1+|\xi|)^{k_l}|\mathcal{F}([\sum_{k=1}^n\lambda_k^{(b)}v_k]f_l)(\xi)|.\end{align*}
Since by definition of $M(n,l)$ we have  $$\sup_{\xi\in V_l}\sup_{\lambda^{(b)}=(1-P_M)(\lambda^{(b)}), |(\lambda^{(b)})|\leq C}(1+|\xi|)^{k_l}|\mathcal{F}([\sum_{k=1}^n\lambda_k^{(b)}v_k]f_l)(\xi)|<\infty,$$
 we deduce $\sup_{\xi\in V_l}\inf_{\lambda, |P_M(\lambda)|\geq 1, |\lambda|\leq C}(1+|\xi|)^{k_l}|\mathcal{F}([\sum_{k=1}^n\lambda_kv_k]f_l)(\xi)|=\infty$ for any $C$, for instance $C=\sqrt{2}\max(\epsilon_1,...,\epsilon_n)/\epsilon.$

Thus we define $\xi_{l,n}\in V_l$ such that $$\inf_{\lambda=(\lambda_1,...,\lambda_n),|P_M(\lambda)|\geq 1, |\lambda|\leq C} (1+|\xi_{l,n}|)^{k_l}|\mathcal{F}([\sum_{k=1}^n\lambda_kv_k]f_l)(\xi)|\geq \frac{n!}{\epsilon}$$
 for an $\epsilon$ to be fixed soon.
 
Consider now $A_{l,n}=\{\lambda=(\lambda_1,...,\lambda_n),|P_M(\lambda)|\geq 1, |\Im(\lambda_i)|,|\Re(\lambda_i)|\leq \epsilon_i/\epsilon\}$ we want to estimate its Lebesgue measure independently of $\xi=\xi_{l,n}$ (appearing in the dependence in $P_M$) thus using only dimensionality. Let $B_{l,n}=\epsilon A_{l,n}$, $Leb (B_{l,n})= (2\epsilon_1)^2...(2\epsilon_n)^2 - Leb(\lambda=(\lambda_1,...,\lambda_n),|P_M(\lambda)|< \epsilon, |\Im(\lambda_i)|,|\Re(\lambda_i)|\leq \epsilon_i)$. If we call $C_n$ a bound on the maximum of the surface of a (complex) hyperplane cutting $[-\epsilon_1-1,\epsilon_1+1]^2\times ...\times [-\epsilon_n-1,\epsilon_n+1]^2$ (which exists by compactness) since $\{\lambda=(\lambda_1,...,\lambda_n),|P_M(\lambda)|\leq \epsilon, |\Im(\lambda_i)|,|\Re(\lambda_i)|\leq \epsilon_i\}$
 is always included in the product of such an intersection of an hyperplane and a product of two dimensional ball of radius $\epsilon<1/10$ (even if there are boundary effects taken care by the enlarged cubes) we have thus $Leb(B_{l,n})\geq 4^n\epsilon_1^2...\epsilon_n^2 -\epsilon^2 C_n.$
 We thus choose $\epsilon^2<\frac{12}{\pi^2n^3}4^n\prod_{k=1}^n \epsilon_k^2/ C_n$. We thus obtained for $B_n=\cap_{l=1}^n B_{l,n},$ $Leb(B_n) \geq (1-n\frac{12}{\pi^2n^3})4^n\prod_{k=1}^n \epsilon_k^2$ and moreover $\forall \lambda\in B_n$ we have $(1+|\xi_{l,n}|)^{k_l}|\mathcal{F}([\sum_{k=1}^n\lambda_kv_k]f_l)(\xi_{l,n})|\geq n!.$

This concludes our inductive construction. We now let $B=\cap_{n=1}^\infty B_n\times\prod_{k=n+1}^\infty[-\epsilon_{k},\epsilon_k]+i[-\epsilon_{k},\epsilon_k].$
Putting on $\Omega=\prod_{k=1}^\infty[-\epsilon_{k},\epsilon_k]+i[-\epsilon_{k},\epsilon_k]$ the product of uniform probability measures, one obtains $P(B)\geq 1-\sum_{n=1}^\infty P(B_n^c)\geq 1-\sum_{n=1}^\infty\frac{12}{\pi^2n^2}=\frac{1}{2}>0$ so that $B$ is not empty. Moreover, for any $\lambda\in B$ one gets $(1+|\xi_{l,n}|)^{k_l}|\mathcal{F}([\sum_{k=1}^\infty\lambda_kv_k]f_l)(\xi_{l,n})|\geq n!-n!/2^n\geq n!/2.$
Thus for $S=\sum_{k=1}^\infty\lambda_kv_k$ since all $\xi_{l,n}\in V_l$ we have $P_{k_l,V_l,f_l}(S)\geq sup_n n!/2=\infty$.
\end{proof}

\begin{proposition}\label{bornologyIdentitiesArens}With the notation of the previous corollary, especially for $\gamma$ a closed cone  $(\mathcal{D}'_{\gamma,\gamma}(U,\mathcal{C};E), \mathcal{I}_{iii})$ is also the Mackey and Arens dual of $(\mathcal{D}'_{\lambda,\lambda}(U,(\mathscr{O}_\mathcal{C})^{o};E^*)),\mathcal{I}_{iii}).$  
Moreover, 
we can identify 
 the bornologification of $\mathcal{D}'_{-\lambda^c,-\lambda^c}(U,\mathcal{C};E)$,as $\mathcal{I}_{iii}^{born}=\mathcal{I}_{ibi}=\mathcal{I}_{pbp}.$

\end{proposition}
 \begin{proof}

Let us summarize the main idea of the proof. For each statement, we will have to compute a Mackey/Arens dual in identifying absolutely convex (weakly) compact sets to equicontinuous sets from the dual so that uniform convergence on the former is the same as on the latter, and Mackey/Arens topologies on the dual will be the original topology we computed equicontinuous sets with. To reach this goal, we will use our description of equicontinuous sets to get a contradiction in assuming the contrary. After some reductions (using a closed graph theorem for quasi-LB spaces in a mild way), we will have to prove an absolutely convex (weakly) compact set has to be uniformly controlled by a support condition and a closed cone for its (dual) wave front set. By contradiction we have a sequence with a wave front set condition implying it goes closer to the completion. This is the point where we use class $\mathfrak{G}$ crucially to get metrizability of compact sets and converging extract subsequences (and not only subnets) to apply our previous lemma (which uses crucially countability of the index set of a sequence in the probabilistic argument). This lemma gives a point in the absolutely convex hull, built from a series with coefficients scalar multiple of  our sequence that is in the completion and not in the space, thus the absolutely convex set we started with could not be compact, it is even not closed in the completion. This is our desired contradiction (in the spirit of our non-completeness proof in \cite{BrouderDabrowski}). Let us make all of this more explicit now.

To compute the Mackey/Arens dual, let us fix an absolutely convex (weakly) compact set $K$ in $(\mathcal{D}'_{\lambda,{\lambda}}(U,(\mathscr{O}_\mathcal{C})^{o};E^*),\mathcal{I}_{iii}) $, from the regular strict inductive limit, we know that $K$ is uniformly supported in some $D\in (\mathscr{O}_\mathcal{C})^{o}.$ If we prove that $K$ is included in $(\mathcal{D}'_{\Pi_n,\Pi_n}(U,\{F\in \mathcal{F}, F\subset D\};E^*)$, this will be enough to get it is bounded there, since from weak compactness, $K$ is a Banach disk (since it is $\ell^1$-disked) thus the identity map $E_K\to (\mathcal{D}'_{\Pi_n,\Pi_n}(U,\{F\in \mathcal{F}, F\subset D\};E^*)$ which is obviously closed (from the continuous inclusion in the Hausdorff space $\mathcal{D}'(U)$) and since $\mathcal{D}'_{\Pi_n,\Pi_n}(U,\{F\in \mathcal{F}, F\subset D\};E^*)$ is a quasi-LB space (as a closed subspace of the quasi-LB space $\mathcal{D}'_{\Pi_n,\Pi_n}(U, \mathcal{F};E^*)$) we get this identity map is continuous (apply for instance the closed graph theorem of De Wilde \cite[\S 35.2.(2)]{Kothe2} since a quasi-LB space is strictly webbed \cite{ValdiviaQuasiLB} 
 and a Banach space is ultrabornological), and thus the image of the bounded set $K$ has to be bounded in $\mathcal{D}'_{\Pi_n,\Pi_n}(U,\{F\in \mathcal{F}, F\subset D\};E^*)$. We thus conclude then that $K$ is actually an equicontinuous set in $(\mathcal{D}'_{\Lambda,{\Lambda}}(U,(\mathscr{O}_\mathcal{C})^{o};E),\mathcal{I}_{iii}) $ from our description of those sets and thus the Mackey and Arens topologies are the original normal topology.

Moreover, if for any $f\in \mathcal{D}(U)$ equal to 1 on a large compact $L_f(K)$ is included in some $\mathcal{D}'_{\Pi_f,\Pi_f}(U,\{F\in \mathcal{F}, F\subset D\};E^*)$, then letting $f_n$ equal to 1 on an exhaustion of compact $K_n$, and as usual $\Pi=\cup_nT^*K_n\cap \Pi_{f_n}$ then gives a closed cone with $K\subset  \mathcal{D}'_{\Pi,\Pi}(U,\{F\in \mathcal{F}, F\subset D\};E^*)$. We are thus reduced to $D$ compact since $L_f(K)$ is also absolutely convex (weakly) compact.

To complete the proof, it thus remains to get a contradiction if $K$ is not included in any $\mathcal{D}'_{\Pi_n,\Pi_n}(U,\{F\in \mathcal{F}, F\subset D\};E^*)$ with $\Pi_n\nearrow \Lambda$ in which case there is a sequence $v_n\in K$ and a sequence $(x_n,\xi_n)\in \dot{T}^*U$ such that $(x_n,\xi_n)\in WF(v_n)$ $(x_n,\xi_n)\in\Pi_{n+1}-\Pi_n,WF(v_n)\subset \Pi_{n+1}, $ and say $|\xi_n|=1$ $(x_n,\xi_n)\to (x,\xi)\in \lambda^c$ (the convergence of $x_n$ up to extraction coming from the reduction to $D$ compact). We can fix $i$ with $x\in U_i$ and thus assume $x_n\in U_i$ for all $n$.

If we replace $K\subset (\mathcal{D}'_{\lambda,{\lambda}}(U,\mathcal{F};E^*),\mathcal{I}_{iii})$ and (weakly) compact there by the (weak) closure of the absolutely convex hull of $v_n$ (in this space), one gets $K$ is separable  (in its relevant topology  induced by $(\mathcal{D}'_{\lambda,{\lambda}}(U,\mathcal{F};E^*),\mathcal{I}_{iii})$ making it compact), thus metrizable either from \cite[Theorem 2]{CascalesOrihuela} in the compact case, or from \cite[Corollary 1.12]{CascalesOrihuela} in the weakly compact case because                                                                  
$F=(\mathcal{D}'_{\lambda,{\lambda}}(U,\mathcal{F};E^*),\mathcal{I}_{iii}) $ is of class $\mathfrak{G}$ (and in the second case because $K$ is now separable). Thus we can now replace $v_n$ by a converging subsequence converging to $v$ (in $F$ but from the a priori support property also in $K$ with its original topology), and since $\frac{1}{2}(v_n-v)$ is still in $K$ which is absolutely convex, one can assume $v_n\to 0$ (weakly in the weakly compact case, one also uses $v$ has a wave front set in some $\Pi_n$ and thus does not change the wave front set condition for $n$ large enough). It only remains to get a series $S_N=\sum_{n=0}^N \lambda_n v_n$ converging to a point $S$ in the completion. 
Using $M_i,N_i$ of lemma \ref{contuinuousdenseInj}, it even suffices to assume $U=\varphi_i(U_i)$ what we now do.

As soon as $\sum |\lambda_n|<\infty$ it is obvious to get $S_N$ converges weakly in the completion since $v_n$ is weakly convergent thus weakly bounded. Now take a sequence of cones $(\text{supp}(f_n)\times V_n)$ around $(x_n,\xi_n)$ ($x_n\in U_{i_n}$) of diameter going to $0$ (so that any sequence of points taken in this sequence converge to $(x,\xi)$
) such that $P_{k_n,V_n,f_n}(v_n)=\infty.$
Thus for $S=\sum_{k=1}^\infty\lambda_kv_k$ given in lemma \ref{SequenceWF}
, we have $P_{k_l,V_l,f_l}(S)=\infty$ and thus $\text{supp}(f_l)\times V_l\cap WF(S)\neq\emptyset$ and thus by our choices $(x,\xi)\in WF(S)$ contradicting the fact that $S\in \mathcal{D}'_{\lambda,{\lambda}}(U,(\mathscr{O}_\mathcal{C})^{o};E^*)$. 
 This contradiction implies the absolutely convex (weakly) compact $K$ has to be bounded in some $\mathcal{D}'_{\Pi,{\Pi}}(U,(\mathscr{O}_\mathcal{C})^{o};E^*)$ and thus equicontinuous in $\mathcal{D}'_{\lambda,{\lambda}}(U,(\mathscr{O}_\mathcal{C})^{o};E^*)$ via our characterization of equicontinuous sets.

\medskip

Let us now fix an absolutely convex (weakly) compact set $K$ in $(\mathcal{D}'_{\lambda,\overline{\lambda}}(U,(\mathscr{O}_\mathcal{C})^{o};E^*),\mathcal{I}_{iii}) $.
We want to check that for each $n$, there is a closed cone $\Pi_n\subset \lambda$ such that $K$ is included in $\mathcal{D}'_{n,\Pi_n}(U,\mathcal{F};E^*)$. From this we will get as above it is necessarily bounded in this quasi-LB space. We can also assume as before $\mathcal{C}=\mathcal{F}$ i.e. elements of $K$ are supported on a compact.
Assuming for contradiction this is not the case for some order $k$, one gets as above $v_n\in K$ and a sequence $(x_n,\xi_n)\to (x,\xi)\in \lambda^c$ (one uses the distance to a closed cone is continuous) with $(x_n,\xi_n)\in DWF_k(v_n)$. Since $x\in U_i$ we are even reduced as before to $U=\varphi_i(U_i)$.

From the definition of dual wave front set  there exists sequences $V_n, f_n$ with  $\text{supp}(f_n)\times V_n$
disjoint (we first assume disjoint of previously built ones and from the tail of the sequence and then work by induction) tending to $(x,\xi)\in \lambda^c$
with $P_{k,V_n,f_n}(v_n)=\infty.$ 
As before since $(\mathcal{D}'_{\lambda,\overline{\lambda}}(U,\mathcal{F};E^*),\mathcal{I}_{iii}) $ is of class $\mathfrak{G}$ one can extract a sequence $v_n$ and assume its limit is $0$.
Applying lemma \ref{SequenceWF} with the constant sequence $k_n=k$, one gets $S=\sum \lambda_nv_n$ converging in $(\mathcal{D}'_{\lambda,\overline{\lambda}}(U,\mathcal{F};E^*),\mathcal{I}_{iii}) $ and with $P_{k,V_n,f_n}(S)=\infty$. Now if $(x,\xi)$ were in $DWF_k(S)^c$ there would be a neighborhood $U$ and a cone $V$, and from the hypothesis  $V_n\subset V, \text{supp}(f_n)\subset U$ for $n$ large enough. Thus $(x,\xi)\in DWF_k(S)$ contradicting $DWF_k(S)\subset  \lambda$. One thus deduces $K$ bounded in some $\mathcal{D}'_{n,\Pi_n}(U,\mathcal{F};E^*).$

We are now ready to check the identification of the agreeing  Mackey, Arens and strong dual of $(\mathcal{D}'_{\lambda,\overline{\lambda}}(U,(\mathscr{O}_\mathcal{C})^{o};E^*),\mathcal{I}_{iii})$, i.e. to compute $(\mathcal{D}'_{-\lambda^c,-\lambda^c}(U,\mathcal{C}),\mathcal{I}_{iii}^{born})$. Note the identity $\mathcal{I}_{ibi}=\mathcal{I}_{pbp}$ is obvious from the definition and lemma \ref{bornologyIdentitiesEqui} (case of equality $\mathcal{I}_{iii}=\mathcal{I}_{ppp}$).

We already know by definition that $\mathcal{I}_{ibi}$ is stronger than $\mathcal{I}_{iii}$ and have the same bounded sets. From the computation in Theorem \ref{duals}, we also know the space with topology $\mathcal{I}_{ibi}$ has the same continuous and bornological dual, thus to check it is bornological and conclude this is the expected bornologification, it suffices to check it carries its Mackey topology.

We only have to check the topology for $\mathcal{I}_{ibi}$ above is stronger than the Mackey topology of its dual thus equal to it. 
But the stated continuity is easy from \eqref{BDbounded} since the seminorms involved for $u$ are among the seminorms as in the proof of the continuous dual, and the seminorms in $v$ are uniform on a bounded set of $(\mathcal{D}'_{\lambda,\overline{\lambda}}(U,(\mathscr{O}_\mathcal{C})^{o};E^*),\mathcal{I}_{iii})$ by our computation of bounded sets, included in absolutely convex compact sets (by completeness and nuclearity for compactness).

 \end{proof}

\section{Functional analytic properties in the general case}
The first result builds on the choice of definitions of topologies made to obtain nuclearity from the closed and open case.

\begin{proposition}\label{FAGeneral}Let $\gamma\subset \Lambda\subset \overline{\gamma}$ cones  and  $\mathcal{C}=\mathcal{C}^{oo}$ an enlargeable polar family of closed sets in $U$ and let $\lambda=-\gamma^c.$
 $(\mathcal{D}'_{\gamma,\Lambda}(U,\mathcal{C};E),\mathcal{I}_{ppp})$ is nuclear and $(\mathcal{D}'_{\gamma,\overline{\gamma}}(U,\mathcal{C};E),\mathcal{I}_{\alpha}), \alpha=ppp,pmp,pip$ (or even $\alpha=pmi,pii$ when $\mathcal{C}$ is countably generated) is also complete, nuclear thus semireflexive (even completely reflexive) and semi-Montel and, as a consequence, on its dual the strong, Mackey and Arens topologies  coincide and moreover they coincide when $\alpha=pip$ with the following limit on closed cones of bornologifications :\begin{align*}(\mathcal{D}'_{\lambda,\lambda}(U,(\mathscr{O}_\mathcal{C})^o;E^*),\mathcal{I}_{b}):=\underrightarrow{\lim}_{ \Pi\subset \lambda}\left(\mathcal{D}'_{\Pi,\Pi}(U,(\mathscr{O}_\mathcal{C})^o;E^*),\mathcal{I}_{iii}^{born}\right),\end{align*}
 which is thus ultrabornological 
  (and 
 stronger than $\mathcal{I}_{iii}$ 
 with equality in the case $\lambda$ open
 ). Actually, we have $\mathcal{I}_{ppp}=\mathcal{I}_{pmp}=\mathcal{I}_{pip}$ on $\mathcal{D}'_{\gamma,\overline{\gamma}}(U,\mathcal{C};E)$ and more generally, it is the completion of $(\mathcal{D}'_{\gamma,\Lambda}(U,\mathcal{C};E),\mathcal{I}_{ppp}).$ Finally, on $\mathcal{D}'_{\lambda,\lambda}(U,(\mathscr{O}_\mathcal{C})^o;E^*)$ we have $\mathcal{I}_{iii}= \mathcal{I}_{ppp}$, and  $(\mathcal{D}'_{\lambda,\lambda}(U,(\mathscr{O}_\mathcal{C})^o;E^*),\mathcal{I}_{iii})$ is also the  Mackey and Arens dual of $(\mathcal{D}'_{\gamma,\gamma}(U,\mathcal{C};E),\mathcal{I}_{ppp}).$  If moreover, either $\mathcal{C}$ or $(\mathscr{O}_\mathcal{C})^o$ is countably generated and $\lambda$ is a $G_\delta$-cone (i.e. $\gamma$ an $F_\sigma$-cone), $\mathcal{D}'_{\lambda,\lambda}(U,(\mathscr{O}_\mathcal{C})^o;E^*)$ for both $\mathcal{I}_{iii},\mathcal{I}_{b}$ is a quasi-LB space 
  and $(\mathcal{D}'_{\gamma,\Lambda}(U,\mathcal{C};E),\mathcal{I}_{ppp})$ is of class $\mathfrak{G}$, 
  and if, instead, $\lambda$ is an $F_\sigma$-cone (i.e. $\gamma$ a $G_\delta$-cone), $\mathcal{D}'_{\lambda,\lambda}(U,(\mathscr{O}_\mathcal{C})^o;E^*)$ for both $\mathcal{I}_{iii},\mathcal{I}_{b}$ is of class $\mathfrak{G}$ and $(\mathcal{D}'_{\gamma,\Lambda}(U,\mathcal{C};E),\mathcal{I}_{ppp})$ is a quasi LB space for both $\Lambda=\gamma,\overline{\gamma}.$
\end{proposition}
\begin{proof}

The nuclearity of $(\mathcal{D}'_{\gamma,\Lambda}(U,\mathcal{C};E),\mathcal{I}_{ppp})$  follows by taking countable inductive limits and projective limits from the open case in the previous corollary since the building block are the spaces $(\mathcal{D}'_{\Lambda,\overline{\Lambda}}(U,\mathcal{F};E),\mathcal{I}_{iii}).$ 

For $(\mathcal{D}'_{\gamma,\overline{\gamma}}(U,\mathcal{C};E),\mathcal{I}_{\alpha}),\alpha=pmp,pip$, there is no inductive limit over ${\Pi}\subset \overline{\gamma}$, thus since $(\mathcal{D}'_{\Lambda,\overline{\Lambda}}(U,\mathcal{F};E),\mathcal{I}_{iii}),$ is complete nuclear, there are only projective limits, and a countable strict inductive limit on support which is also complete and nuclear. In the countably generated case, the inductive limits on support are equivalent to (strict) countable ones, explaining the variant for inductive topologies.
The agreement of strong, Mackey and Arens topologies is a known result for semi-Montel spaces \cite[p.~235]{Horvath}. We will deduce the completeness in the case $\alpha=ppp$ from the last statement identifying it with the case $\alpha=pmp.$

To compute the Mackey dual of $(\mathcal{D}'_{\gamma,\overline{\gamma}}(U,\mathcal{C};E),\mathcal{I}_{pip})$ we use \cite[\S 22.7.(9) p 294]{Kothe} on \eqref{Alternativepip}. 
Using this (the densities involved in the reduced form assumption being checked by our general density lemma \ref{normality}) and the computation of the Mackey duals in the open case in the previous corollary \ref{FAClosedOpen2} and also lemma \ref{bornologyIdentitiesEqui} for identifying the topologies $\mathcal{I}_{iii}=\mathcal{I}_{pip}$ on appropriate spaces, we get  the stated result identifying the Mackey dual with $\mathcal{I}_b.$

 We now prove the completion of $(\mathcal{D}'_{\gamma,\Lambda}(U,\mathcal{C};E),\mathcal{I}_{ppp})$ is $(\mathcal{D}'_{\gamma,\overline{\gamma}}(U,\mathcal{C};E),\mathcal{I}_{pip})$ (this will imply the equality in the case $\Lambda=\overline{\gamma}$ since the completion map gives then the converse map which was not obvious by definition from $\mathcal{I}_{pip}\to\mathcal{I}_{pmp}\to \mathcal{I}_{ppp}$). Since $(\mathcal{D}'_{\gamma,\Lambda}(U,\mathcal{F};E),\mathcal{I}_{ppp})$ is nuclear we can use again \cite[p.~140 Corollary 3]{HogbeNlendMoscatelli}. Note that most our computations of completions came from this lemma and motivated our management of topologies which are by construction nuclear. We computed in the proof of lemma \ref{bornologyIdentitiesEqui} the equicontinuous bornology on the dual, which is also obvious from formula \eqref{Alternativeppp} which is as usual \cite[\S 22.7.(5)]{Kothe} the inductive limit bornology $\underrightarrow{\lim}_{ \lambda_1\supset \gamma}\underrightarrow{\lim}_{\lambda_2\supset \lambda_1\cup \Lambda}\left(\mathcal{D}'_{\lambda_1,\overline{\lambda_1}\cap \lambda_2}(U,\mathcal{C};E),\mathcal{I}_{p mi}\right)'$
 But $\left(\mathcal{D}'_{\lambda_1,\overline{\lambda_1}\cap \lambda_2}(U,\mathcal{C};E),\mathcal{I}_{p mi}\right)'=\left(\mathcal{D}'_{-\lambda_1^c,-\lambda_1^c}(U,(\mathscr{O}_\mathcal{C})^o;E^*)\right)$  we  checked the equicontinuous sets are also independent of $\lambda_2$. 
 
  In any case, the equicontinuous dual is thus $\underrightarrow{\lim}_{ \lambda_1\supset \gamma}\left(\mathcal{D}'_{\lambda_1,\overline{\lambda_1}}(U,\mathcal{C};E),\mathcal{I}_{p mi}\right)'.$ Again from corollary \ref{FAClosedOpen2}, the bornological dual is then $\underleftarrow{\lim}_{ \lambda_1\supset \gamma}\left(\mathcal{D}'_{\lambda_1,\overline{\lambda_1}}(U,\mathcal{C};E),\mathcal{I}_{iii}\right),$ i.e. nothing but \eqref{Alternativepip} for $\mathcal{I}_{pip}$ in this case on $\mathcal{D}'_{\gamma,\overline{\gamma}}(U,\mathcal{C};E).$ 
 
 It remains to compute the Mackey/Arens duals of \eqref{Alternativeppp} : $$(\mathcal{D}'_{\gamma,\gamma}(U,\mathcal{C};E),\mathcal{I}_{ppp})=\underleftarrow{\lim}_{ \lambda_1\supset \gamma}\left(\mathcal{D}'_{\lambda_1,\lambda_1}(U,\mathcal{C};E),\mathcal{I}_{pmp}=\mathcal{I}_{iii}\right).$$ From \cite[\S 22.7.(9)]{Kothe} and its variant for Arens duals\footnote{which uses only the same behaviour of weakly compact and compact sets for intersection with closed sets, products and the inductive limit is topological since the adjoint map is continuous between Arens duals by the proof of \cite[\S 21.4.(6)]{Kothe} (exchanging $A,A'$ to know from continuity of $A$ that it preserves absolutely convex compact sets).}, it suffices to compute the inductive limit of corresponding duals. But from proposition \ref{bornologyIdentitiesArens}, we computed the Arens/Mackey duals of the terms in the projective limit
 , so that one concludes from the topological isomorphism in \eqref{Alternativeiii}.

 First assume $\lambda$ is a $G_\delta$-cone and either $\mathcal{C}$ or $(\mathscr{O}_\mathcal{C})^o$ countably generated. The fact that $(\mathcal{D}'_{\lambda,\lambda}(U,(\mathscr{O}_\mathcal{C})^o;E^*),\mathcal{I}_{iii})$ is a quasi-LB space is stated in proposition \ref{bornologyIdentitiesQLB} since we now know $\mathcal{I}_{iii}=\mathcal{I}_{iip}.$ For its Mackey dual $(\mathcal{D}'_{\gamma,\gamma}(U,\mathcal{F};E),\mathcal{I}_{ppp})$, in case $\mathcal{F}=\mathcal{C}$ the bounded sets of the quasi-LB structure of its dual are  bounded sets in $\mathcal{D}'_{\Pi(\alpha),\Pi(\alpha)}(U,\mathcal{K};E)$, thus equicontinuous there, and from the computation in the proof of lemma \ref{bornologyIdentitiesEqui}, equicontinuous in the dual of $(\mathcal{D}'_{\gamma,\gamma}(U,\mathcal{F};E),\mathcal{I}_{ppp}).$ Thus, $(\mathcal{D}'_{\gamma,\gamma}(U,\mathcal{F};E),\mathcal{I}_{ppp})$ is of class $\mathfrak{G}.$ By subspaces, countable inductive and projective limit, this is the same for every support condition stated. The case of $\lambda$ intersection of closed and open is easy since then formula \eqref{Alternativeiii} is equivalent to a countable inductive limit and likewise from formula \eqref{Alternativepip} in case $\gamma$ is a union of a closed and an open set.
 
 From \cite[Prop 8]{CascalesOrihuela}, its completion $(\mathcal{D}'_{\gamma,\overline{\gamma}}(U,\mathcal{C};E),\mathcal{I}_{ppp})$ is also of class $\mathfrak{G},$ and also all the spaces with the same completion $(\mathcal{D}'_{\gamma,\Lambda}(U,\mathcal{C};E),\mathcal{I}_{ppp})$ as subspaces.

 Finally, for $(\mathcal{D}'_{\lambda,\lambda}(U,(\mathscr{O}_\mathcal{C})^o;E^*),\mathcal{I}_{b})$ we know it is the Mackey dual of the completion, thus the equicontinuous (thus bounded sets) giving property $\mathfrak{G}$ also give the quasi -LB representation  of this dual.
 
Finally, assume $\lambda$ is a $F_\sigma$-cone instead. From \cite[Prop 1]{ValdiviaQuasiLB} and our proposition \ref{bornologyIdentitiesQLB}, since $\mathcal{I}_{ppp}$ is weaker than both $\mathcal{I}_{iip},\mathcal{I}_{iii}$ which are proved there to be quasi-LB spaces if $\gamma$ is a $G_\delta$, we deduce that  $(\mathcal{D}'_{\gamma,\Lambda}(U,\mathcal{C};E),\mathcal{I}_{ppp})$ is a quasi LB  space for both $\Lambda=\gamma,\overline{\gamma}.$
Now  $(\mathcal{D}'_{\lambda,\lambda}(U,(\mathscr{O}_\mathcal{C})^o;E^*)$ for the topology $\mathcal{I}_{b}$ 
have a quasi-LB dual $(\mathcal{D}'_{\gamma,\overline{\gamma}}(U,\mathcal{C};E),\mathcal{I}_{ppp})$ 
and it suffices to see the quasi-LB representation above is formed of equicontinuous sets to get the statement about class $\mathfrak{G}.$ But from the inductive limit definition of $\mathcal{I}_b$ an equicontinuous set in the dual is a set equicontinuous  in all $(\mathcal{D}'_{\Lambda,\overline{\Lambda}}(U,\mathcal{C};E),\mathcal{I}_{ppp})$, $\Lambda\supset \gamma$ open cone, but from lemma \ref{bornologyIdentities}, bounded and equicontinuous sets coincide in this space, thus they also coincide in $(\mathcal{D}'_{\gamma,\overline{\gamma}}(U,\mathcal{C};E),\mathcal{I}_{ppp}).$

The corresponding statement for $\mathcal{I}_{iii}$ was already checked before since $\mathcal{I}_{iii}=\mathcal{I}_{ppp}$ on this space.
 \end{proof}

The following result completes our understanding of functional analytic properties of generalized H\"ormander spaces of distributions. This especially completes the proof of Theorem \ref{FAMain} from the introduction. It emphasizes two topologies $\mathcal{I}_{ppp}$ which was the one built to keep nuclearity as we just seen, and $\mathcal{I}_{ibi}$ which is made to be a description of bornologifications. 

\begin{proposition}\label{FAGeneral2}Let $\gamma$ 
a cone  and  $\mathcal{C}=\mathcal{C}^{oo}$ an enlargeable polar family of closed sets in $U$ and let $\lambda=-\gamma^c.$
 The bounded sets on $\mathcal{D}'_{\gamma,\overline{\gamma}}(U,\mathcal{C};E)$ coincide for $\mathcal{I}_{ppp}$ and $\mathcal{I}_{iii}$ and this last inductive limit is regular. 
Moreover on $\mathcal{D}'_{\lambda,\lambda}(U,(\mathscr{O}_\mathcal{C})^o;E^*)$ we have $\mathcal{I}_{b}=\mathcal{I}_{ibi}=\mathcal{I}_{pbp}.$ 
  $(\mathcal{D}'_{\lambda,\overline{\lambda}}(U,(\mathscr{O}_\mathcal{C})^o;E^*),\mathcal{I}_{ibi})$ is the strong and Mackey dual of $(\mathcal{D}'_{\gamma,\overline{\gamma}}(U,\mathcal{C};E), \mathcal{I}_{iii}^{born}=\mathcal{I}_{ppp}^{born})$, the bornologification $\mathcal{I}_{ppp}^{born},$ the completion of $(\mathcal{D}'_{\lambda,{\lambda}}(U,(\mathscr{O}_\mathcal{C})^o;E^*),\mathcal{I}_{b})$ and also for any cone $\lambda\subset \Lambda\subset \overline{\lambda}$ the completion of $(\mathcal{D}'_{\lambda,{\Lambda}}(U,(\mathscr{O}_\mathcal{C})^o;E^*),\mathcal{I}_{ibi})$, which is also nuclear. 
  
  Thus, $(\mathcal{D}'_{\lambda,\overline{\lambda}}(U,(\mathscr{O}_\mathcal{C})^o;E^*),\mathcal{I}_{ibi})$ is  complete ultrabornological nuclear, 
especially a Montel space. Likewise, $(\mathcal{D}'_{\lambda,\lambda}(U,(\mathscr{O}_\mathcal{C})^o;E^*), \mathcal{I}_b)$ is the bornologification of $\mathcal{I}_{iii},$ so that the situation is summarized in the two following commuting diagrams where $i$ is the canonical injection from a space to its completion, $b$ the canonical map  with bounded inverse between a bornologification and the original space :
$$\begin{diagram}[inline]
&&(\mathcal{D}'_{\gamma,\overline{\gamma}}(U,\mathcal{C};E),\mathcal{I}_{ibi})
\\
&\ldTo^b&\uInto^i 
\\
(\mathcal{D}'_{\gamma,\overline{\gamma}}(U,\mathcal{C};E),\mathcal{I}_{ppp})& &(\mathcal{D}'_{\gamma,{\gamma}}(U,\mathcal{C};E),\mathcal{I}_{b})
\\
\uInto^i&\ldTo^b& 
\\
(\mathcal{D}'_{\gamma,{\gamma}}(U,\mathcal{C};E),\mathcal{I}_{iii})& &
\end{diagram}\quad \quad \quad \begin{diagram}[inline]
(\mathcal{D}'_{\lambda,\overline{\lambda}}(U,(\mathscr{O}_\mathcal{C})^o;E^*),\mathcal{I}_{ibi})&&
\\
\dTo^b& \luInto^i&
\\
(\mathcal{D}'_{\lambda,\overline{\lambda}}(U,(\mathscr{O}_\mathcal{C})^o;E^*),\mathcal{I}_{ppp})& &(\mathcal{D}'_{\lambda,{\lambda}}(U,(\mathscr{O}_\mathcal{C})^o;E^*),\mathcal{I}_{b})
\\
&\luInto^i&\dTo^b
\\& & 
(\mathcal{D}'_{\lambda,{\lambda}}(U,(\mathscr{O}_\mathcal{C})^o;E^*),\mathcal{I}_{iii})
\end{diagram}$$

Spaces symmetric with respect to the middle vertical line are Mackey duals of one another. All the spaces involved are nuclear locally convex spaces, and in each of them, bounded sets which are closed in the completion are metrisable compact sets and are equicontinuous sets from the stated dualities. When $\lambda,\gamma$ are $\mathbf{\Delta_2^0}$-cones (i.e. both $F_\sigma$,$G_\delta$) and if we assume
either $\mathcal{C}$ or $(\mathscr{O}_\mathcal{C})^o$  countably generated, all the space involved are moreover quasi-LB spaces of class $\mathfrak{G}$.  

\end{proposition}
Note that our statement about metrizability of compact sets didn't use the assumption to be of class $\mathfrak{G}$, even though it ultimately follows from this property in the full support closed/open cone case.
\begin{proof}
From the continuity of the map from $\mathcal{I}_{iii}$ to $\mathcal{I}_{ppp}$, a bounded set in the first space is bounded in the second. We thus start with $B$ bounded in $\mathcal{D}'_{\gamma,\overline{\gamma}}(U,\mathcal{C};E)$ for $\mathcal{I}_{ppp}$, and we can assume it absolutely convex and closed, so that by nuclearity, it is also compact in $\mathcal{D}'_{\gamma,\overline{\gamma}}(U,\mathcal{F};E)$. We have to show that $B$ is uniformly supported in some $C\in \mathcal{C}$  and bounded for $\delta=(\Gamma_n)\in \Delta(\gamma)$ in $\mathcal{D}'_{\overline{\gamma}}(U,\mathcal{F}:\delta;E)=\{u\in \mathcal{D}'(U;E)\ | \ 
\forall n ,DWF_n(u)\subset \Gamma_n \}.$ In this way, we also show $\mathcal{I}_{iii}$ is a regular inductive limit. The support condition is already known. From the projective limit representation \eqref{Alternativepip}, we know $B$ is absolutely convex compact in $\mathcal{D}'_{\lambda_1,\overline{\lambda_1}}(U,\mathcal{F};E)$ for any open $\lambda_1\supset \gamma.$ In the course of the proof of proposition \ref{bornologyIdentitiesArens}, we showed that for any $k$, it is bounded in some $(\mathcal{D}'_{k,\Pi_{k,\lambda_1}}(U,\mathcal{F};E))$, for some closed cone $\Pi_{k,\lambda_1}\subset \lambda_1$. Define the closed cone $\Pi_k=\cap_{\lambda_1\supset \gamma}\Pi_{k,\lambda_1}\subset \cap_{\lambda_1\supset \gamma}\lambda_1=\gamma.$ From the projective definition of the topology on  $(\mathcal{D}'_{k,\Pi_k}(U,\mathcal{F};E))$, defined in lemma \ref{bornologyIdentitiesEqui}, $B$ is bounded there, and we can also replace $\Pi_k$ by $\Gamma_k=\cup_{l=1}^k\Pi_l$ and since this is for all $k$, we have $\delta=(\Gamma_n)\in \Delta(\gamma)$ and we found our $\delta$ such that $B$ bounded in $\mathcal{D}'_{\overline{\gamma}}(U,\mathcal{F}:\delta;E).$


To relate the defining formula for $\mathcal{I}_b$ with our other definitions note that since in closed cone case, we already checked $\mathcal{I}_{ibi}=\mathcal{I}_{iii}^{born}$, we have by an inductive limit argument $\mathcal{I}_b$ stronger than $\mathcal{I}_{ibi}$. But as in proposition \ref{bornologyIdentitiesArens}, $\mathcal{I}_{ibi}$ stronger than $\mathcal{I}_{pbp}$ is obvious, and based this time on our computation of compact=(bounded closed)=(weakly compact) sets in the dual for $\mathcal{I}_{pbp}$ namely $\mathcal{D}'_{\gamma,\overline{\gamma}}(U,\mathcal{C};E)$, one gets from estimate \eqref{BDbounded}, that $\mathcal{I}_{pbp}$ is stronger than the strong=Mackey topology 
 coming from the dual which was computed to be $\mathcal{I}_b$ in the previous proposition \ref{FAGeneral}. 

Similarly from \eqref{BDbounded} we deduce that on $\mathcal{D}'_{\lambda,\overline{\lambda}}(U,(\mathscr{O}_\mathcal{C})^o;E^*)$, $\mathcal{I}_{ibi}$ is stronger than the strong dual topology coming from $\mathcal{I}_{iii}^{born}=\mathcal{I}_{ppp}^{born}.$ 
 Again from Arens-Mackey Thm, $\mathcal{I}_{ibi}$ is weaker that the Mackey topology, thus all are equal.

Let us now show that $(\mathcal{D}'_{\gamma,\overline{\gamma}}(U,\mathcal{C};E),\mathcal{I}_{ppp})$ is a nuclear convex bornological space. The remaining results will become easy once this proved. By definition, we have to take any Banach disk $B$ and get a Banach disk $A$ with the canonical map $E_B\to E_A$ absolutely summing (so that as usual by composition of two absolutely summing maps, one will get the same property with nuclear). 

As described above, one can take $B$ uniformly supported in some $C\in \mathcal{C}$ and absolutely convex compact, and thus say bounded for $\delta=(\Gamma_n)\in \Delta(\gamma)$ in $\mathcal{D}'_{\overline{\gamma}}(U,\mathcal{F}:\delta;E)=\{u\in \mathcal{D}'(U;E)\ | \ 
\forall n ,DWF_n(u)\subset \Gamma_n \}=\cap_{n\geq 1} \mathcal{D}_{n,\Gamma_n}'(U,\mathcal{F};E).$ As in \cite[section 10.1]{BrouderDangHelein}, take $f_{i,n,m},V_{i,n,m}$, $\text{supp}(f_{i,n,m})\subset U_i$ such that $(P_{i,2n,f_{i,n,m},V_{n,m}},)_{m\in \N,i\in I}$ is a countable family of seminorms of $\mathcal{D}_{2n,\Gamma_{2n}}'(U,\mathcal{F};E)$ defining with the seminorms of the strong topology of $\mathcal{D}'(U;E)$ its topology. We can also assume $V_{i,n,m}$ regular closed i.e. $V_{i,n,m}=\overline{Int(V_{i,n,m})}$ and with intersection to the sphere having a $C^1$ boundary and bounded stereographic projection. For our convenience we will even use a family with more seminorms than needed, in also considering $f_{i,n,m}^{\alpha}\circ\varphi_i^{-1}=x_1^{\alpha_1}...x_d^{\alpha_d}|f_{i,n,m}\circ\varphi_i^{-1}]$, $\alpha\in \N^d.$

We can assume $B=\{u\in\mathcal{D}'(U;E), \text{supp}(u)\subset C\}\cap \bigcap_{n,m\geq 1,\alpha\in \N^n,k\geq 0,i\in I}B_{i,n,m}^{\alpha,k}\cap B'$ with $B'$ a closed bounded set in $\mathcal{D}'(U;E)$, $B_{i,n,m}^{\alpha,k}=\{u, P_{i,2n,f_{i,n+k,m}^{\alpha},V_{i,n+k,m}}(u)\leq M_{n,m}^{|\alpha|,k,i}\},B_{i,n,m}^{\alpha,0}=B_{i,n,m}^{\alpha}.$ (Indeed every $B$ is included in a $B$ of that form, using $\Gamma_{2n}\subset \Gamma_{2n+2k}$ so that the seminorms written are seminorms for $\mathcal{D}_{2n,\Gamma_{2n}}'(U,\mathcal{F};E)$, and a $B$ of that form is indeed absolutely convex bounded closed, thus compact in $(\mathcal{D}'_{\gamma,\overline{\gamma}}(U,\mathcal{C};E),\mathcal{I}_{ppp})$ and thus completant there). We will build $A$ of the same form and even without index $k$ since it is obviously redundant. The reader should see this redundancy as a way of using $\Gamma_{k}\subset\Gamma_{k+1}$ without worrying about the compatibility of our decompositions of cones using $f_{i,N,m},V_{i,N,m}$ at various indices $N$. We also assume $|I|=\N$ for convenience.

Since $\mathcal{D}'(U;E)$ is a nuclear convex bornological space as any equicontinuous dual of a nuclear locally convex space (and since $B'$ is compact there thus completant) there exists a Banach disk $A'$ such that $E_{B'}\to E_{A'}$ is nuclear thus absolutely summing. Thus by Pietsch's characterization (see \cite[Th 2.3.3]{Pietsch}, cf. also \cite[Th1 p 97]{HogbeNlendMoscatelli}) there is a positive Radon-measure on the closed unit ball of $E_{B'}'$ such that for all $x\in E_{B'}$, we can bound the jauge norm $p_{A'}$ of $E_{A'}$  by: $$p_{A'}(x)\leq \int |\langle x,x'\rangle|d\mu'(x').$$ 

To get the same situation for $B_{i,n,m}^{\alpha}$, one uses the standard bounds to prove $\mathscr{S}$ is nuclear (see e.g. \cite[Th 6.2.5 p 101]{Pietsch}).
Note first that $p_{B_{i,n,m}^{\alpha}}=\frac{1}{M_{n,m}^{|\alpha|,0,i}}P_{i,2n,f_{i,n,m}^{\alpha},V_{n,m}}$
and $(1+|\xi|^2)^n\leq (1+|\xi|)^{2n}\leq 2^n(1+|\xi|^2)^n$ in order to use a slightly smoother expression.
We use the elementary application of the fundamental theorem of calculus :
\begin{align*}&|(1+|\xi|^2)^n\mathcal{F}((f_{j,N,m}^{\alpha}u)\circ\varphi_j^{-1})(\xi)|=\left|\mathcal{F}((f_{j,N,m}^{\alpha}u)\circ\varphi_j^{-1}))(0)+\int_0^{|\xi|}ds\right.\\&\left.\ 2sn(1+s^2)^{n-1}\mathcal{F}((f_{j,N,m}^{\alpha}u)\circ\varphi_j^{-1}))(s\xi/|\xi|)+(1+s^2)^n\sum_{i=1}^d(-2i\pi)\xi_i/|\xi|\mathcal{F}(x_i(f_{j,N,m}^{\alpha}u)\circ\varphi_j^{-1}))(s\xi/|\xi|)\right|
\\&\leq |\mathcal{F}((f_{j,N,m}^{\alpha}u)\circ\varphi_j^{-1}))(0)|\\&+\int_0^{\infty}ds(1+s^2)^n\left(n|\mathcal{F}((f_{j,N,m}^{\alpha}u)\circ\varphi_j^{-1}))(s\xi/|\xi|)|+\sum_{i=1}^d 2\pi|\mathcal{F}(x_i(f_{j,N,m}^{\alpha}u)\circ\varphi_j^{-1}))(s\xi/|\xi|)|\right)\end{align*}

Note that from the assumption on $V_{n,m}$ one can use a Sobolev embedding theorem (in the form of \cite[Corol iX.15 p 169]{Brezis} $W^{d,1}(\Omega)\subset C^0(\overline{\Omega})$, $\Omega$ bounded open with $C^1$ boundary in $\R^{d-1}$) to get :
$$\sup_{\xi\in V_{j,N,m},|\xi|=1}|\mathcal{F}((f_{j,N,m}^{\alpha}u)\circ\varphi_j^{-1})(s\xi)|\leq  c_{j,N,m,d}\sum_{|\beta|\leq d}\int_{V_{j,N,m}\cap S^{d-1}} dLeb_{S^{d-1}}(\eta)s^{|\beta|}|\mathcal{F}((f_{j,N,m}^{\alpha+\beta}u)\circ\varphi_j^{-1})(s\eta)|.$$


Thus, summing up, one obtains :
\begin{align*}\sup_{\xi\in V_{j,N,m}}&|(1+|\xi|^2)^n\mathcal{F}((f_{j,N,m}^{\alpha}u)\circ\varphi_j^{-1})(\xi)|\\&\leq|\mathcal{F}((f_{j,N,m}^{\alpha}u)\circ\varphi_j^{-1})(0)|+c_{j,n,m,d}(n+2\pi)\int_0^{\infty}\frac{ds}{1+s^2}\int_{V_{j,N,m}\cap S^{d-1}} dLeb_{S^{d-1}}(\eta)
\\& \left(\sum_{|\beta|\leq d}|(1+s^2)^{n+|\beta|+1}|\mathcal{F}((f_{j,N,m}^{\alpha+\beta}u)\circ\varphi_j^{-1})(s\eta)|+\sum_{|\beta|\leq d+1}(1+s^2)^{n+|\beta|} |\mathcal{F}((f_{j,N,m}^{\alpha+\beta}u)\circ\varphi_j^{-1})(s\eta)|\right)\end{align*}

Let $N=n+d+1$, $B\subset B^{\prime\ \alpha}_{j,n,m}=\cap_{|\beta|\leq d+1}B^{\alpha+\beta,N-(n+|\beta|+1_{\{|\beta|\leq d\}})}_{j,n+|\beta|+1_{\{|\beta|\leq d\}},m}$ (note that only $f_{j,n+d+1,m},V_{j,n+d+1,m}$ are involved in those spaces) and $$A_{j,n,m}^{\alpha}=\{u, P_{j,2n,f_{j,n+d+1,m},V_{j,n+d+1,m}}(u)\leq M_{n,m}^{|\alpha|,d+1}\}\supset B.$$

 Note that if $s>0,\eta \in V_{N,m}\cap S^{d-1},k\leq d+1$, then  $$\langle u,\epsilon_{s,j,\eta}^{k,\alpha+\beta}\rangle=\frac{(1+s^2)^{n+k+1}}{M^{|\alpha+\beta[,N-(n+k+1_{\{k\leq d\}})}_{n+k+1_{\{k\leq d\}},m}} \mathcal{F}(f_{j,N,m}^{\alpha+\beta}u)\circ\varphi_j^{-1})(s\eta)$$ defines an element of the unit ball of $(E_B)'$, and likewise $\langle u,\delta_{j,n,m}^\alpha\rangle=|\mathcal{F}(f_{j,N,m}^{\alpha}u)\circ\varphi_j^{-1})(0)|$ so that  we can  now define a positive Radon measure on this unit ball by defining its value at a continuous function $\varphi$ on this unit ball :

\begin{align*}\int_{(E_B)'_1}&d\nu_{j,n,m}^\alpha(a)\varphi(a):=
\varphi(\delta_{j,n,m}^\alpha)+c_{j,n,m,d}(n+2\pi)\int_0^{\infty}\frac{ds}{1+s^2}\int_{V_{j,N,m}\cap S^{d-1}} dLeb_{S^{d-1}}(\eta)
\\& \left(\sum_{|\beta|\leq d}\varphi(\epsilon_{s,j,\eta}^{|\beta|,\alpha+\beta})M^{|\alpha+\beta|,N-(n+|\beta|+1_{\{|\beta|\leq d\}})}_{n+|\beta|+1_{\{|\beta|\leq d\}},m}+\sum_{|\beta|\leq d+1} \varphi(\epsilon_{s,j,\eta}^{|\beta|,\alpha+\beta})M^{|\alpha+\beta|,N-(n+|\beta|+1_{\{|\beta|\leq d\}})}_{n+|\beta|+1_{\{|\beta|\leq d\}},m}\right)\end{align*}

and then the probability measure $\mu_{j,n,m}^\alpha=\nu_{j,n,m}^\alpha/c_{j,n,m}^\alpha, c_{j,n,m}^\alpha=\nu_{j,n,m}^\alpha((E_B)'_1)<\infty.$

We can thus rewrite our estimate :
$$p_{A_{j,n,m}^\alpha}(u)\leq C_{j,n,m}^\alpha\int d\mu_{j,n,m}^\alpha(a) |\langle u,a\rangle|,  C_{j,n,m}^\alpha=\max(1,c_{j,n,m}^\alpha \frac{2^n}{M_{j,n,m}^{|\alpha|,d+1}}). $$

Finally, define $$A=\{u\in\mathcal{D}'(U), \text{supp}(u)\subset C\}\cap \bigcap_{n,m\geq 1,\alpha\in \N^d,j\in\N}A_{j,n,m}^{\alpha}2^{j+n+m+[\alpha|}C_{j,n,m}^{\alpha}\cap A'.$$ By construction we have $B\subset A$ and $A$ is bounded in  $\mathcal{D}'_{\overline{\gamma}}(U,\mathcal{F}:\delta'),$ with  $\delta'=(\Gamma'_n)\in \Delta(\gamma)$, $\Gamma'_{2k}=\Gamma'_{2k-1}=\Gamma_{2(k+d+1)}, $ thus a Banach disk in $(\mathcal{D}'_{\gamma,\overline{\gamma}}(U,\mathcal{C}),\mathcal{I}_{ppp})$ as above for $B$.

Note that $p_A(x)=\max(p_{A'}(x),\sup_{j,n,m,\alpha}\frac{ p_{A_{n,m}^\alpha}(x)}{2^{j+n+m+|\alpha|}C_{n,m}^\alpha})$ on $\{u\in\mathcal{D}'(U), \text{supp}(u)\subset C\}\supset E_{A}$ 

But for $x\in E_B, B
\subset A_{n,m}^{\alpha}$ $p_{A_{n,m}^{\alpha}}(x)\leq p_B(x)$ so that $$p_A(x)\leq p_{A'}(x)+\sum_{n,m=1}^\infty\sum_{\alpha\in \N^d} \frac{p_{A_{n,m}^{\alpha}}(x)}{2^{n+m+|\alpha|}C_{n,m}^\alpha}\leq p_B(x)+\sum_{n,m=1}^\infty\sum_{\alpha\in \N^d} \frac{p_{B}(x)}{2^{n+m+|\alpha|}}<\infty$$
and thus, seeing $\mu'$ as a measure on $(E_{B'})'\subset (E_{B})'$ since $E_{B}\subset E_{B'}$ and thus even $\mu'$ supported on the unit ball of $(E_{B})'$, and likewise $\mu_{j,n,m}^{\alpha}$ probability measures on $ (E_{B})'$, we have from our inequalities : $$p_A(x)\leq \int |\langle x,x'\rangle|d\mu(x')\leq p_B(x)\mu'((E_B')_1)+\sum_{n,m=1}^\infty\sum_{\alpha\in \N^d} \frac{p_{B}(x)}{2^{j+n+m+|\alpha|}}<\infty.$$

with $\mu=\mu'+\sum_{n,m=1,\alpha\in \N^d}^\infty \frac{1}{2^{j+n+m+|\alpha|}}\mu_{j,n,m}^\alpha$ which is a positive Radon measure (of finite mass)
 we conclude by Pietsch's characterization again that $E_{B}\to E_{A}$ is absolutely summing.

This concludes the proof of $(\mathcal{D}'_{\gamma,\overline{\gamma}}(U,\mathcal{C};E),\mathcal{I}_{ppp})$  a nuclear convex bornological space. Thus $(\mathcal{D}'_{\lambda,{\lambda}}(U,(\mathscr{O}_\mathcal{C})^o;E^*),\mathcal{I}_{b})$ having this space as equicontinuous dual (since bounded and equicontinuous sets of its dual coincide as for any bornological space), is a nuclear locally convex space, and so is its completion, concluding to all the stated nuclearities.

From \cite[p.~14 Corollary 2]{HogbeNlendMoscatelli}, since $(\mathcal{D}'_{\lambda,{\lambda}}(U,(\mathscr{O}_\mathcal{C})^o;E^*),\mathcal{I}_{b})$ is nuclear, its completion is the bornological dual of its equicontinuous dual which is here $(\mathcal{D}'_{\gamma,\overline{\gamma}}(U,\mathcal{C};E), \mathcal{I}_{ppp})$ with (strongly) bounded sets equal to equicontinuous sets from $\mathcal{I}_{b}$ bornological again \cite[p 75]{HogbeNlend2}. One can also compute the equicontinuous explicity from the inductive limit defining $\mathcal{I}_b$ and as usual \cite{Kothe} and proposition \ref{FAClosedOpen2} give exactly the bounded sets for $\mathcal{I}_{ppp}$ from \eqref{Alternativeppp}. The identification of the completion thus follows from our computation of the bornological dual. From the inclusions from lemma \ref{contuinuousdenseInj} : $$(\mathcal{D}'_{\lambda,{\lambda}}(U,(\mathscr{O}_\mathcal{C})^o;E^*),\mathcal{I}_{ibi}=\mathcal{I}_{b})\subset (\mathcal{D}'_{\lambda,{\Lambda}}(U,(\mathscr{O}_\mathcal{C})^o;E^*),\mathcal{I}_{ibi})\subset (\mathcal{D}'_{\lambda,\overline{\lambda}}(U,(\mathscr{O}_\mathcal{C})^o;E^*),\mathcal{I}_{ibi})$$
and since the last space is the completion of the first, we deduce that so it is for the middle space, thus this one is also nuclear as a consequence. From the identification of closed bounded sets of $(\mathcal{D}'_{\gamma,\overline{\gamma}}(U,\mathcal{C};E), \mathcal{I}_{ppp})$ as closed equicontinuous sets from its dual, which is nuclear, we see they are metrizable compact \cite[Prop 50.2 p 519]{Treves} for the strong topology of the dual, which we will check below to be the same as the strong topology of the completion of the dual (see proposition \ref{QuasiCompletion}) thus it is metrizable compact for $\mathcal{I}_{pbp}$ thus a fortiori for $\mathcal{I}_{ppp}$ (since the bijective continuous map between those compacts is necessarily an homeomorphism). Of course we can say the same for closed bounded sets for $\mathcal{I}_{pbp}$ since by nuclearity they are compact, thus closed and bounded for $\mathcal{I}_{ppp}$. Then for the two other non-complete spaces, closed bounded sets are also so in their completion, thus metrizable compact as before as stated. For $\mathcal{D}'_{\lambda,\overline{\lambda}}(U,(\mathscr{O}_\mathcal{C})^o;E^*)$, with $\mathcal{I}_{ibi}$  which have its strong topology from its dual which is bornological, it is obvious that those sets coincide, bounded closed sets which are strongly bounded coincide with equicontinuous closed sets (as in lemma \ref{bornologyIdentities}), and for $\mathcal{I}_{ppp}$ we just checked this. For $\mathcal{D}'_{\lambda,{\lambda}}(U,(\mathscr{O}_\mathcal{C})^o;E^*)$, with $\mathcal{I}_{iii}$, the same has been checked in proposition \ref{bornologyIdentitiesArens} and it will only remain to do the case $\mathcal{I}_{b}$ at the end of the proof.

Since $(\mathcal{D}'_{\lambda,\overline{\lambda}}(U,(\mathscr{O}_\mathcal{C})^o;E^*),\mathcal{I}_{ibi})$ has its Mackey topology and same continuous and bornological duals by theorem \ref{duals}, it is bornological, and to identify it as the bornologification of $\mathcal{I}_{ppp}$ it suffices to see they have the same bounded sets described before, which is obvious since the supplementary polynomial boundedness is obtained automatically for bounded sets. The nuclearity of $\mathcal{I}_{ibi}$ has thus been checked in identifying it to the completion and from completeness, we deduce the dual of the same form is ultrabornological, thus barrelled, and it only remains to check closed bounded sets are compact to obtain a Montel space. We could do this directly, but this is obvious since it is complete nuclear.

Moreover a bounded set in $(\mathcal{D}'_{\lambda,{\lambda}}(U,(\mathscr{O}_\mathcal{C})^o;E^*),\mathcal{I}_{iii})$ is bounded in the completion, thus in the bornologification of this completion, thus in the restriction of this one to our space, i.e. in $\mathcal{I}_{ibi}$, so that they have the same bounded sets, and since we know one to be bornological, it is the bornologification of the other.

In the $\mathbf{\Delta_2^0}$-cones case with either $\mathcal{C}$ or $(\mathscr{O}_\mathcal{C})^o$  countably generated, the only space not yet known to be a quasi-LB space of class $\mathfrak{G}$, is the bornologification of the completion, but since quasi-LB spaces are stable by bornologification and class $\mathfrak{G}$ by completion, this follows from the diagram and the previous results. The only Mackey dual not yet known is the one for $\mathcal{I}_b$, but 
since any absolutely convex (weakly) compact is also  compact for $\mathcal{I}_{iii}$ computed in the proof of proposition \ref{bornologyIdentitiesArens} and thus has to be in some term of the defining inductive limit and then bounded (weakly compact) there using the bornologification involved is quasi-LB to apply De Wilde's closed graph theorem, so that (weakly) compact absolutely convex sets coincide with closed equicontinuous sets. In this way we also check that bounded closed sets which are included in absolutely convex (bounded closed thus) compact sets as seen before coincide with closed equicontinuous sets also in this case as stated.
\end{proof}

For completeness we also study some more properties of mixed spaces with control of both $DWF$ and $WF.$ We will mostly use $\mathcal{D}'_{\gamma,\Lambda}(U,\mathcal{C};E)=\mathcal{D}'_{\gamma,\overline{\gamma}}(U,\mathcal{C};E)\cap \mathcal{D}'_{\Lambda,\Lambda}(U,\mathcal{C};E).$

\begin{proposition}\label{FAGeneral3}Let $\gamma\subset \Lambda\subset \overline{\gamma}$ 
a cone  and  $\mathcal{C}=\mathcal{C}^{oo}$ an enlargeable polar family of closed sets in $U$ and let $\lambda=-\gamma^c.$
 All the spaces below symmetric with respect to the middle line are Mackey duals of one another ( undefined topologies $\mathcal{I}_1,\mathcal{I}_2$ are defined this way). They are all nuclear and completion maps $i$, bornologification maps $b$ are indicated as in the last proposition :
$$\begin{diagram}[inline]
&&(\mathcal{D}'_{\gamma,\overline{\gamma}}(U,\mathcal{C};E),\mathcal{I}_{ibi})
\\
&\ldTo^b&\uInto^i 
\\
(\mathcal{D}'_{\gamma,\overline{\gamma}}(U,\mathcal{C};E),\mathcal{I}_{ppp})& &(\mathcal{D}'_{\gamma,{\Lambda}}(U,\mathcal{C};E),\mathcal{I}_{ibi})
\\
\uInto^i&\ldTo^b& 
\\
(\mathcal{D}'_{\gamma,{\Lambda}}(U,\mathcal{C};E),\mathcal{I}_{ppp})& &
\end{diagram}\quad \quad \quad \begin{diagram}[inline]
(\mathcal{D}'_{\lambda,\overline{\lambda}}(U,(\mathscr{O}_\mathcal{C})^o;E^*),\mathcal{I}_{ibi})&&
\\
\dTo^b& \luInto^i&
\\
(\mathcal{D}'_{\lambda,\overline{\lambda}}(U,(\mathscr{O}_\mathcal{C})^o;E^*),\mathcal{I}_{2})& &(\mathcal{D}'_{\lambda,{\lambda}}(U,(\mathscr{O}_\mathcal{C})^o;E^*),\mathcal{I}_{b})
\\
&\luInto^i&\dTo^b
\\& & 
(\mathcal{D}'_{\lambda,{\lambda}}(U,(\mathscr{O}_\mathcal{C})^o;E^*),\mathcal{I}_{1})
\end{diagram}$$

\end{proposition}

\begin{proof}
We already know the completion properties from the last proposition on the left diagram. On $\mathcal{D}'_{\gamma,{\Lambda}}(U,\mathcal{C};E),$ bounded sets for $\mathcal{I}_{ibi}$ are the same as for  $\mathcal{I}_{ppp},$ (using they are the same as in the completion, which is the bornologification of the completion of $\mathcal{I}_{ppp}$). To show the last bornologification missing on this side, it remains to show $(\mathcal{D}'_{\gamma,{\Lambda}}(U,\mathcal{C};E), \mathcal{I}_{ibi})$ is bornological. For this we start by showing it is a nuclear convex bornological space (either for its von Neumann bornology or its equicontinuous bornology coming from its dual with Mackey topology, namely the bornology generated by absolutely convex weakly compact sets). Indeed, $\Delta:(\mathcal{D}'_{\gamma,{\Lambda}}(U,\mathcal{C};E), \mathcal{I}_{ibi})\to (\mathcal{D}'_{\gamma,\overline{{\gamma}}}(U,\mathcal{C};E), \mathcal{I}_{ppp})\times (\mathcal{D}'_{\Lambda,{\Lambda}}(U,\mathcal{C};E), \mathcal{I}_{ppp})$ sending $u$ to $(u,u)$ is easily seen to be a bornological isomorphism \footnote{we replaced  $\mathcal{I}_{ibi}$ by $\mathcal{I}_{ppp}$ on the right hand side since they have same bounded sets; for the equicontinuous case, we described a $B$ absolutely convex weakly compact sets jointly in both targets to be bounded for the bornological inductive limit associated to $\mathcal{I}_{iii}$ on $\mathcal{D}'_{\gamma,{\Lambda}}(U,\mathcal{C};E)$ thus bounded in some $(\mathcal{D}'_{\gamma\cap \Pi,\overline{\gamma\cap \Pi}}(U,\mathcal{C};E),\mathcal{I}_{iii})$ $\Pi\subset \Lambda$ closed cone which is complete nuclear, thus $B$ compact there thus also in $(\mathcal{D}'_{\gamma,{\Lambda}}(U,\mathcal{C};E), \mathcal{I}_{iii}).$ From weak compactness implying closed and bounded in $(\mathcal{D}'_{\gamma,\overline{{\gamma}}}(U,\mathcal{C};E), \mathcal{I}_{ppp})$ we checked in the proof of proposition \ref{FAGeneral2} it has to be (metrizable) compact for  $\mathcal{I}_{pbp}$, thus take a net in $B$ it has a subnet converging for $\mathcal{I}_{pbp}$ which has a subsubnet converging for ($\mathcal{D}'_{\gamma,{\Lambda}}(U,\mathcal{C};E), \mathcal{I}_{iii})$ so that the limit of the first subnet was in our original space and for the topology $\mathcal{I}_{ibi}$ which is the one induced from the completion.}. Moreover, we know that the right hand side is a nuclear convex bornological space. Indeed by finite product, we have to see this for both spaces, and from the proof of proposition \ref{FAGeneral2} we know this for the first space (for von Neumann bornology, the equicontinuous bornology follows from the nuclearity of its Mackey dual). For the second, take $B$ completant Banach disk in $(\mathcal{D}'_{\Lambda,{\Lambda}}(U,\mathcal{C};E), \mathcal{I}_{ppp})$, the proof of proposition \ref{bornologyIdentitiesArens} applies verbatim (we didn't really use $K$ compact, only it is a Banach disk so it is $\ell^1$ disked) to prove that it is equicontinuous. But the equicontinuous bornology is known to be nuclear (this is by definition equivalent to the dual being nuclear as a locally convex space). 
 Thus since nuclear c.b.s form a bornological ultra-variety \cite[4:1 Th 1 p 200]{HogbeNlendMoscatelli} it suffices to check the image of $\Delta$ is bornologically closed. This follows from the fact that both our spaces are (boundedly) included in $\mathcal{D}'(U,E)$ which is separated (bornologically). From this nuclearity of its equicontinuous bornology, we deduce the Mackey dual $(\mathcal{D}'_{\lambda,\overline{\lambda}}(U,(\mathscr{O}_\mathcal{C})^o;E^*),\mathcal{I}_{2})$ is nuclear. Moreover, let us explain why it is the bornological dual  of the equicontinuous bornology of $\mathcal{D}'_{\gamma,{\Lambda}}(U,\mathcal{C};E).$ Indeed, a bounded linear functional is also bounded for the equicontinuous bornology (coming from its Mackey dual) of $(\mathcal{D}'_{\gamma,{\gamma}}(U,\mathcal{C};E),\mathcal{I}_{ibi})$
. This Mackey dual is known to be $(\mathcal{D}'_{\lambda,\overline{\lambda}}(U,(\mathscr{O}_\mathcal{C})^o;E^*),\mathcal{I}_{ppp})$ so that the bornological dual of the equicontinuous dual is known to be the completion (since it is nuclear \cite[corol 3 p 140]{HogbeNlendMoscatelli}), thus the bounded linear functional is in  $\mathcal{D}'_{\lambda,\overline{\lambda}}(U,(\mathscr{O}_\mathcal{C})^o;E^*)$. Conversely, any element of this space is known to define a bounded linear functional on $(\mathcal{D}'_{\gamma,\overline{\gamma}}(U,\mathcal{C};E),\mathcal{I}_{ibi})$ thus on our original space by the continuous injection. Thus we identify the bornological dual, and as any bornological dual of separated convex bornological space, $\mathcal{I}_2$ is thus complete. Since it is also nuclear, by Schwartz Theorem again, its strong dual is ultrabornological \cite[7:2.4 corol p 54]{HogbeNlend2}. But a complete nuclear space is semi-Montel thus its strong and Mackey dual coincide, and it suffices to prove this strong dual is $(\mathcal{D}'_{\gamma,{\Lambda}}(U,\mathcal{C};E),\mathcal{I}_{ibi}),$ to finally check this space is bornological (finishing the proof on the left side of the diagram). But on $(\mathcal{D}'_{\gamma,\overline{\gamma}}(U,\mathcal{C};E)$ $\mathcal{I}_{ibi})$ and $\mathcal{I}_2$ have the same bounded sets from the continous maps  $\mathcal{I}_{ibi}\to \mathcal{I}_{2}\to \mathcal{I}_{ppp}$ induced by duality from the known ones with $\mathcal{I}_{ibi}$, and since the first and last spaces have the same bounded sets. Note that since $\mathcal{I}_{ibi}$ is bornological, it is also the bornologification of  $\mathcal{I}_{2}.$ Thus the strong topology is the same as the topology induced by the dual of $(\mathcal{D}'_{\gamma,\overline{\gamma}}(U,\mathcal{C};E),\mathcal{I}_{ibi})$ namely by the completion property, this is indeed $(\mathcal{D}'_{\gamma,{\Lambda}}(U,\mathcal{C};E),\mathcal{I}_{ibi}),$ which is thus bornological.

It remains to study $(\mathcal{D}'_{\lambda,{\lambda}}(U,(\mathscr{O}_\mathcal{C})^o;E^*),\mathcal{I}_{1})$.

The topology induced by $\mathcal{I}_{2}$ having this topology as completion, has the same dual, and is thus weaker than the Mackey topology from the duality with $(\mathcal{D}'_{\gamma,{\Lambda}}(U,\mathcal{C};E),\mathcal{I}_{ppp}),$ namely what we called $\mathcal{I}_{1}$. Conversely, if a net converges for $\mathcal{I}_{2}$, it converges uniformly on absolutely convex weakly compact sets $(\mathcal{D}'_{\gamma,{\Lambda}}(U,\mathcal{C};E),\mathcal{I}_{ibi}),$ but an analogous set $B$ for $\mathcal{I}_{ppp}$ is of the same type in its completion $(\mathcal{D}'_{\gamma,\overline{\gamma}}(U,\mathcal{C};E),\mathcal{I}_{ppp}),$
thus bounded and (weakly closed thus) closed, and also so for $\mathcal{I}_{ibi}$ (inverse image by $b$ for closedness, and same bounded sets), thus it is compact since  we are in a Montel space. But this is thus an absolutely convex compact set included in $\mathcal{D}'_{\gamma,{\Lambda}}(U,\mathcal{C};E)$ for which $\mathcal{I}_{ibi}$ is induced from the completion above, thus our $B$ is compact thus weakly compact there and this is exactly the kind of spaces on which we know uniform convergence with $\mathcal{I}_2$ . Thus $\mathcal{I}_1$ is indeed the topology induced by $\mathcal{I}_2$, explaining it nuclearity, and explaining  its bornologification is $\mathcal{I}_b$ once again by diagram chasing (to see they have same bounded sets since this last space is known to be bornological).

It only remains to compute the Mackey dual for $\mathcal{I}_1$, but a net converging for $(\mathcal{D}'_{\gamma,\Lambda}(U,\mathcal{C};E),\mathcal{I}_{ppp}),$ converges in the completion $(\mathcal{D}'_{\gamma,\overline{\gamma}}(U,\mathcal{C};E),\mathcal{I}_{ppp}),$ thus, since it has its Mackey topology, uniformly on absolutely convex weakly compacts of $(\mathcal{D}'_{\lambda,{\lambda}}(U,(\mathscr{O}_\mathcal{C})^o;E^*),\mathcal{I}_{b})$. But an absolutely convex weakly compact for $\mathcal{I}_{1}$ is closed bounded, thus also for $\mathcal{I}_{b}$, thus also in its completion, thus compact, thus compact for $\mathcal{I}_{b}$ and the net converges uniformly on it, thus  $\mathcal{I}_{ppp}$ is stronger than the Mackey topology, the converse being obvious by Arens-Mackey Thm.
\end{proof}

\section{Pull-back and manifold case}
The following result is now an easy consequence of \cite[Prop 6.1]{BrouderDangHelein} and we can even prove it in only using \cite{Duistermaat}.
\begin{proposition}\label{pullback}Let $\gamma\subset \Lambda\subset \overline{\gamma}$ be cones on $U_2, E\to U_2$ a vector bundle and $f:U_1\to U_2$ a smooth map (between open sets in manifolds of the type of $U$ before fixed at the beginning of part 1). Define the cone  $df^*\gamma=\{(x,df^{*}(x)(\xi)):(f(x),\xi)\in\gamma\}$ and  for an enlargeable polar family of closed sets $\mathcal{C}$ define $f^{-1}(\mathcal{C})=\{(f^{-1}(C), C\in \mathcal{C},\},$ and its polar enlargeable variant $f_e^{-1}(\mathcal{C})=\{(f^{-1}(C))_{(1-1/n)\epsilon(C)}, C\in \mathcal{C},\epsilon_i(C)>0\}^{oo}$ (depending on any function $\epsilon:\mathcal{C}\to ]0,1[^I$)

Assume $df^*\gamma\subset \dot{T}^*U_1$, and $ df^*\Lambda\subset \dot{T}^*U_1$. Then we have  continuous maps for $\mathcal{I}$ either $\mathcal{I}_{ppp}$ or $\mathcal{I}_{ibi}$ : $$f^*:(\mathcal{D}'_{\gamma,{\Lambda}}(U_2,\mathcal{C};E),\mathcal{I})
\to (\mathcal{D}'_{f^*\gamma,f^*{\Lambda}}(U_1,f_e^{-1}(\mathcal{C});f^*E),\mathcal{I}),$$
$$f^*:(\mathcal{D}'_{\gamma,\overline{\gamma}}(U_2,\mathcal{C};E),\mathcal{I})
\to (\mathcal{D}'_{df^*\gamma,\overline{df^*{\gamma}}}(U_1,f_e^{-1}(\mathcal{C});f^*E),\mathcal{I}).$$
Especially, $DWF(f^*u)\subset df^*(DWF(u)).$
\end{proposition}
Note that from this result, one deduces as in lemma \ref{contuinuousdenseInj} that all our topologies don't depend of the various choices of charts. We will also need this result in our study of composition of operators in the second paper of this series.

\begin{proof}
We leave the details of the general bundle case to the reader.

From \cite[Prop 6.1]{BrouderDangHelein} and taking inductive limits over closed cone $\Pi\subset\gamma$ so that $df^*\Pi\subset df^*\gamma$ is a closed cone
, we have a continuous map $$f^*:(\mathcal{D}'_{\gamma,{\gamma}}(U_2,\mathcal{F}),\mathcal{I}_{iii})\to (\mathcal{D}'_{df^*\gamma,df^*{\gamma}}(U_1,\mathcal{F})),\mathcal{I}_{iii}).$$

If we don't want to use \cite{BrouderDangHelein} (which is written for $U_i$ open in $\R^d$), we can reason as follows in the closed cone case.
From \cite[Prop 1.3.3 p 19]{Duistermaat}, we know the map  with $\gamma$ closed is defined (and sequentially continuous since every continuity statement in Duistermaat is implicitly only sequential continuity as in \cite{Hormander}). From sequential continuity, one deduces as in \cite[lemma 21]{BrouderDabrowski} that it is Mackey sequentially continuous thus bounded, thus gives a map $f^*:(\mathcal{D}'_{\gamma,{\gamma}}(U_2,\mathcal{F}),\mathcal{I}_{ibi})\to (\mathcal{D}'_{df^*\gamma,df^*{\gamma}}(U_1,\mathcal{F})),\mathcal{I}_{ibi})$ and we have a  Mackey-continuous adjoint map : $$(f^*)^*:(\mathcal{D}'_{-(df^*\gamma)^c,\overline{-(df^*{\gamma})^c}}(U_1,\mathcal{K})),\mathcal{I}_{ibi})\to(\mathcal{D}'_{-\gamma^c,\overline{-\gamma^c}}(U_2,\mathcal{K}),\mathcal{I}_{ibi}) .$$

By restriction, to get a continuous map $$(f^*)^*:(\mathcal{D}'_{-(df^*\gamma)^c,{-(df^*{\gamma})^c}}(U_1,\mathcal{K})),\mathcal{I}_{ibi}=\mathcal{I}_{iii})\to(\mathcal{D}'_{-\gamma^c,{-\gamma^c}}(U_2,\mathcal{K}),\mathcal{I}_{iii}) ,$$
it suffices to check the space of value since we know our previous target space was the completion and thus induces the right topology. But this is the statement of  \cite[Prop 1.3.4 p 20]{Duistermaat} that computes the wave front set of the adjoint map called push-forward.

To get the general case with support in $\mathcal{C}$, note that if $\text{supp}(u)\subset C$, $\text{supp}(f^*u)\subset f^{-1}(C_\epsilon)$ for some $\epsilon$ (since $u$ is approximated by smooth maps $\varphi$ with support in $C_\epsilon$ so that $f^*\phi$ has support in $f^{-1}(C_\epsilon)$). Thus again, the map restricts to $f^*:(\mathcal{D}'_{\gamma,{\gamma}}(U_2,\{F\in\mathcal{F}, F\subset C\}),\mathcal{I}_{iii})\to (\mathcal{D}'_{df^*\gamma,df^*{\gamma}}(U_1,\{F\in\mathcal{F}, F\subset f^{-1}(C_\epsilon)\}),\mathcal{I}_{iii}).$ This gives the continuous map at the inductive limit level. Now composing with the completion map, and from equality of topologies both from corollary \ref{FAGeneral}, we get a continuous map $$f^*:(\mathcal{D}'_{\gamma,{\gamma}}(U_2,\mathcal{C}),\mathcal{I}_{ppp})\to  (\mathcal{D}'_{df^*\gamma,\overline{df^*{\gamma}}}(U_1,f_e^{-1}(\mathcal{C})),\mathcal{I}_{ppp}),$$ which thus extends to the completion. By considering bornologifications we also get the result for $\mathcal{I}_{ibi}.$  Then general cases $\Lambda$ come from the fact the topology at the target is then induced by the topology of the completion and we thus only have to check the space of value is the stated one. For, we only have to test wave front sets (since $DWF$ does not change while going to the completion) and thus this follows from \cite[Prop 6.1]{BrouderDangHelein} again. For the last statement we only have to take $\gamma=DWF(u)$ in order to deduce get $DWF(f^*u)\subset df^*\gamma.$
\end{proof}

\end{document}